\documentclass[a4paper,12pt]{amsart}
\usepackage[utf8]{inputenc}
\usepackage[T1]{fontenc}
\usepackage[UKenglish]{babel}
\usepackage[margin=18mm]{geometry}
\usepackage{amsmath, amssymb}
\usepackage{color}%cyan,red;magenta,green;yellow,blue

\usepackage{graphicx}

\usepackage{amsmath,amssymb,amsfonts,amsthm}
\usepackage{mathrsfs,eucal,dsfont}
\usepackage{verbatim,enumitem}

\usepackage{hyperref,url}

\newcommand{\R}{\mathds R}

\newcommand{\I}{\mathds 1}

\def\d{{\rm d}}
\def\<{\langle}
\def\>{\rangle}

\def\R{\mathbb R}    
 \def\kk{\kappa} 
  \def\vv{\varepsilon} 
\def\<{\langle} \def\>{\rangle}  
  \def\nn{\nabla}  
\def\d{\text{\rm{d}}}   
  \def\si{\sigma} 
 \def\beq{\begin{align}}  
 
\def\e{\text{\rm{e}}}    
  
 \def\P{\mathbb P}

\def\i{{\rm in}}\def\E{\mathbb E} 
  
  \def\i{{\rm i}} 
\def\to{\rightarrow}
\def\8{\infty}\def\3{\triangle}
\def\1{\lesssim}

\renewcommand{\bar}{\overline}
\renewcommand{\hat}{\widehat}
\renewcommand{\tilde}{\widetilde}

\newtheorem{theorem}{Theorem}[section]
\newtheorem{lemma}[theorem]{Lemma}
\newtheorem{proposition}[theorem]{Proposition}

\theoremstyle{definition}

\newtheorem{remark}[theorem]{Remark}

\numberwithin{equation}{section}
\begin{document}
\allowdisplaybreaks

\title[Quantitative Asymptotics for Time-Inhomogeneous L\'evy-Driven SDEs] {Quantitative Asymptotics for Time-Inhomogeneous L\'evy-Driven SDEs with Asymptotically Vanishing Drifts}

\author{
Jianhai Bao\qquad
Jian Wang}
\date{}
\thanks{\emph{J.\ Bao:} Center for Applied Mathematics, Tianjin University, 300072  Tianjin, P.R. China. \url{jianhaibao@tju.edu.cn}}

\thanks{\emph{J.\ Wang:}
School  of Mathematics and Statistics \& Key Laboratory of Analytical Mathematics and Applications (Ministry of Education) \& Fujian Provincial Key Laboratory
of Statistics and Artificial Intelligence, Fujian Normal University, 350007 Fuzhou, P.R. China. \url{jianwang@fjnu.edu.cn}}

\maketitle

\begin{abstract} In this work,
we are concerned with a class of multi-dimensional time-inhomogeneous stochastic differential equations (SDEs) on
$\R^d$ driven by pure-jump L\'evy processes, where the drift coefficient
$b(t,x)$  satisfies $\lim_{t\to \infty}b(t,x) =0$ for every $x\in \R^d$.
On  account of  three regimes associated with  the index
$\alpha$ of large jumps
corresponding to the driven L\'evy noise, we investigate the quantitative asymptotics  of the corresponding rescaled processes.
More precisely, for
$\alpha\in (0,2)$, we prove that the rescaled processes, governed by time-inhomogeneous SDEs subject to  additive processes, converge with respect to suitably chosen Wasserstein distances to time-homogeneous SDEs driven by symmetric
$\alpha$-stable processes.
Notably, the driven noise in the limiting SDEs depends only on large jumps of the underlying additive processes. In case of $\alpha\ge2$,
a phase transition occurs  and a diffusive phenomenon arises.
In particular, in the setting  $\alpha>2$, we establish the ergodicity of the rescaled process by means of asymptotic pseudotrajectories.
The resulting time-homogeneous limiting SDEs are driven by Brownian motions, even though the transformed time-inhomogeneous SDEs are driven by (discontinuous) additive processes, and the effective noise intensity is determined by the entire L\'evy measure of the original pure-jump process.
As  far as the critical case $\alpha=2 $  is concerned, we demonstrate  that the noise intensity of the limiting SDEs driven by Brownian motions relies merely   on the large-jump part of  the L\'{e}vy measure.
Note that the  L\'{e}vy measure under investigation  is $L^2$-integrable and non-$L^2$-integrable (with respect to large jumps) as soon as $\alpha>2$ and $\alpha=2,$ respectively. This implies that
the scaling transformations we adopted are totally different, and that the approximation between
the infinitesimal generator related to the rescaled time-inhomogeneous process and the counterpart corresponding to the time-homogeneous limiting version  is
 also
 significantly  discrepant.

\medskip


\smallskip

\end{abstract}

\section{Background and main results}
\subsection{Background}
It is generally known that the classical (time-homogeneous) diffusion processes can be applied to model numerous  real-world systems, where the underlying environment is constant or unchanged on a large time scale.
By contrast,
the time-inhomogeneous diffusion processes are ideal candidates that are
 better  suited to describe  realistic systems with  time-varying dynamics (e.g.\ financial quantities in financial markets, biological systems, control  systems, climate modeling, to name just a few). As far as time-homogeneous SDEs and time-inhomogeneous SDEs are concerned,  the
 key difference lies in the fact that the latter
 ones
 allow the coefficients to vary with time, which
 implies that the dynamics described by time-inhomogeneous SDEs are influenced  by external factors that change over time. Apparently,   the time evolution (so the Markov transition kernel and the infinitesimal generator) of a time-inhomogeneous diffusion process
 depends on the initial time,  so the path
exhibits a non-stationary feature.  Such  time-dependence might lead to non-existence of steady states and make the long-term analysis  mathematically intricate. More  prominently, the time-varying nature of a time-inhomogeneous diffusion process
means that many approaches, which are applicable to tackle the long-time asymptotics for time-homogeneous
Markov processes,
do not directly apply (or completely fail) to time-inhomogeneous Markov processes. For instance, Harris' ergodic theorem (see e.g. \cite[Theorem 2.1]{HM}) and the Krylov-Bogoliubov existence theorem (see e.g. \cite[Theorem 3.1.1]{DZ}) for invariant probability measures (IPMs for short) are inapplicable to  time-inhomogeneous Markov processes.
In contrast to a number of classical concepts and research topics concerned with time-homogeneous Markov processes, due to the time-varying feature,   important conceptual shifts are required (e.g.,  IPMs vs evolution system  of measures (ESMs for abbreviation), ergodic theorem vs asymptotic pseudotrajectory, equilibrium vs random attractor, semigroup asymptotics vs homogenization) when
time-inhomogeneous Markov processes are taken into account. Based on the aforementioned viewpoint, it is far from a  trivial task to carry out the research work regarding the
long-time behavior of
time-inhomogeneous Markov processes.

 In the past few decades, the asymptotic analysis (e.g., weak ergodicity \cite{SZ,Zei,ZI}, $L_1$-weak ergodicity \cite{Mukh} and asymptotic ergodicity \cite{BBC}) of time-inhomogeneous Markov  processes/chains
  has advanced considerably. In the meantime, notable  progress has been made on the long-term behavior of time-inhomogeneous SDEs/SPDEs. Below, we give an incomplete overview of   related references.
 By means of the pullback method,   \cite[Proposition 2.3]{DR} and \cite[Proposition 2.5]{DR} investigated
 the existence and the uniqueness of an  ESMs   for a semi-linear SPDE with time-dependent dissipative coefficients, and moreover \cite[Proposition 2.4]{DR} revealed that the non-homogeneous Markov semigroup behaves asymptotically as $t\to\infty$ like a limit curve for
 a continuous and bounded observable
 function. \cite{DRb} generalized the Krylov-Bogoliubov criterion (which was initially developed  to the existence of IPMs for autonomous stochastic  systems) to prove the existence of an ESMs for  finite-dimensional  SDEs with time-dependent coefficients. Later on,  \cite[Theorem 3.1]{LL} recast  the Krylov-Bogoliubov criterion for general  time-inhomogeneous Markov processes  and   \cite[Theorem 3.2]{LL} provided a sufficient Lyapunov condition for the  existence of an ESMs. Subsequently, in terms of space-time Lyapunov functions, the existence of an ESMs for SDEs on   evolving manifolds was explored in \cite{CT} and, under suitable curvature
conditions,   the uniqueness was also  established therein via the coupling method. Recently,
under a Lyapunov condition and a
uniform ``minorization'' condition reminiscent of Doeblin's condition,
by following the strategy in \cite{HM}, the weak contraction under the  weighted total variation distance was derived in \cite[Theorem 3.3]{LL} for time-inhomogeneous Markov processes. In particular, the derived  theory   was also applied to non-degenerate SDEs with time-periodic coefficients. With regard to time-periodic SDEs,  \cite{FZZ}   gave  sufficient
conditions for the existence and the  uniqueness of periodic ESMs, which were  built   based on  the existence of IPMs
 of the associated skeleton Markov chains. With the aid of the classical Harris theorem (which is applicable to the corresponding skeleton chain), the ``moving'' convergence and the pullback convergence under the total variation distance were handled in \cite[Theorem 3.2]{FZZ}  for time-periodic   gradient systems and the  Langevin   dynamics.
 Furthermore, we would like to allude to \cite{ALL,CT} concerning functional inequalities with respect to an ESMs for
 nonautonomous Kolmogorov equations, and  \cite{CL,SWY,SWYb} and references therein on  the exploration of finite-time averaging principles for  time-inhomogeneous multi-scale stochastic systems.

The starting point of this paper is our recent work \cite{BSWX}, which  studied quantitative estimates for L\'{e}vy-driven SDEs with different drifts, where the drift coefficient associated with  one of the SDEs is allowed to be time-dependent.
Specifically,  the following time-inhomogeneous SDE on $\R^d$:
\begin{align}\label{eq1}
\d X_t=F(t,X_t) \,\d t+\d L_t,\quad \forall\, t\ge0
\end{align}
was considered in \cite{BSWX}, where  $F:[0,\infty)\times\R^d\to\R^d$ is continuous,
and $(L_t)_{t\ge0}$ is a pure-jump process with the L\'{e}vy measure $\nu(\d z)$. The SDE \eqref{eq1} is assumed to  have a unique strong solution $(X_t)_{t\ge0}$. Roughly speaking, we further suppose that there exist functions $\bar F:\R^d\to \R^d$, $\phi: [0,\infty)\to [0,\infty)$ satisfying that $\lim_{t\to\infty}\phi(t)=0$, and $V:\R^d\to [1,\infty)$ such that for all $t\ge0$ and $x\in \R^d$,
$$|F(t,x)-\bar F(x)|\le \phi(t) V(x),$$ where 
  $\bar F$ fulfills the so-called dissipative condition and $V$ satisfies the Lyapunov-type condition. The
 explicit convergence rates (which depend on $\phi$) were established in \cite[Theorem 2.5]{BSWX}. A typical example showing \cite[Theorem 2.5]{BSWX} is that for all $t\ge0$ and $x\in\R^d,$
$$F(t,x)=\bar F(x)+\phi(t)V(x) .$$ As $\bar F$ is required to fulfill the  dissipative condition (for example, $\bar F(x)=-x |x|^{\gamma-1}$ for some $\gamma\ge1$ and all $x\in \R^d$), in general $\lim_{t\to\infty} F(t,x)\neq
0$ for fixed $x\in \R^d$.

Drawing on the research background in  physics and mathematics, there has been some progress on the asymptotic behavior  of time-inhomogeneous diffusions with drifts that vanish asymptotically in time,
i.e., $\lim_{t\to\infty}F(t,x)=0$ for all $x\in \R^d$.
As a toy example, we consider a time-inhomogeneous Ornstein–Uhlenbeck-type process $(X_t)_{t\ge0}$ on $\R$, which is governed by the following SDE: for all $t\ge0,$
\begin{align}\label{D-}
\d X_t=-t^{-\beta}X_t\,\d t+\d L_t,
\end{align}
where $(L_t)_{t\ge0}$ is a (rotationally) symmetric $\alpha$-stable process with
parameter $\sigma>0$ (i.e., $\E\e^{iuL_t}=\e^{-\sigma^\alpha|u|^\alpha t}$ 
for all $t>0$ and $u\in \R$). Based on the explicit expression of the solution to \eqref{D-} and the assumption  that $(L_t)_{t\ge0}$ is a symmetric $\alpha$-stable Lévy process, we can deduce that for all $t>0,$
\begin{equation}\label{D1-}
\begin{split}
&t^{-\beta/\alpha}X_t\overset{d}\Longrightarrow S_\alpha\Big(\frac{\sigma}{\alpha^{1/\alpha}}\Big),\quad\qquad\,\, \beta<1; \\
&t^{-1/\alpha}X_t\overset{d}\Longrightarrow S_\alpha\Big(\frac{\sigma}{(\alpha+1)^{1/\alpha}}\Big),\quad \beta=1;\\
&t^{-1/\alpha}X_t\overset{d}\Longrightarrow S_\alpha (\sigma),\quad \qquad\qquad\,\,\,\beta>1,
\end{split}
\end{equation}
where the symbol ``$\overset{d}\Longrightarrow$'' means the convergence in distribution, and $S_\alpha (\sigma)$ stands for  the
symmetric $\alpha$-stable distribution with
 parameter $\sigma$. Nevertheless, when  the asymptotically vanishing drift  under consideration is nonlinear in the spatial variable, the study of the asymptotic behavior of the corresponding rescaled solution poses an extremely non-trivial challenge.
The following known results have attracted our attention. \cite{GO} considered a one-dimensional diffusion with a time-inhomogeneous drift coefficient $F(t,x)=\rho\mbox{sgn}(x)|x|^\alpha/{t^\beta}$, and investigated the asymptotic distributions with the help of the Motoo theorem (which is valid since the driven noise is Brownian motion), the ergodic theorem as well as the comparison theorem. \cite{GLa} studied a kinetic stochastic model with a non-linear time-homogeneous friction force $F(t,v)=-t^{\beta}F(v)$ (where $v\mapsto F(v)$ satisfies a homogeneity condition) and a Brownian-type random force and, by making use of the Portmanteau theorem,  investigated  the limit in distribution of the time-space rescaled process corresponding to the pair $(V_t,X_t)$ with $(X_t)_{t\ge0}$ being the position and $(V_t)_{t\ge0}$ its velocity of the particle involved. Soon afterwards, the work \cite{GLa} was extended to a one-dimensional kinetic stochastic model in \cite{GL}, where the drift is of the form $F(t,v)=t^{-\beta}F(v)$ and the driving process is an
$\alpha$-stable L\'{e}vy process, and the long-time behavior   of the process $(X_t,V_t)_{t\ge0}$ is analysed based on moment estimates and the self-similarity of the driving process.

As regards one-dimensional time-inhomogeneous SDEs with  asymptotically vanishing drifts in time parameters, the references \cite{GO,GLa,GL} mentioned previously
focus mainly on the qualitative exploration (e.g., recurrence, transience and convergence in distribution). However, as far as (high-dimensional) time-inhomogeneous SDEs with asymptotically vanishing drifts are concerned,
the quantitative analysis (e.g., the convergence and the explicit convergence rate under suitable (probability) distances) in the infinite time horizon is still vacant at the present time.
Even for such a simple model \eqref{D-},
the  convergence rate (under suitable probability distances) of the rescaled process to an equilibrium is still unavailable.
The above  urges us to  conduct the present research project and fill in the  gaps left over the past several years.

\subsection{Main results}\label{sec1.2}
Throughout the present paper, we assume that
 the
  non-zero
  L\'evy measure $\nu$ is endowed with  a generic jump structure   for the microscopic behavior   together with an exact
  symmetric $\alpha$-stable tail for the macroscopic behavior. Specifically, we make the following assumption on the L\'evy measure $\nu$.
\begin{itemize}\it
\item[ $({\bf H}_1)$] there exist  constants
$c_*>0$, $\alpha\in (0,\infty]$ and a function  $a:\bar{B}_1:=\{z\in \R^d: |z|\le 1\}\to[0,\infty]$ satisfying that $a(z)=a(-z)$ for all $z\in \bar{B}_1$ and $\int_{\{|z|\le 1\}} |z|^2 a(z)\,\d z<\infty$, so that
 \begin{align}\label{Levy}
\nu(\d z) =a(z)\I_{\{|z|\le1\}}\,\d z+\frac{c_*}{|z|^{d+\alpha}}\I_{\{|z|>1\}}\,\d z.
\end{align}
\end{itemize}

In this section, in order to avoid introducing excessive assumptions enforced on $F(t,x)$, we merely choose a typical representative of $F(t,x)$ to state the main results, while further details on the more general asymptotic theory can be found in Section \ref{sec4},
Remark \ref{rema} and Remark \ref{R:7.3}.
To this end, we state the  hypotheses imposed on $F(t,x)$, which need not enjoy  the scaling property in the time/spatial variables,  as follows.

\begin{itemize}
\item[ $({\bf H}_2)$]there exist constants $\beta,\gamma>0$  and $0\neq\lambda\in\R$ such that for all $t>0$ and $x\in\R^d,$
\begin{align*}
F(t,x)=(1+t)^{-\beta}F_*(x)\quad\mbox{ with } \quad F_*(x):=\lambda x(1+|x|^2)^{\frac{1}{2}(\gamma-1)}.
\end{align*}

\end{itemize}

The purpose of this paper is to establish the long-time behavior of the process $(X_t)_{t\ge 0}$ by adopting an argument based on the space-time change.
We first present the form of the SDE solved by the limiting process of
the rescaled process associated with
 $(X_t)_{t\ge0}$.
In detail,
the relationship among the
  parameters $\alpha$, $\beta$ and $\gamma$ in Assumptions $({\bf H}_1)$ and $({\bf H}_2)$   roughly determines
   the following SDEs:
  for some $\eta>0$ and
$\gamma_0\ge0$
   (which will be fixed later on in the specific setting),
\begin{equation}\label{e:limit}
	\d  \overline{Y}_{\!\! t} = \begin{cases}  -    (\overline{Y}_{\!\! t}/\eta
		) \,   \d t+ \d   L_t^{(\alpha)},&\quad \beta>1+ (\gamma-1)/(\eta\wedge2),\\
		\big(-  \overline{Y}_{\!\! t}  /\eta + \lambda \overline Y_{\!\! t}|\overline Y_{\!\! t}|^{\gamma-1} \big)\,\d t+ \d   L_t^{(\alpha)},&\quad \beta=1+ (\gamma-1)/(\eta\wedge2),\\
		\lambda (1+2\gamma_0)\overline Y_{\!\! t}|\overline Y_{\!\! t}|^{\gamma-1}\, \d t+ \d    L_t^{(\alpha)},&\quad \beta<1+ (\gamma-1)/(\eta\wedge2),\end{cases}
\end{equation}
where $a\wedge b: =\min\{a,b\}$  for $a,b\ge0$.
In \eqref{e:limit},  $(L_t^{(\alpha)} )_{t\ge0}$ is
 a  symmetric $\alpha$-stable process with the L\'{e}vy measure $\nu^{(\alpha)}(\d z):=\frac{c_*}{|z|^{d+\alpha}} \,\d z$ when $\alpha\in (0,2)$; $(L_t^{(\alpha)} )_{t\ge0}$ is a $d$-dimensional Brownian motion
with the covariance matrix $c_* \omega_d I_d$ (resp. $\frac{1}{d}\nu(|\cdot|^2) I_d$) when $\alpha=2$ (resp. $\alpha\in (2,\infty]$), in which $\omega_d$ stands for the volume of $\bar B_1$, $I_d$ represents the identity $d\times d$-matrix.
 Meanwhile, when $\alpha>2$, we assume that
 $a(z)=a(|z|)$ for all $z\in \bar B_1$ and so
 the matrix $\int_{\R^d}  (z\otimes z)\,\nu(\d z)=\frac{1}{d}\nu(|\cdot|^2)I_d$ is strictly positive definite,
where
$z\otimes z$ means the tensor product of the vector $z$ with itself.
On some occasions, we  write $(X_t^\mu)_{t\ge0}$
  in lieu of $(X_t)_{t\ge0}$
   when  we would like to   emphasize the initial distribution $\mathscr L_{X_0}=\mu$,
    where,
for a vector-valued random variable $\xi$,
$\mathscr L_\xi$ denotes   its law.

\subsubsection{\bf{Asymptotics: jump scaling with $\alpha\in(0,2)$}}\label{subsection1.2}
For $p>0,$
let  $\mathcal P_p(\R^d)$ be
the set of probability measures (written as   $\mathcal P(\R^d)$)  with finite $p$-th moment. For a   distance-like function
 $\rho:\R^d\times\R^d\to[0,\8)$, let
$\mathcal W_\rho$ be  the quasi-Wasserstein distance induced by $\rho$, which is defined by
\begin{align*}
\mathcal W_\rho(\mu_1,\mu_2)=\inf_{\pi\in\mathscr C(\mu_1,\mu_2)}\int_{\R^d\times\R^d}\rho(x,y)\,\pi(\d x,\d y),\quad\forall\, \mu_1,\mu_2\in\mathcal P(\R^d),
\end{align*}
where
$\mathscr C(\mu_1,\mu_2)$ stands for the collection of all couplings of  $\mu_1$ and $\mu_2$.  In case of the underlying  metric function $\rho(x,y)=|x-y|,$
we shall write $\mathcal W_1$ instead of $\mathcal W_{|\cdot|}$ for simplicity.

\begin{theorem}\label{theorem-1} {\bf($1\le \alpha<2$)}
Assume that $({\bf H}_1)$, $({\bf H}_2)$, $1\le \alpha<\theta_1\le2 $, $ \theta_2\in(0,\alpha)$ and $ \int_{\{|z|\le 1\}} |z|^{\theta_1}  a( z )\,\d z<\infty$ hold.   Suppose further that
one of the following conditions is satisfied:
\begin{enumerate}
\item[$(a)$]$\beta>1+(\gamma-1)/\alpha$,  $(i)$ $\lambda>0$, $ \gamma\in(0,1)$ and $\alpha=1$, or $(ii)$ $\lambda>0$, $ \gamma\in(0,1]$ and $\alpha\in(1,2)$, or $(iii)$ $\lambda<0 $ and $0<\gamma<\alpha/{(1\vee\theta_2)}$;

\item[$(b)$]$\beta=1+(\gamma-1)/\alpha$, $\lambda<0,$
 $(i)$  $\gamma\in(0,1)$
and $\alpha=1$, or $(ii)$  $0<\gamma<\alpha/{(1\vee\theta_2)}$
and $\alpha\in(1,2)$;

\item[$(c)$]$\beta<1+(\gamma-1)/\alpha$, $\alpha\in(1,2)$,  $\lambda<0$,     $1\le \gamma<\alpha/{(1\vee\theta_2)}$.
\end{enumerate}
Then, under   case $(a)$  $($resp. case $(b)$$)$,
there
is a constant $\lambda_1^*>0$ such that for  all   $t\ge  1$ and  $\mu\in\mathcal P_{\theta_2\vee(\gamma(1\vee\theta_2))}(\R^d) $,
\begin{equation}\label{EW8}
	\mathcal W_{|\cdot|^{\theta_1}\wedge|\cdot|^{\theta_2}}\Big(\mathscr L_{t^{- {1}/{\alpha}}X_t^{\mu}},\pi\Big)  \le C_1^*t^{-\lambda_1^*} ,
\end{equation}
where
$\pi\in \mathcal P_{\theta_2}(\R^d)$ is the unique IPM of $(\overline Y_{\!\! t})_{t\ge0}$ solving the first $($resp. the second$)$ SDE in \eqref{e:limit} with $\eta=\alpha$, and the constant $C_1^*>0$ depends on $\mu$ and $\pi$; under   case $(c)$, there
 exists a constant $\lambda_2^*>0$
 such that for  all
 $t\ge1$ and  $\mu\in\mathcal P_{\theta_2\vee(\gamma(1\vee\theta_2))}(\R^d) $,
 \begin{equation}\label{EY-*}
 	\mathcal W_{|\cdot|^{\theta_1}\wedge|\cdot|^{\theta_2}}\Big(\mathscr L_{t^{-q}X_t^{\mu}},\pi\Big)  \le C_2^*
 	t^{-\lambda_2^*},
 \end{equation}
 where $q:= \beta/{(\alpha+\gamma-1)}$,  $\pi\in \mathcal P_{\theta_2}(\R^d)$ is the unique IPM of $(\overline Y_{\!\! t})_{t\ge0}$ which is determined by the third SDE in \eqref{e:limit} with $\eta=\alpha$ and   $\gamma_0=0$,
 and the constant
 $C_2^*>0$ depends on $\mu$ and $\pi$.
\end{theorem}

Below, we make some comments on Theorem \ref{theorem-1}.
It is easy to see that $q=\beta/\alpha$ when  $\gamma=1.$ Therefore, concerning the three regimes (i.e., $\beta>1$, $\beta=1$ and $\beta<1$) with $\gamma=1$, the respective space scaling in Theorem \ref{theorem-1} is the same as that given in \eqref{D1-}. Additionally, Theorem \ref{theorem-1} demonstrates that the law of the rescaled process associated with \eqref{D-} converges under a well-tailored Wasserstein distance to the respective equilibrium corresponding to the limiting SDE and moreover provides  an explicit convergence rate. Most importantly,  Theorem \ref{theorem-1} shows   that  the L\'evy measure related to  the limiting SDE is determined merely by the large-jump  part of the original L\'evy measure (i.e., only by the parameters $c_*$ and $\alpha$ in \eqref{Levy}).

In \eqref{EW8} and \eqref{EY-*}, explicit formulas for $\lambda_1^*$ and $\lambda_2^*$ are not given owing to their complexity. Indeed,
by tracing  the proof of Theorem \ref{theorem-1},   $\lambda_1^*$ and $\lambda_2^*$ can be presented explicitly as follows:
\begin{align*}
\lambda_1^*&=\frac{1}{2} (\theta_2/
(16\alpha))\wedge(\theta_1/\alpha-1)  \wedge\frac{1}{2}(1\vee\theta_2)\begin{cases}
 \big( (1-\gamma)/\alpha+\beta-1\big)
	& \mbox{if } \beta>1+(\gamma-1)/\alpha, \\
	& \mbox{if } \beta=1+(\gamma-1)/\alpha, \\
   ( \gamma/\alpha\wedge 1)  	&(i)\,\alpha=1,\gamma\in(0,1),\\
	& \mbox{ or }  (ii)\, \alpha\in(1,2),\gamma\in(0,1],\\
	&
	\mbox{ or }  (iii)\, \alpha\in[1,2), \\
	&\quad\quad\quad\,\,\gamma\in(0,\alpha/{(1\wedge\theta_2)}),
	\\
  \big(1\wedge ((2\wedge(\gamma-1))/\alpha)\big) 	& \mbox{if } \beta<
  1+(\gamma-1)/\alpha,\\
	& \alpha\in(1,2),  \gamma\in(1,\alpha/{(1\vee\theta_2)}),
\end{cases}	
\end{align*}
and
\begin{align*}
\lambda_2^*=
\begin{cases}
\lambda_3^*:=   \big(\theta\left( \theta_1/\alpha -1\right)\big)\wedge\big(((1 \vee \theta_2)  (1-\theta))\wedge(\theta [2 \wedge (\gamma-1)]/\alpha)  \big) , 	&\gamma=1,\\
   \big(\theta_1(1-\theta)/{(\gamma-1)}\big)\wedge   \big(\theta_1\lambda_3^*/{ (\theta_1 + \gamma - 1)}\big )   ,&\gamma>1,
	\end{cases}
\end{align*}
where  $\theta:=\frac{\alpha\beta}{\alpha+\gamma-1}\in(0,1)$. It should be noted that the explicit dependence of the constants $C_1^*$ and $C_2^*$ on the measures $\mu$ and $\pi$ can also be provided. More precisely,
concerning case $(a)$  and  case $(b)$,
 the constant $C_1^*>0$ is linearly dependent on
$\mu(|\cdot|^{\theta_2\vee(\gamma(1\vee\theta_2))}) $ and $\pi(|\cdot|^{\theta_2})$,   and, regarding case $(c)$,
  the constant $C_2^*>0$  depends linearly on $\mu(|\cdot|^{(1\vee\gamma)\theta_2})$ and $ \pi(|\cdot|^{\theta_2})$.
  Moreover,  Theorem \ref{theorem-1} can be extended to a much more general framework;  see Section \ref{sec4} for more details. In particular,   Theorem \ref{theorem-1} is still valid provided that the drift $F(t,x)$ is decomposed into two parts, where one part satisfies Assumption $({\bf H}_2)$ and the other part is a corresponding lower-order perturbation. Additionally,  if Assumption $({\bf H}_2)$ is replaced by the following one:
  \begin{itemize}
  	\item[ $({\bf H}_2')$]there exist constants $\beta,\gamma>0$, $c_0,c_0^*\ge0$  and $0\neq\lambda\in\R$ such that for all $t>0$ and $x\in\R^d,$
  	\begin{align*}
  		F(t,x)=(c_0+t)^{-\beta}F_*(x)\quad\mbox{ with } \quad F_*(x):=\lambda x(c_0^*+|x|^2)^{\frac{1}{2}(\gamma-1)},
  	\end{align*}
  \end{itemize}
   we can also establish a counterpart of Theorem \ref{theorem-1}, where the associated limiting SDE remains unchanged, while the range of the parameter $\gamma$ varies  as $c_0$ and $c_0^* $ change.

In Theorem \ref{theorem-1}, the
index $\alpha\in[1,2)$ is restricted. For the case $\alpha\in(0,1)$,
we can still obtain the asymptotics of $(X_t)_{t\ge0}$. In particular, when
 the integrability with respect to the small jumps of the L\'{e}vy measure  is strengthened and   the  quasi-Wasserstein distance is adjusted accordingly, we have the subsequent theorem.

\begin{theorem}\label{thm2} {\bf($0<\alpha<1$)}\,\,
Assume that $({\bf H}_1)$, $({\bf H}_2)$, $\int_{\{|z|\le1\}}|z|a(z)\,\d z<\infty$ and $\beta\ge 1+(\gamma-1)/\alpha$   with
$0<\gamma<\alpha<1$ hold.	 Then,
there
is a constant $\lambda^*>0$
such that for any $t\ge1$ and $\mu\in\mathcal P_{\gamma }(\R^d)$,
\begin{equation} \label{P-23}
	\mathcal W_{|\cdot|\wedge|\cdot|^\gamma}\Big(\mathscr L_{t^{-{1}/{\alpha}}X_t^{\mu}},\pi\Big)   \le    C^* t^{-\lambda^*},
\end{equation}
where $\pi\in\mathcal P_\gamma(\R^d)$ is the unique IPM  of
$(\overline{Y}_{\!\! t})_{t\ge0}$ solving respectively the first $($resp. second$)$ SDE in \eqref{e:limit} with $\eta=\alpha$ when  $\beta> 1+(\gamma-1)/\alpha$   $($resp. $\beta= 1+(\gamma-1)/\alpha$$)$,
and $C^*>0$ depends linearly  on $\mu(|\cdot|^{\gamma })$ and $\pi(|\cdot|^\gamma)$.
\end{theorem}

To proceed, we make some comments on Theorem \ref{thm2}.
Throughout this paper, we are interested in the case $\beta>0$ (see (${\bf H}_2$)) so it is natural to require $\gamma>1-\alpha$,  which is
the exact condition required
in Theorem \ref{thm3} below with $\alpha_*=\gamma$ and $\beta=\alpha$ therein.
In Theorem \ref{thm2}, we present the result regarding the case $\beta\ge1+ (\gamma-1)/\alpha$ while excluding the case $\beta<1+ (\gamma-1)/\alpha$. For the latter case, to derive the counterpart of \eqref{P-23}, it is necessary that the first-order
moment of the rescaled process associated with $(X_t)_{t\ge0}$ is uniformly bounded.
 Nevertheless, such a requirement is impossible in case of $\alpha\in(0,1).$
 Meanwhile, we want to emphasize  that, as stated in  the
 remarks below Theorem \ref{theorem-1}, the convergence 
 rate $\lambda^*$ involved in
Theorem \ref{thm2} can also be provided by inspecting the proof  of
Theorem \ref{thm2}.

\subsubsection{\bf{Asymptotics: diffusive scaling with $\alpha>2$}}\label{sub1.2}
In the previous part, for  the case $\alpha\in(0,2),$
we quantitatively measure
the law of
the rescaled process associated with $(X_t)_{t\ge0}$    and the corresponding
IPM of the
 limiting process $(\overline{Y}_{\!\! t} )_{t\ge0}$ under appropriate  quasi-Wasserstein distances. In case of $\alpha\in(0,2),$
it is worth  stressing   that the limiting process $(\overline {Y}_{\!\! t})_{t\ge0}$ solves an SDE in which the driven noise is
a symmetric $\alpha$-stable noise.  In this and the following  parts, we turn  to analyze  the asymptotics of $(X_t)_{t\ge0}$ as soon as the index $\alpha\ge 2$ related to  large jumps.   In this setting, the corresponding limiting process will change dramatically and the diffusive phenomenon will take place.

In contrast  with the case   $\alpha\in (0,2)$, we adopt a different metric (which is called the asymptotic pseudotrajectory metric for the time being) to quantify    the asymptotics of $(X_t)_{t\ge0}$ in case of $\alpha\ge2$. Such a metric  is closely related  to  the notion of an asymptotic pseudotrajectory  of time-inhomogeneous Markov semigroups (see e.g. \cite[p.\ 3007]{BBC}).
For $\mu_1,\mu_2\in\mathscr P(\R^d)$, define the asymptotic pseudotrajectory metric as follows:
\begin{align*}
 {\rm d}_{\mathcal F}(\mu_1,\mu_2)=\sup_{f\in\mathcal F}|\mu_1(f)-\mu_2(f)|,
\end{align*}
 where $\mathcal F$ is a suitable family of functions $f:\R^d\to\R $ satisfying  $\mu_1(|f|)<\8$ and $\mu_2(|f|)<\8.$

Below,
we
 focus on the case   $\alpha>2$ in $({\bf H}_1)$. The main result in this part is stated as follows.

\begin{theorem}\label{thm5-}{\bf($\alpha>2$)}
Assume that $\alpha>2$, $({\bf H}_1)$ and $({\bf H}_2)$ hold,
and that $a(z)=a(|z|)$ for all $z\in \bar B_1$. Suppose further that one of the following conditions holds:
\begin{itemize}
\item[$(a)$]  $ \beta>(1+\gamma)/2$, $(i)$ $\lambda>0$ and  $\gamma\in(0,1]$, or $(ii)$  $\lambda<0 $ and $0<\gamma<\alpha $;
\item[$(b)$]  $\beta=(1+\gamma)/2$, $\lambda<0$ and
$1\vee(3-d/2)<\gamma<\alpha $;

\item[ $ (c)$]
$ \beta<(1+\gamma)/2$, $\lambda<0$ and
$1\vee(3-d/2)<\gamma<\alpha $.
\end{itemize} Then, for  case $(a)$ $($resp.  case $(b)$$)$,
there exist constants
 $ \lambda_1, t_1>0$   such that for all
 $ t\ge t_1$ and $\mu \in \mathcal P_{ \gamma}(\R^d)$,
\begin{equation}\label{D26}
{\rm d}_{\mathcal F_c}\big(\mathscr{L}_{t^{-1/2}X_t^\mu},\pi  \big)\le C_1 t^{-\lambda_1};
\end{equation}
 regarding  case $(c)$, there exist constants $\lambda_2,t_2>0$ such that for all $t\ge t_2$ and
$\mu\in\mathcal P_{\gamma}(\R^d)$,
\begin{equation}\label{D26-}
{\rm d}_{\mathcal F_c} \big(\mathscr{L}_{t^{- \beta/{(1+\gamma)}}X_t^\mu},\pi  \big)\le C_2 t^{-\lambda_2}.
\end{equation}
Herein, $ \mathcal F_c:=\{f\in C_c^\infty(\R^d):
\sum_{i=0}^3\|\nabla^{i} f\|_\infty\le 1\}$,
$\pi$ represents  the unique IPM  corresponding to $(\overline Y_{\!\! t})_{t\ge0}$ solved respectively by the SDEs in \eqref{e:limit}
with $\eta=2$, $\gamma_0=0$ and driven by Brownian motion with the covariance matrix $\frac{1}{d}\nu(|\cdot|^2) I_d$
and the constants $C_1,C_2>0$ depend on $\mu$ and $\pi$.
\end{theorem}

The constants  $\lambda_1$ and $\lambda_2$ involved in \eqref{D26} and \eqref{D26-}  can  be given explicitly, respectively,   by inspecting the proof of Theorem \ref{thm5-}, even though they are somewhat complex.
Meanwhile, the explicit expressions of $C_1,C_2$ and their dependence on $\mu$ and  $\pi$
can also be furnished (for instance, $C_1$ involved in \eqref{D26} is linearly dependent on $\mu(|\cdot|^{\gamma})$  and  $\pi(|\cdot|)$ concerning   case $(a)$).
Furthermore, 
inspired by the core idea behind the proof of Theorem \ref{thm5-}, the counterpart of Theorem \ref{thm5-} can also be established for
more general models; see Remark \ref{rema} for more related details.

\subsubsection{\bf{Asymptotics: critical scaling with $\alpha=2$}}
In the previous part, we address the asymptotic analysis of the process $(X_t)_{t\ge0}$ in case of the index $\alpha>2.$ In such case, the L\'{e}vy measure $\nu(\d z)$ is 
$L^2$-integrable with respect to large jumps; that is, $\nu(|\cdot|^2\I_{\{|\cdot|>1\}})<\8.$ In the present subsection, we turn to the critical case: $\alpha=2.$ For this case, it is obvious that $\nu(|\cdot|^2\I_{\{|\cdot|>1\}})=\8$. Nevertheless, regarding  this setting, we can still explore the asymptotics of $(X_t)_{t\ge0}$ by choosing  a  suitable time scaling $\varphi_t$ and a space scaling $g_t$ as shown in the subsequent theorem.

\begin{theorem}\label{thm6}{\bf($\alpha=2$)}
	Assume that $\alpha=2$, $({\bf H}_1)$ and $({\bf H}_2)$ hold. Suppose further that one of the following conditions holds:
	\begin{itemize}
		\item[$(a)$]  $ \beta>(1+\gamma)/2$, $(i)$ $\lambda>0$ and  $\gamma\in(0,1]$, or $(ii)$  $\lambda<0 $ and $0<\gamma<2 $;
		\item[$(b)$]  $\beta=(1+\gamma)/2$ with $\gamma=1$,   $\lambda<1/2 $;

		\item[ $ (c)$]
		$ \beta<(1+\gamma)/2$, $\lambda<0$ and
		$1\vee(3-d/2)<\gamma<2 $.
	\end{itemize}
Then,
 for  case $(a)$ $($resp.   case $(b)$ and case $(c)$$)$,
there exists a constant $t_*>0$ such that for all $t\ge t_*$ and
$\mu \in \mathcal P_{1\vee\gamma}(\R^d)$,
 \begin{equation}\label{G4}
 {\rm d}_{\mathcal F_c}\Big(
 \mathscr{L}_{g_{\varphi_t^{-1}}X_t^\mu},\pi\Big )  \le C_0
\begin{cases}
\frac{1}{(\ln t)^{{1}/{2}}} &\mbox{ if } \,    \beta>(1+\gamma)/2,\\
 \frac{1}{(\ln t)^{1/2-\lambda }} &
 \mbox{ if }     \beta=(1+\gamma)/2,\\
\frac{1}{(\ln t)^{{1}/{2}}}  & \mbox{ if }\, 1=\beta<(1+\gamma)/2,
 \end{cases}
\end{equation}
where $\mathcal F_c$ is the same as that given in Theorem $\ref{thm5-}$, $\pi$
  denotes    the unique IPM associated with the respective    SDEs   in \eqref{e:limit}
with $\eta=2$, $\gamma_0=1/{(\gamma-1)}$ and
 driven by Brownian motion with the covariance matrix $c_*\omega_d I_d$,
 $	g_t:=\e^{-(t+1)/2}$ and
 $\varphi_t^{-1}$ is the inverse function of $\varphi_t:=1+2\int_1^{1+t}\frac{\e^{ u}}{ u}\,\d u$ for  cases $(a)$ and $(b)$  $($resp. $g_t:=(\e+t)^{-\gamma_0}$  and $\varphi_t^{-1}$ is the inverse function  of $\varphi_t:=1+ \frac{1}{\gamma_0}
 \int_\e^{\e+t}\frac{ u^{2\gamma_0}}{ \ln u}\,\d u$ with regard  to  case $(c)$$)$, and the constant $C_0>0$ depends on $\mu$ and $\pi$.
\end{theorem}

Before the end of this section, we summarize some novel aspects of the present work.
\begin{enumerate}
\item[(1)] As far as SDEs with asymptotically vanishing drifts are concerned, the  existing literature focuses on the qualitative analysis of the rescaled solution processes
    essentially only concerned with the case that $\alpha\in (0,2)$; see, for instance,  \cite{GO,GLa,GL} and references therein. Nonetheless, the present work is devoted to   quantitatively exploring and providing explicit convergence rates under well-adapted  probability distances
    for all $\alpha\in (0,\infty]$. Indeed, to the best of our knowledge, we believe that the assertions for the case  $\alpha\in [2,\infty]$ in our paper are new in the literature.
    For SDEs with asymptotically non-vanishing drifts, \cite{BSWX}  addressed the issue of quantitative convergence bounds. However,  the approach adopted in \cite{BSWX}  is not applicable  to the present framework due to the asymptotically vanishing property of drifts under investigation.
     On the one hand, we choose well-designed time scaling and space scaling in respective scenarios to derive the rescaled processes and then to study the corresponding convergence by invoking the coupling approach or the   method based on asymptotic pseudotrajectories.
   On the other hand, we shall emphasize that
     we need to consider  two SDEs with different driven noises in order to deal with the quantitative analysis over an infinite time horizon of the rescaled process subject to an additive process (not necessarily a Lévy process). Therefore, the framework under consideration in the present paper
is essentially distinct from the counterpart studied in \cite{BSWX}.

\item[(2)]In case of the index $\alpha\in(0,2)$ and $\alpha\ge2$ associated with large jumps of the Lévy measure, respectively, the time-homogeneous limiting  SDEs corresponding to the rescaled processes are completely different, where the driven noise of the former one is an  $\alpha$-stable process whereas  the latter   is driven by Brownian motion. It is worth noting that a phase transition   takes place and moreover a diffusive phenomenon occurs as long as the index $\alpha\ge2 $, which  have not yet been explored in the existing literature. Additionally, we realize  that, for the case $\alpha\in(0,2),$ the Lévy measure related to the limiting SDE is determined merely by the large-jump  part of the original Lévy measure; however, in the case $\alpha>2$ (resp. $\alpha=2$), the corresponding noise intensity  depends on the entire (resp. the large-jump part  of) Lévy measure.

\item[(3)]Concerning the case $\alpha\in(0,2)$, we employ  the Wasserstein distances which are constructed  subtly to measure the distance between the
 laws of  the rescaled process and the corresponding time-homogeneous version. By contrast, in the case $\alpha\ge2,$ the SDE  governed by the rescaled process and the limiting SDE are driven by two completely different types of noise, where  one is an additive process and the other one is Brownian motion. In this setup, it is a challenging task to investigate the convergence of the rescaled process via the Wasserstein distance with the aid of the coupling approach or the trick based on Duhamel's principle. To this end,
 we turn to the asymptotic pseudotrajectory distance to portray  convergence of  the rescaled process to the IPM corresponding to  the associated limiting SDE.
 In particular, to quantify the distance
 between laws associated with an  SDE with discontinuous driven noise and an SDE driven by Brownian motion, we believe that the metric $ {\rm d}_{\mathcal F_c}$ with the set of test functions $\mathcal F_c\subset C^3(\R^d;\R)$ is a perfect candidate.
  Meanwhile, the explicit convergence rate under the asymptotic pseudotrajectory distance is provided.
\end{enumerate}

The remaining part of this paper is arranged in the following way. In Section \ref{sec2}, with the aid of the coupling approach, we derive uniform-in-time estimates under well-tailored Wasserstein distances for two  SDEs with different drifts and different driven noises, 
where 
one of which is an additive process and the other a pure-jump L\'{e}vy process. In Section \ref{Section3-},
we present two preliminary results. In detail,  one is to characterize the rescaled noise process $(Z_t)_{t\ge0}$ defined in \eqref{RR-}, and the other
one
 is to provide sufficient conditions for establishing  the uniform-in-time moment boundedness of time-inhomogeneous SDEs driven by additive processes, where the index $\alpha$ of large jumps is allowed to be greater than $2$.
 In Section \ref{sec4}, we establish a  general asymptotic theory for the rescaled process under the Wasserstein distance
 (see Theorems \ref{thm0} and \ref{thm2-1}), and  prove Theorems \ref{theorem-1}  and \ref{thm2} by using the theory developed. Moreover, concerning the case $\nu(\d z)=\nu^{(\alpha)}(\d z)$,
 the corresponding  counterparts   of Theorem  \ref{theorem-1} and  Theorem \ref{thm0} are also derived in
 Subsection \ref{section4.4}. For this special case, the admissible range of the corresponding parameters is substantially relaxed.
 In Section \ref{Section4}, we set up
a general framework (see Theorem \ref{Th4.1}) for the convergence rates of the long-time behaviors of time-inhomogeneous SDEs, which is motivated by the theory on asymptotic pseudotrajectory related to
 time-inhomogeneous Markov semigroups.
Sections \ref{sec6} and  \ref{sec7}
are  devoted   to the proofs of Theorems  \ref{thm5-} and \ref{thm6}, where the quantitative error bound
(see Lemma \ref{lemma5} and Lemma \ref{lemma7}, respectively)
 between the infinitesimal generator related to the rescaled time-inhomogeneous process and the counterpart corresponding to the time-homogeneous limiting version plays an important role.

\section{Uniform-in-time estimates for SDEs with different drifts and different driven  noises: the coupling approach}\label{sec2}
In this part, we establish general theory to tackle the uniform-in-time estimates for SDEs with different drifts and (discontinuous) driven noises. For our purpose,
we focus on the following time-inhomogeneous SDE on $\R^d$:
\begin{align}\label{E1-}
	\d Y_t=b  (t,Y_t)\,\d t+\d Z_t,\quad \forall\,t\ge0,
\end{align}
where $b :[0,\infty)\times\R^d\to\R^d$ is measurable, and $(Z_t)_{t\ge0}$ is an additive process (that is, it has independent increments and is continuous in probability; see \cite[p.\ 3]{Sato}), defined on the filtered probability space $(\Omega,\mathscr F,(\mathscr F_t)_{t\ge0},\P)$, with the time-dependent L\'{e}vy measure $(\nu_t(\d z))_{t\ge0}$.

We assume that
$\bar b(x):=\lim_{t\to\8}b(t,x)$ exists for each fixed $x\in\R^d.$ By regarding
$\bar b(x)$ as a drift term, we consider  the time-homogeneous counterpart of the SDE \eqref{E1-}:
\begin{align} \label{E2-}
	\d \overline {Y}_{\! \! t}=\bar b(\overline {Y}_{\! \! t})\,\d t+\d \bar Z_t,\quad \forall\,t\ge0,
\end{align}
where $(\bar Z_t)_{t\ge0}$ is a pure jump L\'{e}vy process, supported on the same probability space as $(Z_t)_{t\ge0}$, with the (time-independent) L\'{e}vy measure $\bar\nu(\d z)$.

Throughout this section, we always assume that
the
   SDEs  \eqref{E1-} and   \eqref{E2-} are strongly well-posed.
We denote
 by $Y^\mu:=(Y_t^{\mu})_{t\ge0}$ (resp. $\bar Y^\mu:=(\overline {Y}_{\! \! t}^\mu)_{t\ge0})$
 the strong solution to the SDE \eqref{E1-} (resp. \eqref{E2-}) when the distribution of the initial point $Y_0$ (resp. $\overline Y_{\!\!0}$) is $\mu$.

\subsection{Coupling process}\label{section21}
To begin, we introduce some notation.
For   a given threshold $\kk\in(0,\infty]$, let
$(x)_\kk= (|x|\wedge\kk)x/|x|\I_{\{|x|\neq0\}}$
be the  truncated version of $x\in\R^d.$ Denote
by  $\delta_x\ast\mu$
the convolution of    the $\sigma$-finite measure $\mu$ on $(\R^d,\mathscr B(\R^d))$   and the Dirac delta measure $\delta_x$  centered at $x\in\R^d$. For $\sigma$-finite measures $\mu_1$ and $\mu_2,$ set $\mu_1\wedge \mu_2:=\mu_1-(\mu_1-\mu_2)^+$, where $(\mu_1-\mu_2)^+$ denotes the positive part of the Jordan decomposition of the signed measure $\mu_1-\mu_2$. For notational simplicity,
we set   for any $t\ge0$ and $x\in\R^d$,
$$\nu_{t,x}:=( \nu_t \wedge \bar\nu)\wedge(\delta_x\ast( \nu_t\wedge \bar \nu)). $$

In the following analysis, we stipulate  $\delta\in [0,1/2]$. Based on  the following  routine:
\begin{equation*}
	(x,y)\rightarrow
	\begin{cases}
		(x+z,y+z+(x-y)_\kk),\quad\quad &\delta\nu_{t,(y-x)_\kk}(\d z),\\
		(x+z,y+z+(y-x)_\kk),\quad\quad& \delta\nu_{t,(x-y)_\kk}(\d z),\\
		(x+z,y+z),\quad\quad&(\nu_t\wedge\bar\nu)(\d z)-\delta\nu_{t,(y-x)_\kk}(\d z)-\delta\nu_{t,(x-y)_\kk}(\d z),\\
		(x+z,y),\quad\quad&(\nu_t-(\nu_t\wedge\bar\nu))(\d z),\\
		(x,y+z), \quad\quad& (\bar\nu-(\nu_t\wedge\bar\nu))(\d z),
	\end{cases}
\end{equation*}
we define the   operator $\big(\overline{\mathscr L}^{(\delta)}_{\!\! t}\big)_{t\ge0}$ as follows:  for all $h\in C_b^2(\R^{2d})$ and $t\ge0$,
\begin{align}
	\big(\overline{\mathscr L}^{(\delta)}_{\!\! t} h\big)(x,y)&= \<\nn_1h(x,y), b(t,x)\>+\<\nn_2h(x,y),\bar b(y)\>\nonumber \\
		&\quad+
		\delta\int_{\R^d}\big(h(x+z,y+z+(x-y)_\kk)-h(x,y)-\<\nn_1h(x,y), z\> \I_{\{|z|\le1\}}\nonumber\\
		&\qquad\qquad\quad -\<\nn_2h(x,y), z+(x-y)_\kk \>\I_{\{|z+(x-y)_\kk|\le1\}}\big)\nu_{t,(y-x)_\kk}(\d z)\nonumber\\
		&\quad+\delta\int_{\R^d}\big(h(x+z,y+z+(y-x)_\kk)-h(x,y)-\<\nn_1h(x,y), z\> \I_{\{|z|\le1\}}\label{Q10-}\\
		&\qquad\qquad\quad -\<\nn_2h(x,y),  z+(y-x)_\kk\>\I_{\{|z+(y-x)_\kk|\le1\}}\big)\nu_{t,(x-y)_\kk}(\d z)\nonumber\\
		&\quad+\int_{\R^d }\big(h(x+z,y+z)-h(x,y)-\<\nn_1h(x,y)+\nn_2h(x,y), z \>\I_{\{|z|\le1\}}\nonumber\\
		&\qquad\quad\quad  \times\big((\nu_t\wedge\bar\nu)(\d z)-\delta\nu_{t,(y-x)_\kk}(\d z)-\delta\nu_{t,(x-y)_\kk}(\d z)\big)\nonumber\\
		&\quad+\int_{\R^d }\big(h(x+z,y)-h(x,y)-\<\nn_1h(x,y), z\> \I_{\{|z|\le1\}}\big)(\nu_t-(\nu_t\wedge\bar\nu))(\d z)\nonumber\\
		&\quad+\int_{\R^d }\big(h(x ,y+z)-h(x,y)-\<\nn_2h(x,y), z\> \I_{\{|z|\le1\}}\big)(\bar\nu-(\nu_t\wedge\bar\nu))(\d z),\nonumber
\end{align}
where $\nn_1h(x,y) $ and $\nn_2h(x,y) $ stand  for the gradients     in the   variables $x$  and   $y$, respectively.

\begin{lemma}\label{lem1}
	The operator $ (\overline{\mathscr L}^{(\delta)}_{\!\! t})_{t\ge0} $ defined in \eqref{Q10-} is a coupling operator for $(\mathscr L_t)_{t\ge0}$ and $\overline{\mathscr L}$, which are the respective infinitesimal generators of $(Y_t)_{t\ge0}$ and $(\overline {Y}_{\!\! t})_{t\ge0}$.
\end{lemma}

\begin{proof}
	Notice that for any $f\in C^2_b(\R^d)$, $t\ge0 $ and $x\in\R^d,$
	\begin{align}\label{EE1}
		(\mathscr L_t f)(x)=\<\nn f(x), b(t,x)\>+\int_{\R^d}\big(f(x+z)-f(x)-\<\nn f(x),z\>\I_{\{|z|\le1\}}\big)\nu_t(\d z),
	\end{align}
	and
	\begin{align*}
		(\bar{\mathscr L} f)(x)=\<\nn f(x), \bar b(x)\>+\int_{\R^d}\big(f(x+z)-f(x)-\<\nn f(x),z\>\I_{\{|z|\le1\}}\big)\bar \nu(\d z).
	\end{align*}
	To prove the assertion, it suffices to prove that for all $f,g\in C_b^2(\R^d)$, $t\ge0$ and $ x,y\in\R^d$,
	\begin{align}\label{Q11-}
	\big(\overline{\mathscr L}^{(\delta)}_{\!\! t}h\big)(x,y)=(\mathscr L_tf)(x)+(\bar{\mathscr L}g)(y),
	\end{align}
	in which  $h(x,y):=f(x)+g(y)$.  According to the definition of $\overline{\mathscr L}^{(\delta)}_{\!\! t}$, we apparently  have that for all $t\ge0$ and $x\in\R^d,$
	\begin{align*}
		\big(\overline{\mathscr L}^{(\delta)}_{\!\! t}f\big)(x)&= \<\nn f (x ), b(t,x)\>+
		\delta\int_{\R^d }\big(f(x+z)-f(x)- \<\nn f(x), z\> \I_{\{|z|\le1\}}\big)\nu_{t,(y-x)_\kk}(\d z)\\
		&\quad+\delta\int_{\R^d }\big(f(x+z )-f(x )- \<\nn f (x ),  z\> \I_{\{|z|\le1\}}\big)\nu_{t,(x-y)_\kk}(\d z)\\
		&\quad+\int_{\R^d }\big(f(x+z )-f(x )- \<\nn f(x ), z\> \I_{\{|z|\le1\}}\big)\\
		&\qquad\qquad\times\big((\nu_t\wedge \bar\nu)(\d z)-\delta\nu_{t,(y-x)_\kk}(\d z)-\delta\nu_{t,(x-y)_\kk}(\d z)\big)\\
		&\quad +\int_{\R^d }\big(f(x+z )-f(x )-  \<\nn f (x ), z\> \I_{\{|z|\le1\}}\big)(\nu_t- (\nu_t\wedge \bar\nu))(\d z)
		\\
		&=(\mathscr L_tf)(x).
	\end{align*}
	Once more,  the definition of $\overline{\mathscr L}^{(\delta)}_{\!\! t}$  enables us to derive that for all $t\ge0$ and $
y\in\R^d,$
	\begin{align*}
		\big(\overline{\mathscr L}^{(\delta)}_{\!\! t}g\big)(y)&= \<\nn g (y),\bar b(y)\> +
		\delta\int_{\R^d}\big(g(y+z+(x-y)_\kk)-g(y)\\
		&\qquad\qquad\qquad\qquad\qquad \quad-\<\nn g (y), z+(x-y)_\kk\>\I_{\{|z+(x-y)_\kk|\le1\}}\big)\nu_{t,(y-x)_\kk}(\d z)\\
		&\quad+\delta\int_{\R^d}\big(g(y+z+(y-x)_\kk)-g(y)\\
		&\qquad\qquad \quad-\<\nn g (y), z+(y-x)_\kk\>\I_{\{|z+(y-x)_\kk|\le1\}}\big)\nu_{t,(x-y)_\kk}(\d z)\\
		&\quad+\int_{\R^d }\big(g(y+z)-g(y) -\<\nn g (y), z \>\I_{\{|z|\le1\}}\big)\\
		&\qquad\quad\quad\times\big((\nu_t\wedge\bar\nu)(\d z)-\delta\nu_{t,(y-x)_\kk}(\d z)-\delta\nu_{t,(x-y)_\kk}(\d z)\big)\\
		&\quad+\int_{\R^d }\big(g(y+z)-g(y)-\<\nn g (y), z\> \I_{\{|z|\le1\}}\big)(\bar\nu-(\nu_t\wedge\bar\nu))(\d z)
		\\
		&=\<\nn g (y), \bar b(y)\> \\
		&\quad+
		\delta\int_{\R^d}\big(g(y+z )-g(y) -\<\nn g (y), z\>  \I_{\{|z |\le1\}}\big)\big(\delta_{(x-y)_\kk}\ast\nu_{t,(y-x)_\kk}\big)(\d z)\\
		&\quad+\delta\int_{\R^d}\big(g(y+z )-g(y)  -\<\nn g (y), z\> \I_{\{|z |\le1\}}\big)\big(\delta_{(y-x)_\kk}\ast\nu_{t,(x-y)_\kk}\big)(\d z)\\
		&\quad+\int_{\R^d }\big(g(y+z)-g(y) -\<\nn g (y), z\> \I_{\{|z|\le1\}}\big)\\
		&\qquad\quad\quad\times\big((\nu_t\wedge\bar\nu)(\d z)-\delta\nu_{t,(y-x)_\kk}(\d z)-\delta\nu_{t,(x-y)_\kk}(\d z)\big)\\
		&\quad+\int_{\R^d }\big(g(y+z)-g(y)-\<\nn g (y), z\> \I_{\{|z|\le1\}}\big)(\bar\nu-(\nu_t\wedge\bar\nu))(\d z)
		\\
		&=\<\nn g (y),\bar b(y)\> +\int_{\R^d }\big(g(y+z)-g(y)-\<\nn g (y), z\> \I_{\{|z|\le1\}}\big) \bar\nu (\d z) \\
		&= (\bar{\mathscr L}g)(y),
	\end{align*}
	where in the third identity we used the fact  that $\delta_x\ast\nu_{t,-x}=\nu_{t,x}$ for all $ x\in\R^d$; see e.g. \cite[Corollary A.2]{LW19} for more details. Therefore,  \eqref{Q11-} follows, and so the proof is complete.
\end{proof}

By \cite[Theorem 9.8]{Sato},   there are mutually independent additive processes $ ( Z_t^{(1)})_{t\ge0}$ and $ (Z_t^{(2)})_{t\ge0}$
with respective time-dependent L\'{e}vy measures
$(( \nu_t\wedge \bar\nu)(\d z))_{t\ge0}$ and $(( \nu_t-(\nu_t\wedge \bar\nu))(\d z))_{t\ge0}$ such that for all $t\ge0, $
\begin{align*}
	Z_t\overset{d}= Z_t^{(1)}+ Z_t^{(2)}.
\end{align*}
Via the   L\'{e}vy–It\^o decomposition (see e.g. \cite[Theorem 19.2, p.\ 120]{Sato}), there exists a Poisson random measure $N^{(1)}
(\d z,\d t)$
with the L\'{e}vy measure $( \nu_t\wedge \bar\nu)(\d z)\,\d t$ such that for all $t\ge0,$
\begin{align*}
	Z_t^{(1)}=\int_0^t\int_{\{|z|\le1\}}z\tilde N^{(1)}
(\d z,\d s)+\int_0^t\int_{\{|z|>1\}}z  N^{(1)}
(\d z,\d s),
\end{align*}
in which $\tilde N^{(1)}(
\d z, \d t):=N^{(1)}(
\d z, \d t)-( \nu_t\wedge \bar\nu)(\d z)\,\d t$ is the   compensated Poisson measure.
Furthermore, by following  the procedure in \cite[p.\ 3140]{LW19},
$ (Z_t^{(1)})_{t\ge0}$ can be reformulated as below: for all $t\ge0,$
\begin{align*}
	Z_t^{(1)}=&\int_0^t\int_{\{|z|\le1\}\times[0,1]}z\tilde N^{(1)}(\d z,\d u,\d s)+\int_0^t\int_{\{|z|>1\}\times[0,1]}z  N^{(1)}(\d z,\d u,\d s)\\
	=&:\int_0^t\int_{\R^d \times[0,1]}z\bar N^{(1)}(\d z,\d u,\d s),
\end{align*}
where $N^{(1)}(\d z,\d u,\d t)$ is a lifted Poisson random measure on $\R^d\times[0,1]\times[0,\infty)$ with the associated L\'{e}vy measure $( \nu_t\wedge \bar\nu)(\d z)\,\d u\,\d t$, and $\tilde N^{(1)}(\d z,\d u,\d t):=N^{(1)}(\d z,\d u,\d t)-( \nu_t\wedge \bar\nu)(\d z)\,\d u\,\d t$.

In the sequel, we consider
the following SDE on $\R^{2d}:=\R^d\times\R^d$: for any $\delta\in[0,1/2]$ and $t\ge0$,
\begin{equation}\label{Q9}
	\begin{cases}
		\d  {Y_t}^{(\delta)}=b(t,Y_t^{(\delta)})\,\d t+\d   Z_t^{ (1)}+\d   Z_t^{ (2)},\\
		\d   \overline {Y}^{(\delta)}_{\!\! t}=\bar b( \overline {Y}^{(\delta)}_{\!\! t})\,\d t+\d  Z_t^{ (1)}\\
		\quad\quad\quad\quad+\displaystyle\int_{\R^d\times[0,1]}\Lambda_t\big((Y_{t-}^{(\delta)}-  \overline {Y}_{\!\! t-}^{(\delta)})_\kk,z,u\big)\,
 N^{(1)}(\d z,\d u,\d t)+\d   Z_t^{ (3)},
	\end{cases}
\end{equation}
where for any $t\ge0$, $x,z\in\R^d$ and $u\in[0,1],$
\begin{align*}
	\Lambda_t (x,z,u):=x\big(\I_{\{u\le\delta \rho_t(-x,z)\}}-\I_{\{\delta \rho_t(-x,z)<u\le \delta (\rho_t(-x,z)+\rho_t(x,z))\}}\big)~\mbox{ with } ~\rho_t(x,z):=\frac{\d \nu_{t,x}}{\d ( \nu_t\wedge \bar\nu)}(z),
\end{align*}
and $(Z_t^{(3)})_{t\ge0}$ is an  additive process  with the  L\'{e}vy measure $((\bar\nu-(\nu_t\wedge \bar\nu))(\d z))_{t\ge0}$
that is independent of $(Z_t^{(1)})_{t\ge0}$ and $(Z_t^{(2)})_{t\ge0}$.
By following  the argument used in the proof of \cite[Proposition 2.2]{LW19},
we can obtain that
the SDE \eqref{Q9}
has a unique strong solution $(Y_t^{(\delta)} ,\overline{Y}^{(\delta)}_{\!\! t} )_{t\ge0}$.

\begin{lemma}\label{lem4}
	The infinitesimal generator of $(Y_t^{(\delta)},  \overline {Y }^{(\delta)}_{\!\! t})_{t\ge0}$ solving \eqref{Q9} is
	$ (\overline{\mathscr L}^{(\delta)}_{\!\! t})_{t\ge0} $ defined in \eqref{Q10-}. In particular, $(Y_t^{(\delta)},  \overline {Y}^{(\delta)}_{\!\! t})_{t\ge0}$ is a Markov coupling process on $\R^{2d}$ of the processes $(Y_t)_{t\ge0}$ and $(\overline {Y}_{\!\! t})_{t\ge0}$ given by \eqref{E1-} and \eqref{E2-}, respectively.
\end{lemma}

\begin{proof}
	Below, we write $U_t^{(\delta)}  =Y_t^{(\delta)} -\overline {Y}^{(\delta)}_{\!\! t} $ for convenience, and denote by $(\tilde{\mathscr L}_t^{(\delta)} )_{t\ge0}$
 	 the infinitesimal generator of $(Y_t^{(\delta)} ,\overline {Y}^{(\delta)}_{\!\! t} )_{t\ge0}$ which is governed by   \eqref{Q9}.
	By imitating the procedure to derive \cite[(2.15)]{LW19}, we find that for all $t\ge0$,
	\begin{align}\label{E3-}
		\d & Z_t^{ (1)}+\int_{\R^d\times[0,1]}\Lambda_t((U_{t-}^{(\delta)})_{\kappa},z,u)\,N^{(1)}(\d z,\d u,\d t)\nonumber\\
		&=\int_{\R^d\times[0,1]}\Big[(z+(U_{t-}^{(\delta)})_\kk)\I_{\{u\le \delta \rho_t(-(U_{t-}^{(\delta)})_\kk,z)\}}\nonumber\\
		&\qquad\qquad\quad+(z-(U_{t-}^{(\delta)})_\kk)\I_{\{\delta\rho_t(-(U_{t-}^{(\delta)})_\kk,z)<u\le \delta (\rho_t(-(U_{t-}^{(\delta)})_\kk,z)+\rho_t((U_{t-}^{(\delta)})_\kk,z))\}}\\
		&\qquad\qquad\quad+z\I_{\{ \delta (\rho_t(-(U_{t-}^{(\delta)})_\kk,z)+\rho_t((U_{t-}^{(\delta)})_\kk,z))<u\le1\}}\Big]\,\bar N^{ (1)}(\d z,\d u,\d t)\nonumber\\
		&\quad-\int_{\R^d\times[0,1]}\Big[(z+(U_{t-}^{(\delta)})_\kk)\Big(\I_{\{|z +(U_{t-}^{(\delta)})_\kk|\le1\}}-\I_{\{|z|\le1\}}\Big)\I_{\{u\le \delta \rho_t(-(U_{t-}^\delta)_\kk,z)\}}\nonumber\\
		&\qquad\qquad\qquad+(z-(U_{t-}^{(\delta)})_\kk)\Big(\I_{\{|z +(U_{t-}^{(\delta)})_\kk|\le1\}}-\I_{\{|z|\le1\}}\Big)\nonumber\\
		&\qquad\qquad\qquad\quad\times\I_{\{ \delta\rho_t(-(U_{t-}^{(\delta)})_\kk,z)<u\le\delta (\rho_t(-(U_{t-}^{(\delta)})_\kk,z)+
			\rho_t((U_{t-}^{(\delta)})_\kk,z))\}}\Big]\,\nu_t(\d z)\,\d u\,\d t. \nonumber
	\end{align}
	From \eqref{Q9} and \eqref{E3-}, we deduce that for any $h\in C_b^2(\R^{2d})$,  $t\ge0$ and
	$x,y\in\R^d,$
	\begin{align*}
		\big(
\tilde{\mathscr L}^{(\delta)}_{\!\! t}h\big)(x,y)&=\<\nn_1h(x,y),b(t,x)\>+\<\nn_2h(x,y),\bar b(y)\>\\
		&\quad+
		\delta\int_{\R^d }\Big(h(x+z,y+z+(x-y)_\kk)-h(x,y)\\
		&\quad \qquad\qquad-\<\nn_1h(x,y),z\>\I_{\{|z|\le1\}}-\<\nn_2h(x,y),z+(x-y)_\kk\>\I_{\{|z|\le1\}}\Big)\nu_{t,(y-x)_\kk}(\d z)\\
		&\quad +\delta\int_{\R^d }\Big(h(x+z,y+z+(y-x)_\kk)-h(x,y)-\<\nn_1h(x,y),z\>\I_{\{|z|\le1\}}\\
		&\quad \qquad\qquad-\<\nn_2h(x,y),z+(y-x)_\kk\>\I_{\{|z|\le1\}}\Big)\nu_{t,(x-y)_\kk}(\d z)\\
		&\quad+ \int_{\R^d }\Big(h(x+z,y+z )-h(x,y)-\<\nn_1h(x,y),z\>\I_{\{|z|\le1\}}\\
		&\qquad\qquad\,\,-\<\nn_2h(x,y),z \>\I_{\{|z|\le1\}}\Big)\Big((\nu_t\wedge\bar\nu)-\delta\nu_{t,(x-y)_\kk}-\delta\nu_{t,(y-x)_\kk}\Big)(\d z)
		\\
		&\quad-\delta
		\int_{\R^d}\<\nn_2h(x,y), z+(x-y)_\kk\>(\I_{\{|z +(x-y)_\kk|\le1\}}-\I_{\{|z|\le1\}})\nu_{t,(y-x)_\kk}(\d z)\\
		&\quad-\delta
		\int_{\R^d}\<\nn_2h(x,y), z+(y-x)_\kk\>(\I_{\{|z +(y-x)_\kk|\le1\}}-\I_{\{|z|\le1\}})\nu_{t,(x-y)_\kk}(\d z)\\
		&\quad+\int_{\R^d}\big(h(x+z,y)-h(x,y)-\<\nn_1h(x,y),z\>\I_{\{|z|\le1\}}\big)( \nu_t-(\nu_t\wedge \bar\nu))(\d z)\\
		&\quad+\int_{\R^d}\big(h(x ,y+z)-h(x,y)-\<\nn_2h(x,y),z\>\I_{\{|z|\le1\}}\big)( \bar\nu-(\nu_t\wedge \bar\nu))(\d z)\\
		&=\big( \overline{\mathscr L}^{(\delta)}_{\!\! t}h\big)(x,y).
	\end{align*}
	Therefore, we conclude that the infinitesimal generator of $(Y_t^{(\delta)} ,  \overline {Y}^{(\delta)}_{\!\! t} )_{t\ge0}$   is
	$ (\overline{\mathscr L}^{(\delta)}_{\!\! t} )_{t\ge0} $.
The second assertion is a direct consequence of the first one, and the proof is complete.
\end{proof}

\subsection{Uniform-in-time estimates: the uniformly dissipative-type case}
In this subsection, we assume that the drift term $\bar b(x)$ in the SDE \eqref{E2-}
satisfies the uniformly dissipative-type condition.
The main result in this subsection is presented as follows.
\begin{theorem}\label{thm1} Assume that the following assumptions hold:
	\begin{itemize}
		\item[{\rm(i)}]
		There exist constants   $\lambda>0$ and $\kk\ge1$ such that for all $x,y\in\R^d,$
		\begin{align*}
			\<x-y,\bar b(x)-\bar b(y)\>\le -\lambda  |x-y|^{1+\kk }.
		\end{align*}
		\item[{\rm(ii)}]  There exists  a function  $\mathcal V: [0,\8)\times\R^d\to[0,\8)$  satisfying $\lim_{t\to \infty} \mathcal V(t,x)=0$ for all $x\in \R^d$
		so that for all $t\ge0$ and $x\in \R^d$,
		\begin{align}\label{R2}
			|b(t,x)-\bar b(x)|\le  \mathcal V(t,x).
		\end{align}
		\item[{\rm(iii)}] There exist constants   $\theta_1\in (1 ,2]$ and $\theta_2\in (0,\theta_1)$  such that for all $t\ge0$,
		\begin{equation}\label{EE6}
			\begin{split}
			H(t):=&
\int_{\{|z|\le 1\}} |z|^{\theta_1} |\nu_t-\bar\nu|(\d z) \\
&+\left(\int_{\{|z|>1\}} |z|^{\theta_2}|\nu_t-  \bar \nu|(\d z)\right)\I_{\{\theta_2\le 1\}} +\left(\int_{\{|z|>1\}} |z||\nu_t-  \bar \nu|(\d z)\right)^{\theta_2}\I_{\{\theta_2>1\}} <\infty,
			\end{split}
		\end{equation}
		where $|\nu_t-  \bar \nu|$ means the total variation of the signed measure $\nu_t-  \bar \nu$.
	\end{itemize}
	Then,   there exists a constant $C_0>0$ such that for all $t>0$ and   $\mu_1,\mu_2\in\mathcal P_{\theta_2}(\R^d)$,
	\begin{equation}\label{R4}
		\begin{split}
			\mathcal W_{|\cdot|^{\theta_1}\wedge|\cdot|^{\theta_2}}\big(\mathscr L_{Y_t^{\mu_1}},\mathscr L_{\overline Y_{\!\! t}^{\mu_2}}\big) &\le 	C_0
			\begin{cases}
		 \e^{- \lambda_*t}\mathcal W_{|\cdot|^{\theta_1}\wedge|\cdot|^{\theta_2}}(\mu_1,\mu_2)  + \int_0^t\e^{-{  \lambda_*(t-s)}}g_*(s)\,\d s ,&\kk=1,\\
		  t^{-\theta_1/{(\kk-1)}}\vee \sup_{r\ge t/2}(g_*(r))^{1/(1+(\kk-1)/{\theta_1})} ,&\kk>1,
			\end{cases}
		\end{split}
	\end{equation}
where 	$\lambda_*:= 2^{\theta_2/{\theta_1}-4}
	 \theta_2\lambda
$  and    $  g_*(t):=H(t)+   \E\mathcal V(t, Y_t^{\mu_1} )^{1\vee\theta_2}  $.
\end{theorem}

Before proving Theorem \ref{thm1}, it is indispensable to establish  the following   lemma.

\begin{lemma}\label{lem0}
	Let $\theta\in(1,2]$ and
	$h_\theta(x)=|x|^{\theta-2}x$
	for all
	$x\in\R^d.$ Then, for any $x,y\in\R^d,$
	\begin{align}\label{WW2}
		|h_\theta(x)-h_\theta(y)|\le(2^{\theta-1}+3^{\theta-1})|x-y|^{\theta-1} .
	\end{align}
	That is, $\R^d\ni x\mapsto h_\theta(x)$ is  H\"older continuous with the H\"older index $\theta-1.$
\end{lemma}
\begin{proof}
The proof is split into two cases.
	
	Case (i): {\it $x,y\in\R^d$ with $|y|\ge 2|x-y|$.} For this case, we obviously have for any $s\in[0,1],$
	$$
	|y+s(x-y)|\ge|y|-|x-y|\ge|x-y|.
	$$
	Note   that
	\begin{align*}
		h_\theta(x)-h_\theta(y)&=\int_0^1(\nn_{x-y}h_\theta)(y+s(x-y))\,\d s\\
		&=\int_0^1\big[|y+s(x-y)|^{\theta-2}I_d\\
		&\quad\quad\quad+(\theta-2)|y+s(x-y)|^{\theta-4} (y+s(x-y))\otimes(y+s(x-y))\big](x-y) \,\d s.
	\end{align*}
	Whence, we obtain from  $\theta\in(1,2]$ that
	\begin{align*}
		|h_\theta(x)-h_\theta(y)|
		\le(3-\theta)\int_0^1 |y+s(x-y)|^{\theta-2}|x-y|\,\d s \le (3-\theta)|x-y|^{\theta-1}.
	\end{align*}
	
	Case (ii): {\it $x,y\in\R^d$ with $|y|\le 2|x-y|$.} In this case, via the triangle inequality, it follows that
	\begin{align*}
		|x|\le|y|+|x-y|\le 3|x-y|.
	\end{align*}
	Thus, by means of the triangle inequality once more, we find from $\theta\in(1,2]$ that
	\begin{align*}
		|h_\theta(x)-h_\theta(y)|\le|h_\theta(x)|+|h_\theta(y)|
		\le |x|^{\theta-1}+|y|^{\theta-1}\le (2^{\theta-1}+3^{\theta-1})|x-y|^{\theta-1}.
	\end{align*}

	Therefore, combining both estimates above yields the assertion \eqref{WW2} as long as $\theta\in(1,2]$. 	
\end{proof}

Below, we establish a Gronwall-type inequality which plays an indispensable role in
studying the long-term behavior of SDEs with super-dissipativity (i.e.,
$\kappa>1$ in Theorem \ref{thm1} (i)).

\begin{lemma}\label{lemma-*}
Let    $f:[0,\infty)\to[0,\infty)$ be absolutely continuous, and $g_*:[0,\infty)\to[0,\infty)$
satisfy the following differential inequality: for some constants $c>0$ and $\theta\ge1,$
\begin{align}\label{WP-1}
  f'(t)\le -cf(t)^\theta+g_*(t) ,\quad \mbox{ a.e. } \, t>0.
\end{align}
Then,
 for all $t>0,$
\begin{align}\label{WR-1}
f(t)\le
\begin{cases}
 \e^{-ct}f(0)+\int_0^t\e^{-c(t-s)}g_*(s)\,\d s,&\theta=1,\\
 c_0\big(t^{-1/{(\theta-1)}}\vee \sup_{r\ge t/2} (g_*(r))^{1/\theta}\big),&\theta>1,
 \end{cases}
\end{align} where $$c_0:=((\theta-1)/4)^{-1/(\theta-1)}\vee (2/c)^{1/\theta}.$$
\end{lemma}

\begin{proof}
In case of $\theta=1,$	the assertion \eqref{WR-1} follows from the classical Gronwall inequality. In the sequel, we focus on proving  \eqref{WR-1} for the case $\theta>1.$
		
In the following analysis, we  fix  $s\ge0 $  and set
\begin{align*}
r_s:=(2G(s)/c)^{1/\theta}\quad \mbox{ with } \quad G(s):=\sup_{r\ge s}g_*(r).
\end{align*}
First of all, we show the subsequent statement ({\bf A}): $f(t)\le r_s$ for all $t\ge t_0$ provided that
there exists a constant $t_0\ge s$ such that  $f(t_0)\le r_s$.
To this end, we adopt
a contradiction  argument.
We assume that there exists a constant $t_1>t_0$ such that $f(t_1)>r_s$, and subsequently  define
\begin{align*}
\tau_1=\sup\{t\in[t_0,t_1]: f(t)=r_s\}.
\end{align*}
Thus, the absolute continuity of $f(t)$ (which definitely implies that  $[0,\infty)\ni t\mapsto f(t)$ is continuous) enables us to
conclude  that $f(\tau_1)=r_s$ and $f(t)>r_s$ for all $t\in(\tau_1,t_1]$. This yields
that for all $t\in(\tau_1,t_1]$,
\begin{align*}
f'(t)\le -cf(t)^\theta+g_*(t)<-cr_s^\theta+g_*(t)=-2G(s)+g_*(t)\le0.
\end{align*}
Whence, $(
\tau_1,t_1]\ni t\mapsto f(t)$ is strictly decreasing almost everywhere, which contradicts
the fact that $f(t)>r_s$ for all $t\in(
\tau_1,t_1]$. Therefore, the statement $({\bf A})$ established above holds.

On the one hand, we assume the initial value $f(s)\le r_s$. Define
\begin{align*}
\tau_2=\inf\{t\ge s: f(t)=r_s\}.
\end{align*}
Obviously, we have $f(t)\le r_s$ for all $t\in[s,\tau_2]$ according to the definition of $\tau_2$ and also $f(t)\le r_s$ for all $t>\tau_2 $ by taking the  statement $({\bf A})$ into consideration. So, if $f(s)\le r_s$, we infer that $f(t)\le r_s$ for all $t\ge s.$

On the other hand, we assume that the initial value $f(s)>r_s$.  In this case, we have
$f(t)>r_s$ for all $t\in[s,\tau_2)$, so $cf(t)^\theta/2> G(s) $ for all $t\in[s,\tau_2)$. Whereafter,
we derive from \eqref{WP-1} and the continuity of $f(t)$ that for all $t\in[s,\tau_2)$,
\begin{align*}
f'(t)\le -cf(t)^\theta+G(s)\le-\frac{1}{2}cf(t)^\theta,
\end{align*}
which indeed is a Bernoulli-type differential inequality. As a result, we arrive at
\begin{align*}
f(t)\le \big(f(s)^{1-\theta}+(\theta-1)(t-s)/2\big)^{-1/{(\theta-1)}},\quad \forall\, t \in[s,\tau_2).
\end{align*}
This, together with the statement $({\bf A})$, leads to the subsequent estimate: for any $t\ge s,$
\begin{align*}
f(t)&=f(t)\I_{\{t\in[s,\tau_2)\}}+f(t)\I_{\{t\ge \tau_2 \}}\\
&\le \big(f(s)^{1-\theta}+(\theta-1)(t-s)/2\big)^{-1/{(\theta-1)}}\I_{\{t\in[s,\tau_2)\}}+r_s\I_{\{t\ge \tau_2 \}}\\
&\le  \big(f(s)^{1-\theta}+(\theta-1)(t-s)/2\big)^{-1/{(\theta-1)}}\vee r_s.
\end{align*}

Based on the preceding analysis, by taking $s=t/2,$ we deduce from $\theta>1$ that for all $t>0,$
\begin{align*}
	f(t)
 \le  \big( (\theta-1) t /4\big)^{-1/{(\theta-1)}}\vee \Big(\frac{2}{c}  \sup_{r\ge t/2}g_*(r)\Big)^{1/\theta}.
\end{align*}
Accordingly, the assertion \eqref{WR-1} follows immediately for the case $\theta>1.$
\end{proof}

With the aid of  Lemma \ref{lem0} and Lemma \ref{lemma-*},  we proceed to accomplish the

\begin{proof}[Proof of Theorem $\ref{thm1}$]
	Since the drift $\bar b(x)$
satisfies the uniformly dissipative-type condition, we apply the synchronous coupling approach to carry out the proof of Theorem \ref{thm1}, i.e., we take the parameter  $\delta=0$ involved in  \eqref{Q10-}. Accordingly, the associated coupling   is based on
	the following relationship:
	for all $t\ge0$ and $x,y\in\R^d,$
	\begin{equation*}
		(x,y)\rightarrow
		\begin{cases}
			(x+z,y+z),\quad\quad&(\nu_t\wedge\bar\nu)(\d z),\\
			(x+z,y),\quad\quad&(\nu_t-(\nu_t\wedge\bar\nu))(\d z),\\
			(x,y+z), \quad\quad& (\bar\nu-(\nu_t\wedge\bar\nu))(\d z).
		\end{cases}
	\end{equation*}
	
	By virtue of Lemma \ref{lem1} and Lemma \ref{lem4}, $(Y_t^{(0)},\overline{Y}_{\!\! t}^{(0)})_{t\ge0}$  solving
	\eqref{Q9}
	is a coupling process of
$(Y_t)_{t\ge0}$ and $(\overline{Y}_{\!\! t})_{t\ge0}$
defined by \eqref{E1-} and \eqref{E2-}, respectively.
 Below, we shall write $(Y_t ,\overline{Y}_{\!\! t})_{t\ge0}$ instead of $(Y_t^{(0)},\overline{Y}_{\!\! t}^{(0)})_{t\ge0}$ for notational simplicity.
	For $c_1:=  2^{1- \theta_2/{\theta_1 }} \theta_1 /{\theta_2}$,
	let
	\begin{equation}\label{EW11}
		h(r)=
		\begin{cases}
			r,& 0\le r\le 1,\\
			1+c_1\big((1+r)^{ \theta_2/{ \theta_1}  }-2^{ \theta_2/{ \theta_1}  }\big),& r>1.
		\end{cases}
	\end{equation}
	According to the expression  of $c_1$, it is easy to see that $[0,\8)\ni r\mapsto h (r)$
	is a $C^1$-function.  For  $U_t:=Y_t-\overline{Y}_{\!\! t}$,
	we obviously obtain from \eqref{Q9} with $\delta=0$ that for all $t\ge0,$
	\begin{align*}
		\d U_t =\big(b(t,Y_t )-\bar b(\overline{Y}_{\!\! t} )\big)\,\d t +\d   Z_t^{  (2)}-\d   Z_t^{(3)}.
	\end{align*}
	Thus,  applying  It\^o's formula yields that for all $t\ge0,$
	\begin{align*}
		\d h (|U_t |^{\theta_1}) &=\d M_t+
\theta_1
h'(|U_t |^{\theta_1})|U_t |^{ \theta_1-2}  \<U_t , b(t,Y_t )-\bar b(\overline{Y}_{\!\! t} )\> \,\d t\\
		&\quad+ \int_{\R^d}\big(h ( |U_t +z|^{\theta_1}) - h (|U_t |^{\theta_1}) -
\theta_1
h'(|U_t |^{\theta_1}) |U_t |^{ \theta_1-2} \<U_t ,z\>\I_{\{|z|\le1\}}\big)\\
&\quad\quad\quad\quad\times ( \nu_t-(\nu_t\wedge \bar\nu))(\d z)\,
		\d t
		\\
		&\quad + \int_{\R^d}\big( h ( |U_t -z|^{\theta_1}) - h ( |U_t |^{\theta_1}) -
\theta_1
h'( |U_t |^{\theta_1}) |U_t |^{ \theta_1-2} \<U_t ,z\>\I_{\{|z|\le1\}}\big)\\
&\quad\quad\quad\quad\times  (\bar\nu- (\nu_t\wedge\bar\nu))(\d z)\,
		\d t
		\\
		&=:\d M_t+\big(I_1(t)+I_2(t) +I_3(t) \big)\,\d t,
	\end{align*}
	where $(M_t )_{t\ge0}$ is a martingale.

	Below, we aim at estimating the terms $I_1(t)$, $I_2(t)$ and $I_3(t)$, separately. From (i) and \eqref{R2},
	it is ready to see that
	for any $x,y\in\R^d$ and $t\ge0,$
	\begin{align*}
		|x-y|^{ \theta_1-2} \<x-y,b (t,x)-\bar b(y)\>
		&\le  -\lambda  |x-y|^{\theta_1+\kk-1}+  |x-y|^{\theta_1-1}\mathcal V(t, x).
	\end{align*}
	Next, the definition of $h $ enables us to derive that
	\begin{equation}\label{EW--}
\begin{split}
		h'(r^{\theta_1})r^{\theta_1+\kk-1}=&
		\begin{cases}
			h (r^{\theta_1})^{1+(\kk-1)/{\theta_1}},&  0\le r\le 1, \\
			c_1\theta_2(1+r^{\theta_1})^{ \theta_2/{\theta_1}-1}r^{\theta_1 +\kk-1}/{\theta_1},& r\ge1
		\end{cases}\\
\ge &\begin{cases}h (r^{\theta_1})^{1+(\kk-1)/{\theta_1}},&  0\le r\le 1, \\
 2^{ \theta_2/{\theta_1}-1}
 (2\theta_1/{\theta_2})^{-1-(\kappa-1)/{\theta_2}}	h (r^{\theta_1})^{1+(\kk-1)/{\theta_2}}, &r>1,
\end{cases}\\
\ge&
2^{ \theta_2/{\theta_1}-1}
(\theta_2/{(2\theta_1)})^{1+(\kk-1)/{\theta_2}}h (r^{\theta_1})^{1+(\kk-1)/{\theta_1}},
\end{split}
\end{equation}
	where in the
first
inequality we employed the fact that for all $r\ge1,$
	\begin{align*}
		\frac{(1+r^{\theta_1})^{ \theta_2/{\theta_1}-1}r^{\theta_1+\kk-1 }}{ (1+c_1 ((1+r^{\theta_1})^{\theta_2/{\theta_1}}-2^{\theta_2/{\theta_1}} ))^{1+(\kk-1)/{\theta_2}} }  \ge  2^{ \theta_2/{\theta_1}-1} (2\theta_1/{\theta_2})^{-1-(\kappa-1)/{\theta_2}},
	\end{align*}
and the second inequality is valid by taking advantage  of $\theta_2\in(0,\theta_1)$ and
$h(r^{\theta_1})\ge1$ for all $r\ge1$. 	
Furthermore,  it follows from $\theta_2/{\theta_1}\in(0,1)$ and $\theta_1>1$ that
 for all $r\ge0,$
	\begin{align*}
		h'(r^{\theta_1})r^{\theta_1-1}&=
		\begin{cases}
			r^{\theta_1-1},\quad &r\in[0,1],\\
			2^{1-\theta_2/{\theta_1}}(1+r^{\theta_1})^{ \theta_2/{\theta_1}-1}r^{\theta_1-1},\quad &r>1,
		\end{cases}\\
		&\le \begin{cases}
			r^{\theta_1-1},\quad &r\in[0,1],\\
2^{1-\theta_2/{\theta_1}}r^{\theta_2-1},\quad &r>1,
		\end{cases}\\
&\le \begin{cases}
			2^{1-\theta_2/{\theta_1}},\quad &\theta_2\in (0,1],\\
2^{1-\theta_2/{\theta_1}}r^{\theta_2-1},\quad &\theta_2>1, r>1.
		\end{cases}
	\end{align*}
	Consequently,
	$I_1(t)$ can be dominated  as follows: for some constant $c_2>0 $ and all $t\ge0,$
	\begin{align*}
		I_1(t) &\le   - \lambda_* h (|U_t|^{\theta_1})^{1+(\kk-1)/{\theta_1}} +
\theta_1
2^{1-\theta_2/{\theta_1}}\mathcal V(t, Y_t)\I_{\{\theta_2\in(0,1]\}}\\
&\quad+
\theta_1
2^{1-\theta_2/{\theta_1}}|U_t|^{\theta_2-1}\mathcal V(t, Y_t)\I_{\{|U_t|>1,\theta_2>1\}}\\
		&\le -\frac{3}{4}\lambda_*h (|U_t|^{\theta_1})^{1+(\kk-1)/{\theta_1}} +
c_2\mathcal V(t, Y_t)^{1\vee\theta_2},
	\end{align*}
	where $\lambda_*:=
	2^{ \theta_2/{\theta_1}-1}
	\theta_1\lambda(\theta_2/{(2\theta_1)})^{1+(\kk-1)/{\theta_2}}$.
In the previous estimate, to obtain the second inequality,
we used   the subsequent  fact that
for some constant $c_3>0 $ and all $t\ge0,$
\begin{equation}\label{e:AD}
\begin{split}	
 &|U_t|^{\theta_2-1}\mathcal V(t, Y_t)\I_{\{|U_t|>1,\theta_2>1\}}\\ &\le \Big( \frac{1}{ \theta_1}2^{\theta_2/{\theta_1}-3}\lambda_*|U_t|^{\theta_2 }+c_3\mathcal V(t, Y_t)^{\theta_2}\Big)\I_{\{|U_t|>1,\theta_2>1\}}  \\
	&\le \Big( \frac{1}{ \theta_1}2^{\theta_2/{\theta_1}-3}\lambda_*h (|U_t|^{\theta_1})^{1+(\kk-1)/{\theta_1}} +c_3\mathcal V(t, Y_t)^{\theta_2}\Big)\I_{\{|U_t|>1,\theta_2>1\}},
	\end{split}
\end{equation}
where the first inequality follows from Young's inequality and the second inequality
holds true by noting that $r^{\theta_2}/h(r^{\theta_1})\le 1$ for all $r\ge1.$

	By virtue of $\theta_1\in(1,2]$ and $\theta_2\in(0,\theta_1)$,
	it follows that $[0,\8)\ni r\mapsto h'(r)$ is decreasing  so that for any $a,b\ge0,$
	\begin{align*}
		h (b)-h (a)\le h' (a)(b-a).
	\end{align*}
	On the other hand,  it holds from $1-c_12^{\theta_2/{\theta_1}}=1-2\theta_1/{\theta_2}<0$
that for some constant $c_4>0 $ and all  $x,z\in \R^d$ with $|z|\ge1$,
	\begin{align*}h ( |x+z|^{\theta_1}) - h (|x|^{\theta_1})&\le h ( (|x|+|z|)^{\theta_1}) - h (|x|^{\theta_1})\\
		&\le \begin{cases}h ( (1+|z|)^{\theta_1}),&\quad |x|\le 1,\\
			c_1\left((1+(|x|+|z|)^{\theta_1})^{{\theta_2}/{\theta_1}}-(1+|x|^{\theta_1})^{{\theta_2}/{\theta_1}}\right),&\quad |x|>1,\end{cases}\\
		&\le \begin{cases} c_1(1+ (1+|z|)^{\theta_1})^{ {\theta_2}/{\theta_1}},&\quad |x|\le 1,\\
			c_1(\theta_1(|x|+|z|)^{\theta_1-1}|z|)^{ {\theta_2}/{\theta_1}},&\quad 1\le|x|\le|z|,\\
			c_1\theta_2(1+|x|^{\theta_1})^{{\theta_2}/{\theta_1}-1}(|x|+|z|)^{\theta_1-1}|z|,&\quad |x|>|z|,\end{cases}\\
		&\le  \begin{cases}  c_4|z|^{\theta_2},&\quad |x|\le |z|,\\
			c_4|x|^{\theta_2-1} |z|,&\quad |x|\ge |z|,\end{cases}\\
&\le \begin{cases}  c_4|z|^{\theta_2},&\quad \theta_2\le 1,\\
			c_4|x|^{\theta_2-1} |z|,&\quad \theta_2>1,|x|\ge |z|,\end{cases}
	\end{align*}
	where in the third inequality we   used  the basic inequality:  $(a+b)^\kappa\le a^\kappa+b^\kappa $ for all $a,b\ge0$ and $\kappa\in (0,1)$, the fact that
	\begin{align}\label{J-}
		(|x|+|z|)^{\theta_1}-  |x|^{\theta_1}\le \theta_1(|x|+|z|)^{\theta_1-1}|z|
	\end{align}
	by applying  the mean value theorem and taking $\theta_1\in(1,2]$ into consideration, as well as that
	\begin{align*}
		(1+(|x|+|z|)^{\theta_1})^{ {\theta_2}/{\theta_1}}-(1+|x|^{\theta_1})^{ {\theta_2}/{\theta_1}}&\le \frac{\theta_2}{\theta_1}(1+|x|^{\theta_1})^{{\theta_2}/{\theta_1}-1}\big((|x|+|z|)^{\theta_1}-|x|)^{\theta_1}\big)\\
		&\le \theta_2(1+|x|^{\theta_1})^{ {\theta_2}/{\theta_1}-1}(|x|+|z|)^{\theta_1-1}|z|
	\end{align*}
	by invoking $\theta_2/{\theta_1}\in(0,1)$  and \eqref{J-}.
	Subsequently, by making use of the aforementioned   facts,
	we derive from Lemma \ref{lem0}  that for some constants  $c_5,c_6 >0 $ and all $t\ge0,$
	\begin{align*}
		I_2(t)&=  \int_{\{|z|\le1\}}\Big(h ( |U_t+z|^{\theta_1}) - h (|U_t|^{\theta_1}) -
\theta_1
h'(|U_t|^{\theta_1})|U_t|^{ \theta_1-2} \<U_t,z\> \Big)\,(\nu_t-(\nu_t\wedge \bar\nu))(\d z)\nonumber\\
		&\quad+  \int_{\{|z|>1\}}\Big(h ( |U_t+z|^{\theta_1}) - h (|U_t|^{\theta_1})\Big)\,(\nu_t-(\nu_t\wedge \bar\nu))(\d z)\nonumber\\
		&\le  h'(|U_t|^{\theta_1})\int_{\{|z|\le1\}}\big(  |U_t+z|^{\theta_1}  -
 |U_t|^{\theta_1}  -
 \theta_1
 |U_t|^{ \theta_1-2 }\<U_t,z\>\big)\,(\nu_t-(\nu_t\wedge \bar\nu))(\d z)\nonumber\\
		&\quad+ c_4
		\I_{\{\theta_2\in(0,1]\}}
		\int_{\{|z|>1\}} |z|^{\theta_2}(\nu_t-(\nu_t\wedge \bar \nu))(\d z)\\
		&\quad+ c_4
\I_{\{|U_t|\ge1,\theta_2>1\}} |U_t|^{\theta_2-1}\int_{\{1<|z|\le |U_t|\}}|z|\,(\nu_t-(\nu_t\wedge \bar\nu))(\d z)\\
		&\le
\theta_1
h'(|U_t|^{\theta_1})\int_{\{|z|\le 1\}}\int_0^1\big( |U_t+sz|^{ \theta_1-2} \<U_t+sz,z\> -|U_t|^{ \theta_1-2} \<U_t,z\> \big)\,\d s\,(\nu_t-(\nu_t\wedge \bar\nu))(\d z)\nonumber\\
		&\quad+  c_5
			\I_{\{\theta_2\in(0,1]\}}
		\int_{\{|z|>1\}} |z|^{\theta_2}(\nu_t-(\nu_t\wedge \bar\nu))(\d z) \\
		&\quad+c_5\left(\int_{\{|z|>1\}} |z|(\nu_t-(\nu_t\wedge \bar\nu))(\d z)\right)^{\theta_2}\I_{\{\theta_2>1\}}+ \frac{1}{4}\lambda_*  h (|U_t|^{\theta_1})^{1+(\kk-1)/{\theta_1}}\nonumber\\
		&\le  c_6 \int_{\R^d}\big(|z|^{\theta_1}\I_{\{|z|\le1\}}+|z|^{\theta_2}\I_{\{|z|>1\}}\I_{\{\theta_2\in(0,1]\}}\big)
		 (\nu_t-(\nu_t\wedge \bar\nu))(\d z) \\
		&\quad+ c_6\left(\int_{\{|z|>1\}} |z|(\nu_t-(\nu_t\wedge \bar\nu))(\d z)\right)^{\theta_2}\I_{\{\theta_2>1\}} +\frac{1}{4}\lambda_*  h (|U_t|^{\theta_1})^{1+(\kk-1)/{\theta_1}},\nonumber
	\end{align*} where the second
	inequality is available by following the strategy to derive
\eqref{e:AD}, and
	the last inequality is valid by invoking  $h'(r)\le1$.
	Next, by repeating the preceding procedure to estimate $I_2(t)$, we also have that for some constant $c_7>0 $ and all $t\ge0,$
	\begin{align*}
		I_3(t)\le & c_7 \int_{\R^d} \big(|z|^{\theta_1}\I_{\{|z|\le1\}}+|z|^{\theta_2}\I_{\{|z|>1\}}\I_{\{\theta_2\in(0,1]\}}\big)
		 (\bar\nu -(\nu_t\wedge \bar\nu))(\d z)
		 \\
		&+ c_7\left(\int_{\{|z|>1\}} |z|(\bar\nu-(\nu_t\wedge \bar\nu))(\d z)\right)^{\theta_2}\I_{\{\theta_2>1\}}+\frac{1}{4}\lambda_* h (|U_t|^{\theta_1})^{1+(\kk-1)/{\theta_1}}.
	\end{align*}

	Using the preceding analysis together with the fact
  that $|\nu_t-\bar \nu|=(\nu_t-(\nu_t\wedge \bar \nu))+(\bar \nu-(\nu_t\wedge \bar \nu))$,
	we deduce from Jensen's inequality that for some constant $c_8>0 $ and all $t\ge0,$
	\begin{align*}
		\d \E h ( |U_t |^{\theta_1})
		&\le   \Big[ -\frac{1}{4}\lambda_*  \E h (|U_t|^{\theta_1})^{1+(\kk-1)/{\theta_1}} +c_8g_*(t)
		\Big]\,\d t\\
		&\le \Big[ -\frac{1}{4}\lambda_*  \big(\E h (|U_t|^{\theta_1})\big)^{1+(\kk-1)/{\theta_1}} +c_8g_*(t)
		\Big]\,\d t.
	\end{align*}
	Whereafter,  an application of Lemma \ref{lemma-*} yields  that for all $t\ge0,$
	\begin{align*}
		\E h ( |U_t |^{\theta_1})\le
		\begin{cases}\e^{-\frac{1}{4}\lambda_*t}\E h ( |U_0 |^{\theta_1})+c_9\int_0^t\e^{-\frac{1}{4}\lambda_*(t-s)}g_*(s)\,\d s,&\kk=1,\\
		 c_9\big(t^{-\theta_1/{(\kk-1)}}\vee \sup_{r\ge t/2}g_*(r)^{1/(1+(\kk-1)/{\theta_1})}\big),&\kk>1.
		 \end{cases}
	\end{align*}
	Then,   \eqref{R4} follows by noticing that there exist constants $c_{10},c_{11}>0$ such that for all $r\ge0,$
\begin{align}\label{EW-*}
 c_{10}( r \wedge r^{\theta_2/{\theta_1}})\le h (r )\le c_{11}( r \wedge r^{\theta_2/{\theta_1}})
 \end{align}
	and taking  $(Y_0,\overline Y_{\!\!0})$ such that $\E  (|Y_0-\overline {Y}_{\!\! 0}|^{\theta_1} \wedge|Y_0-\overline {Y}_{\!\! 0}|^{\theta_2}) =\mathcal W_{|\cdot|^{\theta_1}\wedge|\cdot|^{\theta_2}}(\mu_1,\mu_2)$.
\end{proof}

\begin{remark}
	In Theorem \ref{thm1}, $ \bar b(x)$ is assumed  to be uniformly dissipative
(i.e., $\theta=1$ in Assumption {\rm(i)}). For this setting,
	it is quite natural to choose the Euclidean distance $|x-y|$ (which induces the $L^1$-Wasserstein distance); that is, one may take $\theta_1=\theta_2=1$ in the proof above. If so,
then
 the term
	$\psi(x,z):=|x+z|-|x|-\<x/|x|,z\>$ will appear in estimates on $I_2(t)$ and $I_3(t)$. Obviously, in the general setting, we
 only
 have $|\psi(x,z)|\le2|z|$, and so
 both
 $\int_{\{|z|\le1\}}|z|(\nu_t-\nu_t\wedge\bar\nu)(\d z)$ and $\int_{\{|z|>1\}}|z|(\nu_t-\nu_t\wedge\bar\nu)(\d z)$ will be involved in the corresponding estimates. Nevertheless, they
	are not integrable simultaneously in most   cases, for example, even when driven  noises $(Z_t)_{t\ge0}$ and $(\bar Z_t)_{t\ge0}$
	are distinct and have symmetric $\alpha$-stable-like distributions with $\alpha\in (0,2)$. Based on this, in order to make Theorem \ref{thm1} more practical, it is necessary to amend   the associated (quasi)-distance function so   we
	construct a novel
	function $h $ given in \eqref{EW11}.
\end{remark}

\subsection{Uniform-in-time estimates: the partially dissipative-type case}
In this subsection, we move from the uniformly dissipative-type case to the    partially dissipative-type setup. More precisely, we have the following statement.

\begin{theorem}\label{thm3}
	Assume that the following conditions are satisfied:
	\begin{itemize}
		\item[{\rm(i)}] There are constants $\lambda_1,\lambda_2>0$,
		 $\alpha_*\in(0,1]$
		 and $\ell_0\ge1$ such that for all $x,y\in\R^d,$
		\begin{align*}
			\<x-y,\bar b(x)-\bar b(y)\>\le   \lambda_1\I_{\{|x-y|\le\ell_0\}}|x-y|^{1+\alpha_*}-\lambda_2 \I_{\{|x-y|>\ell_0\}}|x-y|^2.
		\end{align*}
		\item[{\rm(ii)}] \eqref{R2} is  satisfied, i.e., there is a function  $\mathcal V: [0,\8)\times\R^d\to[0,\8)$  satisfying $\lim_{t\to \infty} \mathcal V(t,x)=0$ for all $x\in \R^d$
		so that for all $t\ge0$ and $x\in \R^d$,
		$$
		|b(t,x)-\bar b(x)|\le  \mathcal V(t,x).
		$$
		\item[{\rm(iii)}]There are constants $c^*>0$,
		$\beta\in (1-\alpha_*,
		1)$, $\kappa,\theta\in (0,1]$ and a continuous  function $h:[0,\infty)\times [0,\kappa]\rightarrow [0,c^*]$ satisfying that $\lim_{t\to\infty} h(t,r)=0$ for all $r\in [0,\kappa]$ so that
		for all $r\in (0,\kappa]$ and $t\ge0$,  \begin{align}\label{R5--}
			\inf_{|x|\le r}\nu_{t,x}(\R^d) \ge r^{-\beta} (c^*-h(t,r) ),
		\end{align}
		and for all
$t\ge0$,
		\begin{align}\label{WY}
			H(t):=\int_{\R^d} (|z|\wedge |z|^\theta)\, |\nu_t-\bar \nu|(\d z)<\infty.
	\end{align}\end{itemize}
	Then, there is a constant $C_0 >0$ such that for all $t\ge0$ and $\mu_1,\mu_2\in \mathcal P_{\theta}(\R^d)$,
	\begin{equation}\label{R8}
		\begin{split}
			 \mathcal W_{|\cdot|\wedge|\cdot|^\theta}\Big (\mathscr L_{Y_t^{\mu_1}},\mathscr L_{\overline Y_{\!\! t}^{\mu_2}}\Big)
			&\le     C_0 \e^{- \lambda_* t}\mathcal
W_{|\cdot|\wedge|\cdot|^\theta}(\mu_1,\mu_2) \\
 &\quad+ C_0\int_0^t\e^{-\lambda_*(t-s)}\big( H(s) +h^*(s,\kappa) +\E \mathcal V(s, Y_s^{\mu_1} ) \big)\,\d s.
		\end{split}
	\end{equation}
Here, $h^*(t,\kappa):=\sup_{r\le \kappa} (r^{1+\rho-\beta}h(t,r)) $ for $\rho\in(0,\beta+\alpha_*-1)$, and
	$$
	\lambda_*:= (\lambda_\star\e^{-g(\ell_0)}/2)\wedge(\lambda_2\e^{-g(2\ell_0)})\wedge\bigg(
	\frac{
\lambda_2}{3}\bigg(\theta\wedge\frac{ 2 \e^{-g(2\ell_0)}(1+2\ell_0)}{2\ell_0\e^{-g(2\ell_0)}+\int_0^{2\ell_0}\e^{-g(s)}\,\d s}\bigg)\bigg),$$ where $g(r):= c_\star r^\rho$, $ \forall\, r\ge0$,
and
\begin{align}\label{WWW-1}
	\lambda_\star:=  2\lambda_1\ell_0^{\alpha_*-1} ,
\quad c_\star:=8\lambda_1\kk^{\beta-2}\ell_0^{\alpha_*+1-\rho}/{(c^*\rho)}.
\end{align}	
\end{theorem}

\begin{proof} In the following analysis, we take $\delta=1/2$ and $\kk\in (0,1]$ (which is stipulated in {\rm(iii)})
	in the coupling process $\big(Y_t^{(\delta)},\overline {Y}_{\!\! t}^{(\delta)}\big)_{t\ge0}$
defined by \eqref{Q9}. In the sequel, we write $(Y_t ,\overline {Y}_{\!\! t})_{t\ge0}$ instead of $\big(Y_t^{(1/2)},\overline {Y}_{\!\! t}^{(1/2)}\big)_{t\ge0}$ for notational simplicity. For $\theta\in(0,1]$ (which is
also
 involved in {\rm(iii)}) and
	$c^{**}:
	= 2 \e^{-g(2\ell_0)}(1+2\ell_0)^{1-\theta}/\theta$,
	define
	\begin{equation}\label{LW1}
		f_\theta (r)=
		\begin{cases}
			\e^{-g(2\ell_0)} r+\int_0^r\e^{-g(s )}\,\d s,&0\le r\le 2\ell_0,\\
			f_\theta(2\ell_0)+c^{**}\big((1+r)^\theta-(1+2\ell_0)^\theta\big), & r>2\ell_0.
		\end{cases}
	\end{equation}
	In terms of the definition  of $c^{**}$,
	it is ready to see that $[0,\infty)\ni r\mapsto f_\theta(r)$ is a $C^1$-function.

	From \eqref{Q9} with $\delta=1/2$, we have that for all $t\ge0,$
	\begin{align*}
		\d  {U}_t = (b(t,Y_t  )-\bar b(\overline {Y}_{\!\! t}))\,\d t-\int_{\R^d\times[0,1]}\Lambda_t ((\bar {U}_{t-})_{\kappa} ,z,u)\,N^{(1 )}(\d z,\d u,\d t) +\d   Z_t^{ (2) }-\d   Z_t^{ (3)},
	\end{align*}
where 	$ {U}_t :=Y_t-\overline {Y}_{\!\! t}$.
	Thus,
	applying  It\^o's formula enables us to derive that for all $t\ge0,$
	\begin{align*}
		\d f_\theta(| {U}_t |  )&= \d M_t+\frac{f'_\theta(| {U}_t|  )}{| {U}_t | }\<U_t, b(t,Y_t)-\bar b(\overline {Y}_{\!\! t})\> \,\d t\\
		&\quad+\int_{\R^d\times[0,1]}\big(f_\theta(| {U}_t-\Lambda_t(( {U}_t)_{\kappa},z,u)|)-f_\theta(| {U}_t|)\big)( \nu_t\wedge \bar\nu)(\d z)\,\d u\,\d t\\
		&\quad+\bigg[\int_{\R^d}\Big(f_\theta(|U_t +z|)-f_\theta(| {U}_t|)-\frac{f'_\theta(| {U}_t|)}{| {U}_t|}\< {U}_t,z\>\I_{\{|z|\le1\}}\Big)\,( \nu_t-(\nu_t\wedge \bar\nu))(\d z) \\
		&\quad\quad +\int_{\R^d}\Big(f_\theta(| {U}_t +z|)-f_\theta(| {U}_t|)-\frac{f'_\theta(| {U}_t|)}{| {U}_t|}\< {U}_t,z\>\I_{\{|z|\le1\}}\Big)\,( \bar \nu-(\nu_t\wedge \bar\nu))(\d z)\bigg]\,\d t \\
		&=:\d M_t+ \big(I_1(t)+I_2(t)+ I_3(t) \big)\,\d t,
	\end{align*}
	where $(M_t )_{t\ge0}$ is a martingale.

	Below, we aim at estimating the terms $I_1(t)$, $I_2(t)$ and $I_3(t)$, one by one.
	By means of   $ f'_\theta\in(0, 2]$,
	we deduce from \eqref{R2} and (i) that for all $t\ge0,$
	\begin{align*}
		I_1(t) \le     f'_\theta(| {U}_t|) \big(\lambda_1
		| {U}_t|^{\alpha_*}\I_{\{| {U}_t|\le \ell_0\}}
		-\lambda_2| {U}_t|\I_{\{| {U}_t|> \ell_0\}}\big) + 2 \mathcal V(t,Y_t).
	\end{align*}
	With the aid of  $\nu_{t,-(U_{t} )_{\kappa}}(\R^d)=\nu_{t,(U_{t})_{\kappa}}(\R^d)$, we derive   that  for all $t\ge0,$
	\begin{align*}
		I_2(t)&=\frac{1}{2} \int_{\R^d}\big(f_\theta(| {U}_t-( {U}_t)_{\kappa}|)-f_\theta(| {U}_t|)\big)\rho_t(-( {U}_t)_{\kappa},z)\,( \nu_t\wedge \bar\nu)(\d z) \\
		&\quad+\frac{1}{2} \int_{\R^d}\big(f_\theta(| {U}_t+( {U}_t)_{\kappa}|)-f_\theta(| {U}_t|)\big)\rho_t(( {U}_t)_{\kappa},z)\,( \nu_t\wedge \bar\nu)(\d z) \\
		&=\frac{1}{2} \big(f_\theta(| {U}_t|-({\kappa}\wedge| {U}_t|))-f_\theta(| {U}_t|)\big)\,\nu_{t,-( {U}_t)_{\kappa}}(\R^d) \\
		&\quad+\frac{1}{2} \big(f_\theta( | {U}_t|+ ({\kappa}\wedge| {U}_t|)  )-f_\theta(| {U}_t|)\big)\nu_{t,( {U}_t)_{\kappa}}(\R^d) \\
		&=\frac{1}{2} \nu_{t,( {U}_t)_{\kappa}}(\R^d) \big(f_\theta(| {U}_t|-({\kappa}\wedge| {U}_t|))
		+f_\theta(| {U}_t|+({\kappa}\wedge| {U}_t|))-2f_\theta(| {U}_t|)\big).
	\end{align*}
	This, together with \eqref{R5--} and the fact that
	\begin{align}\label{SD}
		f_\theta(r-\delta)+f_\theta(r+\delta)-2f_\theta(r)\le0, \quad \forall\, 0\le \delta\le r
	\end{align}
	by noticing that  $[0,\infty)\ni r\mapsto f'_\theta(r)$ is decreasing, leads to the following estimate:  for all $t\ge0,$
	\begin{align*}
		I_2(t) \le\frac{1}{2}  |( {U}_t)_\kappa|^{-\beta} (c^*-h(t, |({U}_t)_{\kappa}|))\big(f_\theta(| {U}_t|-({\kappa}\wedge| {U}_t|))
		+f_\theta(|{U}_t|+({\kappa}\wedge| {U}_t|))-2f_\theta(| {U}_t|)\big).
	\end{align*}
	Again, by noting that $[0,\infty)\ni r\mapsto f'_\theta(r)$ is decreasing and $f_\theta(0)=0$,  we have $f_\theta(b)-f_\theta(a)\le f'_\theta(a)(b-a)$ and $f_\theta(a+b)\le f_\theta(a)+f_\theta(b)$ for all $a,b\ge0.$
	Subsequently,  with the help of  $f'_\theta\in(0, 2],$
	we derive  that for all $t\ge0,$
	\begin{align*}
		I_3(t)&\le f'_\theta(| {U}_t | )\int_{\{|z|\le 1\}}\Big( | {U}_t +z|  - | {U}_t |  -\frac{1}{| {U}_t|}\< {U}_t ,z\>\Big) \,( \nu_t-(\nu_t\wedge \bar\nu))(\d z) \\
		&\quad+   f'_\theta(| {U}_t| )\int_{\{|z|\le 1\}}\Big( | {U}_t +z|  - | {U}_t|  -\frac{1}{| {U}_t|}\< {U}_t ,z\>\Big) \,( \bar \nu-(\nu_t\wedge \bar\nu))(\d z)\\
		&\quad+ \int_{\{|z|> 1\}}\Big( f_\theta(| {U}_t| +|z|)-f_\theta (| {U}_t |)\Big) \,( \nu_t-(\nu_t\wedge \bar\nu))(\d z)\\
		&\quad+ \int_{\{|z|> 1\}}\Big( f_\theta(| {U}_t| +|z|)-f_\theta (| {U}_t |)\Big)  \,( \bar \nu-(\nu_t\wedge \bar\nu))(\d z)\\
		&\le 2 \int_{\{|z|\le 1\}} |z|\,|\bar \nu-\nu_t|(\d z)+\int_{\{|z|> 1\}} f_\theta(|z|)\,|\bar \nu-\nu_t|(\d z)\\
		&\le C_*H(t),
	\end{align*}
	where in the last display $$
	C_*:=2 \vee (2(2\ell_0)^{1-\theta})\vee (2^\theta
c^{**}+f_\theta(2\ell_0)),$$ and we employed the fact that
	$f_\theta(r)\le [(2(2\ell_0)^{1-\theta})\vee (2^\theta
c^{**}+f_\theta(2\ell_0))] r^\theta$ for all $r>0$,
thanks to $\ell_0\ge1$.
	Now, putting together the previous estimates on $I_1(t)$, $I_2(t)$ and $I_3(t)$   yields that for all $t\ge0,$
	\begin{align*}
		\d f_\theta(| {U}_t |  )&\le \big(\Lambda_\theta(| {U}_t |)  + C_*H(t)+ 2 \mathcal V(t,Y_t )\big)\,\d t+\d M_t,
	\end{align*}
	where for all $r>0,$
	\begin{align*}
		\Lambda_\theta(r):= & f'_\theta(r) \big(\lambda_1
		r^{\alpha_*}\I_{\{r\le \ell_0\}}
		-\lambda_2r\I_{\{r> \ell_0\}}\big)\\
		&+\frac{1}{2}(\kk \wedge r)^{-\beta} (c^*-h(t,\kk\wedge r))\big(f_\theta(r-({\kappa}\wedge r))
		+f_\theta(r+({\kappa}\wedge r))-2f_\theta(r)\big).
	\end{align*}
	By virtue of    $f'_\theta(r)>0,f''_\theta(r)<0$ as well as $f'''_\theta(r)>0$ for all $r\in(0,2\ell_0)$, it follows from
	\cite[Lemma 4.1]{LW19} that $f_\theta(r+\delta)+f_\theta(r-\delta)-2f_\theta(r)\le f''_\theta(r)\delta^2$ for all $0\le \delta\le r\le \ell_0$. Then, we find that
	for all $r\in(0,\ell_0]$,
	\begin{align*}
		\Lambda_\theta(r)&=\lambda_1f'_\theta(r)r^{\alpha_*}
		+\frac{1}{2}(\kappa\wedge r)^{-\beta} (c^*-h(t, \kappa\wedge r))\big(f_\theta(r+\kappa\wedge r)+f_\theta(r-\kappa\wedge r)-2f_\theta(r)\big)\\
		&\le \lambda_1f'_\theta(r)r^{\alpha_*}
		+\frac{1}{2}(\kappa\wedge r)^{2-\beta} (c^*-h(t, \kappa\wedge r)) f''_\theta(r)   \\
		&=\Big(\lambda_1\big(\e^{-g(2\ell_0)}+\e^{-g(r)}\big)
		r^{\alpha_*-1}
		 -\frac{1}{2}c^* g'(r)\e^{-g(r)} r^{-1}(\kappa\wedge r)^{2-\beta}\Big)r-\frac{1}{2}(\kappa\wedge r)^{2-\beta}  h(t, \kappa\wedge r)  f''_\theta(r) \\
&\le  \Big(2\lambda_1
 r^{\alpha_*-1}
-\frac{1}{2}c^* c_\star\rho r^{\rho-2}
  (\kappa\wedge r)^{2-\beta}\Big)\e^{-g(r)}r+
  \frac{1}{2}\rho c_\star r^{\rho-1}
  (\kappa\wedge r)^{2-\beta}  h(t, \kappa\wedge r) \e^{-g(r)},
	\end{align*}
	where in the
	second
	 identity we used the facts that for all $r\ge0,$
	\begin{align*}
		f'_\theta(r)=\e^{-g(2\ell_0)}+\e^{-g(r)}\quad \mbox{ and } \quad f''_\theta(r)=-g'(r)\e^{-g(r)}=-c_\star\rho r^{\rho-1}\e^{-g(r)}.
	\end{align*}
Note from $1+\rho-\beta-\alpha_*<0$
and $\rho\in(0,1)$ that 	for all $r\in(0,\ell_0]$,
\begin{align*}
2\lambda_1  r^{\alpha_*-1}-\frac{1}{2}c^* c_\star\rho r^{\rho-2}  (\kappa\wedge r)^{2-\beta}&
=
\begin{cases}
\big(2\lambda_1  -\frac{1}{2}c^* c_\star\rho r^{1+\rho-\beta-\alpha_*}\big) r^{\alpha_*-1},&
0<r\le \kk\\
\big(2\lambda_1r^{\alpha_*-1}  -\frac{1}{2}c^* c_\star\rho\kk^{2-\beta} r^{ \rho-2}\big) ,& \kk<r\le \ell_0\\
\end{cases}\\
&\le
\begin{cases}
 -\frac{1}{4}c^* c_\star\rho r^{ \rho-\beta } ,&
 0<r\le \kk\\
  -\frac{1}{4}c^* c_\star\rho\kk^{2-\beta} r^{ \rho-2}  ,& \kk<r\le \ell_0\\
\end{cases}\\
&\le - \lambda_\star,
\end{align*}
 where in the
 first inequality we took
  the definition of $c_\star$
  into consideration, and $\lambda_\star>0$ was defined in \eqref{WWW-1}.
Furthermore,  by virtue of \eqref{SD},  we achieve that
	$$
	\Lambda_\theta(r)\le
	-2\lambda_2\e^{-g(2\ell_0)}r\I_{\{\ell_0<r\le2\ell_0\}}-\frac{2 }{3}c^{**}
	\theta\lambda_2(1+r)^{\theta } \I_{\{r>2\ell_0\}},
	\quad \forall\, r\ge \ell_0\ge1,
	$$
	where we also exploited the fact that
	\begin{equation}\label{LW-3}
		\begin{split}
		f_\theta'(r)r&=\big(\e^{-g(2\ell_0)}+\e^{-g(r)}\big)r\I_{\{0\le r\le 2\ell_0\}}+c^{**}
		\theta(1+r)^{\theta-1}r\I_{\{r>2\ell_0\}}\\
		&\ge 2\e^{-g(2\ell_0)}r\I_{\{0\le r\le 2\ell_0\}}+\frac{2}{3}c^{**}
		\theta(1+r)^{\theta }
\I_{\{r>2\ell_0\}}.
	\end{split}
	\end{equation}
	Accordingly, based on the preceding analysis, we conclude that for all $t\ge0,$
	\begin{align*}
		\d f_\theta(| {U}_t |  )\le \Big(&- \big((\lambda_\star
		\e^{-g(\ell_0)})\wedge(2\lambda_2\e^{-g(2\ell_0)})\big) | {U}_t|\I_{\{|U_t|\le2\ell_0\}}-\frac{2 }{3}c^{**}
		\theta\lambda_2(1+|U_t|)^{\theta } \I_{\{|U_t|>2\ell_0\}} \\
		&+C_*H(t) + 2\mathcal V(t,Y_t)\\
&+\frac{1}{2}\rho c_\star| {U}_t|^{\rho-1}(| {U}_t|\wedge\kappa)^{2-\beta} h(t,  | {U}_t|\wedge\kappa) \I_{\{| {U}_t|\le\ell_0\}} \Big)\,\d t+ \d M_t.
	\end{align*}
	This, in addition to
	\begin{align}\label{LW-4}
		f_\theta(r)\le 2r, \quad   \forall\, r\le2\ell_0;  \quad
		\frac{1}{2}\bigg(\frac{
			c^{**} (1+2\ell_0)^\theta}{f_\theta(2\ell_0)}\wedge1\bigg)f_\theta(r)\le c^{**}
			 (1+r)^\theta,\quad \forall\, r>2\ell_0,
	\end{align}
	implies that for all $t\ge0,$
	\begin{align*}
		\d f_\theta(| {U}_t |  )&\le \d M_t   -\lambda_*f_\theta(| {U}_t |)\,\d t\\
		&\quad  + \Big(C_*H(t)  + 2\mathcal V(t,Y_t )+\frac{1}{2}\rho c_\star (| {U}_t|\wedge\kappa)^{1+\rho-\beta} h(t,   | {U}_t|\wedge\kappa)
 \Big)\,\d t\\
 &\le \d M_t   -\lambda_*f_\theta(| {U}_t |)\,\d t + \Big(C_*H(t)  + 2\mathcal V(t,Y_t )+\frac{1}{2}\rho c_\star h^*(t,\kappa)
 \Big)\,\d t.
	\end{align*}
	Thereby,  applying Gronwall's inequality and then using $\mathcal W_{f_\theta}(\mathscr L_{Y_t },\mathscr L_{\overline {Y}_{\!\! t}})\le \E f_\theta(| {U}_t|  )$, which is valid by taking Lemma \ref{lem1} and Lemma \ref{lem4} into account,
	and  choosing $(Y_0,\overline {Y}_{\!\! 0})$ such that $\E f_\theta(| {U}_0 |  )=\mathcal W_{f_\theta}(\mu_1,\mu_2)$
	gives that for all $t\ge0$,
	\begin{align*}
		\mathcal W_{f_\theta}\Big(\mathscr L_{Y_t^{\mu_1}},\mathscr L_{\overline {Y}_{\!\! t}^{\mu_2}}\Big)
		\le & \e^{- \lambda_* t}\mathcal W_{f_\theta}(\mu_1,\mu_2)\\
	& + \int_0^t\e^{-\lambda_*(t-s)}\Big ( C_*H(s)  + 2\E\mathcal V(s, Y_s^{\mu_1})
+\frac{1}{2}\rho c_\star  h^*(s,\kappa)
  \Big)\,\d s.
	\end{align*} Correspondingly, by noting  that there are constants $c_1,c_2>0$ so that for all $r\ge0, $
\begin{align}\label{LW-5}
c_1(r\wedge r^\theta)\le f_\theta(r)\le c_2(r\wedge r^\theta),
\end{align}
the assertion \eqref{R8}
	follows.
\end{proof}

Before the end of this section,  we make  a comment on  Theorem \ref{thm3}, which will be applied to prove Theorem \ref{thm4}.

\begin{remark}\label{REm}
Once $\nu_t=\bar \nu$ is
the L\'evy measure
corresponding to a symmetric $\alpha$-stable process
with $\alpha\in (1,2)$,  \eqref{WY} holds true trivially with $
H(t)\equiv0$, and \eqref{R5--} is satisfied
 with $h(t,r)\equiv0$ (see, for example, \cite[Example 1.2]{LW19}). Then, in this special case, the assertion \eqref{R8} holds with $\theta=1$ and correspondingly becomes
\begin{equation*}
 \mathcal W_1(\mathscr L_{Y_t^{\mu_1}},\mathscr L_{\overline {Y}_{\!\! t}^{\mu_2}})   \le    C_0
 \bigg(\e^{- \lambda_* t}\mathcal W_1(\mu_1,\mu_2)   + \int_0^t\e^{-\lambda_*(t-s)}\E   \mathcal V(s, Y_s^{\mu_1} ) \,\d s\bigg).
\end{equation*} The assertion above   has been proved in \cite[Theorem 1.2]{BSWX} and  applied to investigate the exponentially contractive property of distributions corresponding to partially dissipative McKean-Vlasov SDEs with jumps; see \cite[Theorem 2.4]{BSWX} for related  details.
\end{remark}

\section{Two preliminaries}\label{Section3-}

\subsection{Characterization  of the time-space changed noise $(Z_t)_{t\ge0}$}\label{Section3-1}
To treat the long-term behavior of $(Y_t)_{t\ge0}$ solving \eqref{RR-}, it  is necessary  to investigate
the properties of the driven noise $(Z_t)_{t\ge0}$ involved in \eqref{RR-}.
The following proposition provides an explicit expression of
the characteristic function of $(Z_t)_{t\ge0}$, which is defined via the space-time change of a pure jump L\'evy process
 $(L_t)_{\ge0}$. In
 detail,
 for any $t\ge0$, let $$Z_t=\int_0^tg_s\,\d L_{\varphi_s},$$ where $(L_t)_{t\ge0}$ is a pure jump L\'evy process with the L\'evy measure $\nu$,  $\varphi:[0,\8)\to[0,\8)$ is an increasing $C^1$-function, and $g:[0,\8)\to (0,1]$ is a decreasing $C^1$-function obeying $\lim_{t\to\infty}g_t=0$.

\begin{proposition}\label{lem} For any $\xi\in\R^d$ and $t\ge0,$
	\begin{align*}
		\E\e^{i\langle\xi,  Z_t\rangle}=& \exp\bigg(-\int_0^t\int_{\R^d}\big(1-\e^{i\langle\xi, z\rangle}+i\langle\xi, z\rangle\I_{\{|z|\le1\}}\big) \,\nu_s(\d z)\,\d s-i\int_0^t\int_{\R^d} \langle\xi,  z\rangle\I_{\{g_s<|z|\le 1\}} \, \nu_s(\d z)\,\d s \bigg),
	\end{align*} where
	\begin{align}\label{EW}
		\nu_t (\d z):=\varphi_t'(\nu\circ (\bar g_t)^{-1})\,(\d z)  \quad \mbox{ with } \quad \bar g_t (z):=g_tz.
	\end{align}
\end{proposition}

\begin{proof}
	Let $\Phi$ be
	the  characteristic exponent (or the symbol)  of
	the   pure jump L\'{e}vy process $(L_t)_{t\ge0}$, which  is given  by
	\begin{align}\label{E-}
		\Phi(\xi) =\int_{\R^d}\big(1-\e^{i\langle \xi, z\rangle}+i\langle\xi, z\rangle\I_{\{|z|\le1\}}\big)\,\nu(\d z),\quad \forall\,\xi\in\R^d.
	\end{align}
	It suffices to show  that  for any $\xi\in\R^d$ and $t\ge0,$
	\begin{align}\label{E3}
		\E \e^{i\langle\xi ,Z_t\rangle}=\e^{-\int_0^t\varphi_s'\Phi(\xi g_{s })\,\d s},
	\end{align}
	then the assertion follows by noting  that for all $t\ge0,$
	\begin{align*}
		\int_0^t\varphi_s'\Phi(\xi g_{s })\d s&=\int_0^t\varphi_s'\int_{\R^d}\big(1-\e^{i\langle g_s\xi, z\rangle}+i\langle g_s\xi,  z\rangle\I_{\{| z|\le 1\}}\big)\,\nu(\d z)\,\d s\\
		&=\int_0^t\varphi_s'\int_{\R^d}\big(1-\e^{i\langle\xi, g_s z\rangle}+i\langle\xi, g_s z\rangle\I_{\{|g_sz|\le 1\}}\big)\,\nu(\d z)\,\d s\\
		&\quad- i\int_0^t\int_{\R^d}\langle\xi,\varphi_s' g_s z\rangle\I_{\{g_s<|g_sz|\le 1\}} \,\nu(\d z)\,\d s \\
		&=\int_{\R^d}\int_0^t\big(1-\e^{i\langle\xi,  z\rangle}+i\langle\xi,  z\rangle\I_{\{| z|\le 1\}}\big)\,\nu_s(\d z)\,\d s\\
		&\quad -i\int_0^t\int_{\R^d} \langle\xi, z\rangle\I_{\{g_s<|z|\le 1\}} \nu_s(\d z)\d s.
	\end{align*}

	Below, we focus on the establishment of \eqref{E3}.  For any $t\ge0$,  let
	$$\pi=\{0=s_0<s_1<\cdots<s_n=t\}$$ be a partition of $[0,t]$ with mesh $|\pi|\to0$, and define
	\begin{align*}
		Z_t^{\pi}=\int_0^tg_s^{\pi}\,\d L_{\varphi_s}\quad \mbox{ with } \quad
		g_t^{\pi}:=\sum_{j=1}^ng_{s_{j-1}}\I_{[s_{j-1},s_j)}(t).
	\end{align*}
	It is easy to see  that for any $\vv>0$ and $t\ge0$,
	\begin{align*}
		\P(|Z_t^{\pi}-Z_t|\ge\vv)&=\P\Big(\Big|\int_0^t(g_s^{\pi}-g_s)\,\d L_{\varphi_s}\Big|\ge\vv\Big)\\
		&=\P\Big(\Big|\int_{\varphi_0}^{\varphi_t}\big(g_{\varphi_s^{-1}}^{\pi}-g_{\varphi_s^{-1}}\big)\,\d L_s\Big|\ge\vv\Big)\\
		&=\P\Big(\Big|\int_{\varphi_0}^{\varphi_t}\big(g_{\varphi_s^{-1}}^{\pi}-g_{\varphi_s^{-1}}\big)\Big(\int_{\{|z|\le1\}}z\,\tilde N(\d z,\d s)+\int_{\{|z|>1\}}z \,N(\d z,\d s)\Big) \Big| \ge\vv\Big)\\
		&\le \P\Big(\Big|\int_{\varphi_0}^{\varphi_t} \int_{\{|z|\le1\}}\big(g_{\varphi_s^{-1}}^{\pi}-g_{\varphi_s^{-1}}\big)z\,\tilde N(\d z,\d s)\Big| \ge\vv/2\Big)\\
		&\quad+\P\Big(\Big|\int_{\varphi_0}^{\varphi_t}\int_{\{|z|>1\}}\big(g_{\varphi_s^{-1}}^{\pi}-g_{\varphi_s^{-1}}\big) z  \,N(\d z,\d s)\Big| \ge\vv/2\Big)\\
		&=:I_1^\pi(t)+I_2^\pi(t),
	\end{align*}
  	where the third identity holds true by the L\'{e}vy-It\^o decomposition of $(L_t)_{t\ge0}$ with $N(\d z,\d t)$ being the Poisson measure  and $\tilde N(\d z, \d t)= N(\d z,\d t)-\nu(\d z)\d t$ being the corresponding compensated Poisson measure.
	By Chebyshev's inequality and It\^o's isometry, it follows that for all $t\ge0,$
	\begin{align*}
		I_1^\pi(t)\le \frac{4}{\vv^2}\int_{\{|z|\le1\}}|z|^2\,\nu(\d z) \int_{\varphi_0}^{\varphi_t} \big|g_{\varphi_s^{-1}}^{\pi}-g_{\varphi_s^{-1}}\big|^2\d s.
	\end{align*}
	To treat the term $I_2^\pi(t)$, we introduce
	\begin{align*}
		\hat Z_t=\int_0^t\int_{\{|z|>1\}}z\,N(\d z,\d s)\quad \mbox{ with } \quad\triangle \hat Z_t:=\hat Z_t-\hat Z_{t-},\quad \forall\, t\ge0.
	\end{align*}
	It is known from $1-\e^{-r}\le r$ for $r\ge0$
	that for any $\eta>1$ and $t\ge0,$
	$$
	\P\left(\sup_{s\in [0,\varphi_t]}|\triangle \hat Z_s | \ge \eta\right)=1-\exp\left(-\varphi_t\nu\{\cdot\in \R^d:|\cdot|\ge \eta\} \right)\le \varphi_t\nu\{\cdot\in\R^d:|\cdot|\ge \eta\}.
	$$
	This, combined   with Chebyshev's  inequality, yields that for any $\eta>1 $ and $t\ge0,$
	\begin{align*}
		I_2^\pi(t)&\le \P\Big(\Big|\int_{\varphi_0}^{\varphi_t}\int_{\{|z|>1 \}}\big(g_{\varphi_s^{-1}}^{\pi}-g_{\varphi_s^{-1}}\big) z \, N(\d z,\d s)\Big| \ge\vv/2, \sup_{s\in [0,\varphi_t]}|\triangle \hat Z_s | <\eta \Big)\\
		&\quad+ \P\Big(\sup_{s\in [0,\varphi_t]}|\triangle \hat Z_s | \ge \eta\Big)\\
		&\le \P\Big(\Big|\int_{\varphi_0}^{\varphi_t}\int_{\{1<|z|< \eta \}}\big(g_{\varphi_s^{-1}}^{\pi}-g_{\varphi_s^{-1}}\big) z  \,N(\d z,\d s)\Big| \ge\vv/2 \Big) +\P\Big(\sup_{s\in [0,\varphi_t]}|\triangle \hat Z_s | \ge \eta\Big)\\
		&\le \frac{2}{\vv}\int_{\{1<|z|< \eta \}}| z| \, \nu(\d z)\int_{\varphi_0}^{\varphi_t}\big|g_{\varphi_s^{-1}}^{\pi}-g_{\varphi_s^{-1}}\big|\,\d s+ \varphi_t\nu\{\cdot\in\R^d:|\cdot|\ge \eta\}.
	\end{align*}
	Based on the  preceding analysis, for any $\eta>1 $ and $t\ge0,$ we arrive at
	\begin{align*}
		\P(|Z_t^{\pi}-Z_t|\ge\vv) &\le \frac{4}{\vv^2}\int_{\{|z|\le1\}}|z|^2\,\nu(\d z) \int_{\varphi_0}^{\varphi_t} \big|g_{\varphi_s^{-1}}^{\pi}-g_{\varphi_s^{-1}}\big|^2\,\d s\\
		&\quad+ \frac{2}{\vv}\int_{\{1<|z|< \eta \}}| z|  \,\nu(\d z)\int_{\varphi_0}^{\varphi_t}\big|g_{\varphi_s^{-1}}^{\pi}-g_{\varphi_s^{-1}}\big|\,\d s+ \varphi_t\nu\{\cdot\in\R^d:|\cdot|\ge \eta\}.
	\end{align*}
	Therefore,  we conclude that for any $t\ge0,$
	\begin{align}\label{Q8}
		Z_t^\pi\overset{d}\to Z_t\quad \mbox{ as } \quad |\pi|\to0
	\end{align}
	by taking the continuity of $g$
	into consideration, and approaching $|\pi|\to0$ followed by sending  $\eta\to \infty$ in the estimate above.  Subsequently, we derive from \cite[Proposition 2.5(vi)]{Sato}  that
	for any $\xi\in\R^d $ and $t\ge0,$
	\begin{align*}
		\E\e^{i\langle\xi, Z_t\rangle}=\lim_{|\pi|\to0}\E\e^{i\langle\xi, Z_t^{\pi}\rangle }&=\lim_{|\pi|\to0}\E\Big(\Pi_{j=1}^n\e^{i\langle\xi, g_{s_{j-1}}(L_{\varphi_{s_j}}-L_{\varphi_{s_{j-1}}}) \rangle}\Big)\\
		&=\lim_{|\pi|\to0}\Pi_{j=1}^n\E \e^{i\langle\xi, g_{s_{j-1}}(L_{\varphi_{s_j}}-L_{\varphi_{s_{j-1}}})\rangle } \\
		&=\lim_{|\pi|\to0}\Pi_{j=1}^n  \e^{-(\varphi_{s_j}-\varphi_{s_{j-1}})\Phi(\xi g_{s_{j-1}}) }\\
		&=\e^{-\int_0^t\varphi_s'\Phi(\xi g_{s })\d s},
	\end{align*}
	where the third identity holds true due to the independent increments of $(L_t)_{t\ge0}$, the fourth equality is valid since the symbol of $(L_t)_{t\ge0}$ is
$\Phi(\xi)$
given by \eqref{E-}, and the last display is true by virtue of the continuity of $g$ and the differential property of $\varphi$. As a consequence, the statement \eqref{E3}  follows.
\end{proof}

The process $(Z_t)_{t\ge0}$   is an additive process as stated in the following proposition.

\begin{proposition}\label{lem2}
	$(Z_t)_{t\ge0}$ is an additive process.
\end{proposition}

\begin{proof}
	(1) First, we show that $(Z_t)_{t\ge0}$ has independent increments. It is sufficient to verify that $Z_{u_4}-Z_{u_3}$ is independent of $Z_{u_2}-Z_{u_1}$
	for any $0\le u_1<u_2<u_3<u_4 $. Let
	\begin{align*}
		\pi=\{u_1=s_1<s_2<\cdots s_n=u_2\}\quad \mbox{ and } \quad \pi^*=\{u_3=s_1^*<s_2^*<\cdots s_m^*=u_4\}
	\end{align*}
	be the  partitions  of $[u_1,u_2]$ and $ [u_3,u_4]$ with the mesh $|\pi|\to0$ and $|\pi^*|\to0$, respectively.
	Below, define
	\begin{align*}
		Z_{u_1,u_2}^\pi=\int_{u_1}^{u_2}g^\pi_s\,\d L_{\varphi_s}\quad \mbox{ and } \quad Z_{u_3,u_4}^{\pi^*}=\int_{u_3}^{u_4}g^{\pi^*}_s\,\d L_{\varphi_s}
	\end{align*}
	with $g^\pi_s:=\sum_{i=1}^ng_{s_{i-1}}\I_{[s_{i-1},s_i)}(s)$ and $g^{\pi^*}_s:=\sum_{i=1}^mg_{s_{i-1}^*}\I_{[s_{i-1}^*,s_i^*)}(s)$, respectively.
	Since $(L_t)_{t\ge0}$ has   independent increments and the function $g$ is deterministic,
	$Z_{u_1,u_2}^\pi$ is independent of $Z_{u_3,u_4}^{\pi^*}$ so that for any $\xi\in\R^d,$
	\begin{align*}
		\E\e^{i\langle\xi,Z_{u_1,u_2}^\pi+Z_{u_3,u_4}^{\pi^*}\rangle}=\E \e^{i\langle\xi,  Z_{u_1,u_2}^\pi\rangle }\E \e^{i\langle\xi , Z_{u_3,u_4}^{\pi^*} \rangle}
	\end{align*}
	By following exactly the procedure to derive \eqref{Q8}, we have
	\begin{align*}
		Z_{u_1,u_2}^\pi\overset{d}\to Z_{u_2}-Z_{u_1}\quad {{\rm as }}\quad  |\pi|\to0 \quad \mbox{ and } \quad Z_{u_3,u_4}^{\pi^*}\overset{d}\to Z_{u_4}-Z_{u_3}\quad { {\rm as }} \quad |\pi^*|\to0.
	\end{align*}
	Hence, we arrive at
	\begin{align*}
		\E\e^{i\langle\xi, Z_{u_2}-Z_{u_1}+Z_{u_4}-Z_{u_3} \rangle}=\E\e^{i\langle\xi,Z_{u_2}-Z_{u_1}\rangle}\E\e^{i\langle\xi ,Z_{u_4}-Z_{u_3}\rangle}.
	\end{align*}
	This, together with  Kac's theorem,  implies that $Z_{u_4}-Z_{u_3}$ and $Z_{u_2}-Z_{u_1}$ are mutually independent.

	(2) Secondly, we prove that $(Z_t)_{t\ge0}$ is continuous in probability.
	Notice that for any $0\le s<t<t_0$ and $\vv>0,$
	\begin{align*}
		\P(|Z_t-Z_s|\ge\vv)
		&=\P\Big(\Big|\int_{\varphi_s}^{\varphi_t}g_{\varphi^{-1}_u}\,\d L_u\Big|\ge\vv\Big)\\
		&\le \P\Big(\Big|\int_{\varphi_s}^{\varphi_t}\int_{\{|z|\le1\}}g_{\varphi^{-1}_u} z\,\tilde N(\d z,\d u)\Big| \ge\vv/2\Big)\\
		&\quad+\P\Big(\Big|\int_{\varphi_s}^{\varphi_t}\int_{\{|z|>1\}}g_{\varphi^{-1}_u} z\, N(\d z,\d u) \Big|\ge\vv/2\Big),
	\end{align*}
	where the inequality holds true due to the L\'{e}vy-It\^o decomposition of the pure jump L\'evy process $(L_t)_{t\ge0}$ as used in the proof of Proposition \ref{lem}. Once again, by mimicking the trick   to deduce \eqref{Q8},  we conclude that $\lim_{s\uparrow t}\P(|Z_t-Z_s|\ge\vv)=0$ for any $\vv>0$ and
$t>0$, so $(Z_t)_{t\ge0}$ is continuous in probability. The proof is complete.
\end{proof}

\subsection{Uniform  moment estimates}\label{Section3-2}
In  Theorems \ref{thm1} and \ref{thm3}, note that there are error terms
arising from the difference of  drift terms; see, in particular,  \eqref{R4} and \eqref{R8} for more details. To handle the error terms mentioned previously,
it requires to establish a uniform-in-time moment estimate for $(Y_t)_{t\ge0}$ solving \eqref{E1-}, which is stated in the  proposition below.

\begin{proposition}\label{lem3} Let $(Y_t)_{t\ge0}$ be the strong solution to the SDE \eqref{E1-}. Assume that  there exist constants $C_*,C^*>0$ and a continuous function $\psi:[0,\infty)\to[0,\infty)$ satisfying $\lim_{t\to\infty}\psi(t)=0$ such that for all $t\ge0$ and $ x\in\R^d,$
	\begin{align}\label{D14}
		\<x,b(t,x)\>\le -(C_*-\psi(t))|x|^2+C^*.
	\end{align}
	Suppose further that for some  $p>0$ $($which might be greater than $2$$)$,
	\begin{align}\label{EE5}
		\sup_{t\ge0}\bigg(\int_{\{|z|\le1\}}   |z|^2 \,\nu_t(\d z)\bigg)+\sup_{t\ge0}\bigg(\int_{\{|z|>1\}}   |z|^p    \,\nu_t(\d z)\bigg)<\8.
	\end{align}
	Then, there exists a constant  $C_0>0$ such that for all $t\ge0$,
	\begin{align}\label{EE3}
		\E( |Y_t|^p|\mathscr F_{0})\le \e^{-p\int_0^t(C_*/2-\psi(s))\,\d s}(1+|Y_0|^2)^{p/2}+C_0\int_0^t\e^{-p\int_s^t(C_*/2-\psi(u))\,\d u}\,\d s.
	\end{align}
\end{proposition}

\begin{proof} Let $(\mathscr L_t)_{t\ge0}$ be the infinitesimal generator of the process $(Y_t)_{t\ge0}$, which is given in \eqref{EE1}.
	For given $p>0,$
	let
	$
	\mathcal V_p(x)=(1+|x|^2)^{ p/2} $ for all $x\in\R^d.$
	It is easy to see from \eqref{EE1} that
	for all $t\ge 0$ and $x\in\R^d,$
	\begin{align*}
		(\mathscr L_t\mathcal V_p)(x)&= \<\nn \mathcal V_p(x), b(t,x)\> +\int_{\{|z|\le1\}}\big(\mathcal V_p(x+z) -\mathcal V_p(x) - \<\nn \mathcal V_p(x),z\> \big)\,\nu_t(\d z)\\
		&\quad+ \int_{\{|z|>1\}}\big(\mathcal V_p(x+z) -\mathcal V_p(x) \big)\,\nu_t(\d z)\\
		&=:J_1(t,x)+J_2(t,x)+J_3(t,x).
	\end{align*}

	Provided that we can claim that there is a constant $C_1>0$ such that for all $t\ge 0 $ and   $x\in\R^d$,
	\begin{align}\label{EE4}
		(\mathscr L_t\mathcal V_p)(x)\le -p(C_*/2-\psi(t))\mathcal V_p(x)+C_1,
	\end{align}
	the assertion \eqref{EE3} follows   from It\^o's formula followed by Gronwall's inequality.

	By virtue of \eqref{D14} and  $\nn \mathcal V_p(x)=p(1+|x|^2)^{p/2-1}x$ for all $x\in\R^d,$
	together with  the assumption that $\lim_{t\to\infty}\psi(t)=0$,
	there exists a   constant   $C_2 >0 $ such that  for all $t\ge 0 $ and   $x\in\R^d$,
	\begin{align*}
		J_1(t,x)\le -p(3C_*/4-\psi(t))\mathcal V_p(x)+C_2.
	\end{align*}
	Next, by using the fact that for all $x\in\R^d,$
	\begin{align*}
		\nn^2 \mathcal V_p(x)=p(1+|x|^2)^{p/2-1}I_d+p(p-2)(1+|x|^2)^{p/2-2}x\otimes x,
	\end{align*}
	we find from Young's inequality   that  for some constant $C_3 >0 $ and all $t\ge0$ and $x\in\R^d,$
	\begin{align*}
		J_2(t,x)&= \frac{1}{2}\int_{\{|z|\le1\}}\int_0^1\int_0^s  \<\nn^2\mathcal V_p(x+uz)z,z\> \big)\,\d u\,\d s\,\nu_t(\d z)\\
		&\le \frac{1}{2}p(1+(p-2)^+)\int_{\{|z|\le1\}}\int_0^1\int_0^s  (1+|x+uz|^2)^{p/2-1} |z|^2\,\d u\,\d s\,\nu_t(\d z)\\
		&\le \begin{cases}
			\frac{1}{4}p\int_{\{|z|\le1\}}  |z|^2 \,\nu_t(\d z),\quad & p\in(0,2],\\
\frac{1}{4} p(p-1)(
3+2|x|^2
)^{p/2-1}\int_{\{|z|\le1\}}  |z|^2 \,\nu_t(\d z),\quad & p>2,
		\end{cases}\\
		&\le \begin{cases}
			\frac{1}{4}p\int_{\{|z|\le1\}}  |z|^2\, \nu_t(\d z),\quad &  p\in(0,2],\\
			\frac{1}{8}pC_*\mathcal V_p(x)+C_3\Big(\int_{\{|z|\le1\}}  |z|^2 \,\nu_t(\d z)\Big)^{p/2},\quad & p>2.
		\end{cases}
	\end{align*}
	Furthermore, there exists a constant $C_4>0$ such that for all $x\in\R^d$ and  $z\in\R^d$ with $|z|>1,$
	\begin{align*}
		\mathcal V_p(x+z) -\mathcal V_p(x)&\le
		\begin{cases}
			(2|x||z|+|z|^2)^{p/2},\quad & p\in(0,2],\\
			p\int_0^1(1+|x+sz|^2)^{(p-1)/2}|z|\,\d s,\quad & p>2,
		\end{cases}\\
		&\le
		\begin{cases}
			C_4\mathcal V_p(x)^{1/2} |z|^p,\quad& p\in(0,2],\\
			C_4 \mathcal V_p(x)^{1-1/{p}} |z|^p,  \quad & p>2,
		\end{cases}
	\end{align*}
	where,  in the first inequality,  we exploited the basic inequality:
	$(a+b)^\kk\le a^\kk+b^\kk$ for $a,b\ge0$ and $\kk\in(0,1].$ Subsequently,   Young's inequality implies that for some constant $C_5>0$,
	\begin{align*}
		J_3(t,x)\le \frac{1}{8}pC_*\mathcal V_p(x)+C_5\bigg(\int_{\{|z|>1\}}  |z|^p\,\nu_t(\d z)\bigg)^{2\vee p},\quad \forall\, t\ge0, \forall\, x\in\R^d.
	\end{align*}
As a result,  \eqref{EE4} is available by taking the estimates on $J_1 (t,x)$, $J_2(t,x) $ and $J_3(t,x) $ into consideration and   making use of \eqref{EE5}.
\end{proof}

\section{General asymptotic theory and Proofs of Theorem \ref{theorem-1}
	and Theorem \ref{thm2}}\label{sec4}

\subsection{General asymptotic theory}	
To investigate the long-time behaviors of the process $(X_t)_{t\ge0}$ solving the SDE  \eqref{eq1}, we will adopt the space-time change approach. To this end,
let $\varphi:[0,\8)\to[0,\8)$ be an increasing $C^1$-function, and $g:[0,\8)\to (0,1]$ be a decreasing $C^1$-function satisfying $\lim_{t\to\infty}g_t=0$, which are reference functions for the space and time transformations, respectively. By applying It\^o's formula, the time-space  changed process $(Y_t)_{t>0}:=(g_tX_{\varphi_t})_{t>0} $ solves the subsequent SDE: for all $t\ge0,$
\begin{equation}\label{RR-}
	\d Y_t
	=\big({g'_t}
	Y_t/g_t+ g_t  \varphi_t' F(\varphi_t, Y_t/{g_t}) \big)\,\d t+\d Z_t\quad\mbox{ with } \quad Z_t:=\int_0^tg_s\d L_{\varphi_s}.
\end{equation}
To study   the asymptotics of  $(Y_t)_{t\ge0}:=(g_tX_{\varphi_t})_{t\ge0} $ solving the SDE \eqref{RR-}, we choose  the functions $\varphi$ and $g$ such  that  $\varphi_t'g_t^\alpha=1$
for all $t\ge0$
  in case  the
index of large jumps in Assumption $({\bf H}_1)$ satisfies $\alpha\in(0,2)$. On the one hand, we  take $g_t=\e^{-t/\alpha}$ and $\varphi_t=\e^t$, which obviously satisfy $\varphi_t'g_t^\alpha=1$ for all $t\ge0$. In this case,
by virtue  of \eqref{RR-},  $(Y_t)_{t\ge0} $ solves the following  SDE: for all $t\ge0,$
\begin{align}\label{UP-2}
	\d Y_t=\big(-Y_t/\alpha+\e^{-(1/\alpha-1)t}F(\e^t,\e^{t/\alpha}Y_t)\big)\,\d t+\d Z_t.
\end{align}
On  the other hand, we  choose for some constant $\theta\in(0,1)$,
\begin{align}\label{UP-1}
	g_t=(1+(1-\theta)t)^{-\frac{\theta}{\alpha(1-\theta)}} \quad \mbox{ and } \quad  \varphi_t=(1+(1-\theta)t)^{\frac{1}{ 1-\theta }},\quad \forall\,t\ge0,
\end{align}
which also fulfil $\varphi_t'g_t^\alpha=1$ for all $t\ge0.$
Correspondingly, by invoking  \eqref{RR-} once again,
the associated $(Y_t)_{t\ge0} $ is governed by the  SDE below: for all $t\ge0,$
\begin{align}\label{UP-3}
	\d Y_t= \Big(- \frac{\theta Y_t}{\alpha
		(1+(1-\theta)t)  }  +  g_t^{1-\alpha} F (\varphi_t,Y_t/{g_t} )\Big)\,\d t+ \d  Z_t.
\end{align}

To explore the asymptotic behavior of the space-rescaled process related to $(X_t)_{t\ge0}$, as an intermediate step, we need to establish the asymptotics of
 $(Y_t)_{t\ge0}$ solving the SDE \eqref{UP-2} and the SDE \eqref{UP-3}, respectively. To this end,   we suppose respectively that  the following assumptions are  satisfied:
\begin{enumerate}

\item[$({\bf H}_3)$]\it
 the limit $F_1 (x):=\lim_{t\to\infty}(\e^{-(1/\alpha-1)t}F(\e^t,\e^{t/\alpha} x) ) $  exists, and
there exist constants  $K_1,K_{2}>0$, $K_{3}\ge0$,
 $0\le\gamma_1<\alpha/{(1\vee\theta_2)}$,  decreasing  continuous  functions  $\phi_1,\psi_1  :[0,\infty)\to[0,\infty)$ satisfying $\lim_{t\to\infty} \phi_1 (t)=\lim_{t\to\infty} \psi_1 (t) =0$ such that
for all $t\ge0$ and  $x,y\in\R^d $,
	\begin{flalign}
	 \<x-y,-(x-y)/\alpha+F_1 (x)-F_1 (y)\>&\le -K_1 |x-y|^2,\label{TY-3}
		\\
	 \big|\e^{-(1/\alpha-1)t}F(\e^t,\e^{t/\alpha} x)-F_1 (x)\big|&\le  \phi_1 (t)(1+|x|^{\gamma_1}), \label{TY-4}
		\\
	\<x,-x/\alpha+\e^{-(1/\alpha-1)t}F(\e^t,\e^{t/\alpha} x)\>&\le-(K_2-\psi_1 (t))|x|^2+K_3;\label{TY-5}
	\end{flalign}
\item[$({\bf H}_4)$] 	 the limit
	$F_2 (x):=\lim_{t\to\infty} (g_t^{1-\alpha} F (\varphi_t,x/{g_t} ) )$ exists for  $\varphi_t$ and $g_t $ defined  in \eqref{UP-1},  and  there exist constants  $L_1,L_2>0$, $L_3\ge0$,  $\kk\ge1,$
	$0\le\gamma_2<\alpha/{(1\vee\theta_2)}$,  decreasing continuous   functions  $\phi_2,\psi_2 :[0,\infty)\to[0,\infty)$ fulfilling $\lim_{t\to\infty} \phi_2 (t) = \lim_{t\to\infty} \psi_2 (t) =0$ such that
	for all $t\ge0$ and $x,y\in\R^d $,
	\begin{align}
\<x-y, F_2(x)-F_2 (y)\>&\le -L_1 |x-y|^{1+\kk},\label{TY-6}
		\\
\big|g_t^{1-\alpha} F (\varphi_t,x/{g_t} )-F_2 (x)\big|&\le  \phi_2  (t)(1+|x|^{\gamma_2}),\label{TY-7}
		\\
  g_t^{1-\alpha}\<x,  F (\varphi_t,x/{g_t} )\>&\le-(L_2-\psi_2 (t))|x|^2+L_3.\label{TY-8}
	\end{align}
\end{enumerate}

Obviously, in case of $\beta>1+(\gamma-1)/\alpha$, the function
$$F(t,x)=t^{-\beta}\Big(-|x|^{\gamma-1}+\frac{1}{1+|x|^2}\Big)x,\quad \forall\, t\ge0, x\in\R^d,$$
satisfies $({\bf H}_3)$ with $F_1(x)=0$. In addition, as long as $\beta=1+(\gamma-1)/\alpha$ with $\gamma>1$, the function
$$F(t,x)=-t^{-\beta}|x|^{\gamma-1}x+t^{-\beta}x+b(x),\quad \forall\,  t\ge0, x\in\R^d,$$
where $b:\R^d\to\R^d$ is bounded, fulfils  $({\bf H}_3)$ with $F_1(x)=-x|x|^{\gamma-1}.$
Moreover, concerning
the case $\beta<1+(\gamma-1)/\alpha$ with $\gamma>1$, the function
\begin{align*}
F(t,x)=- t^{-\beta}x|x|^{\gamma-1}+\frac{t^{-s}|x|^{\kk_*-1}x}{1+|x|^2},\quad \forall\,  t\ge0,x\in\R^d
\end{align*}
obeys  $({\bf H}_4)$ with $F_2(x)=-x|x|^{\gamma-1}$
 for some constants $\kk_*\in(0,2 )$ and  $s>\beta(\alpha-1)/{(\alpha+\gamma-1)}$.

In contrast to Theorem \ref{theorem-1} and Theorem \ref{thm2},
which are concerned with a typical candidate of $F(t,x)$, in this section
we establish the respective counterparts for a general $F(t,x)$ when the limit function of the time-space rescaled version satisfies suitable conditions (see $({\bf H}_3)$ and $({\bf H}_4)$ for more details).

\begin{theorem}\label{thm0}
	Assume that $({\bf H}_1)$,   $1\le \alpha<\theta_1\le2 $, $ \theta_2\in(0,\alpha)$ and $ \int_{\{|z|\le 1\}} |z|^{\theta_1}  a( z )\,\d z<\infty$ hold.
	 Then, under Assumption $({\bf H}_3)$, for all $t\ge1$ and  $\mu\in\mathcal P_{\theta_2\vee(\gamma_1(1\vee\theta_2))}(\R^d)$,
	\begin{equation}\label{TY}
		\mathcal W_{|\cdot|^{\theta_1}\wedge|\cdot|^{\theta_2}}\Big(\mathscr L_{t^{- {1}/{\alpha}}X_t^{\mu}},\pi\Big)  \le C_1^* \big(\phi_1((\ln t)/2)^{1\vee\theta_2}\vee t^{-( ((\theta_1/\alpha-1)/2)\wedge  (K_1\theta_2/{32}) ) } \big),
	\end{equation}
	where $\pi\in\mathcal P_{\theta_2}(\R^d)$ is the unique IPM of   $(\overline Y_{\!\! t})_{t\ge0}$ solving  the  subsequent SDE: for all $t\ge0,$
	\begin{equation}\label{EP}
		\d  \overline{Y}_{\!\! t} =
			(-  \overline{Y}_{\!\! t}  /\alpha + F_1( \overline{Y}_{\!\! t}   )  )\,\d t+ \d   L_t^{(\alpha)},
	\end{equation}
and
	$C_1^*>0$ is a constant depending   on $\mu $ and $\pi $; under Assumption  $({\bf H}_4)$,
	for  all $t\ge1$ and
	 $\mu\in \mathcal P_{(1\vee\gamma_2)(1\vee\theta_2)}(\R^d)$,
	\begin{equation}\label{EY}
		\begin{split}
	&	\mathcal W_{|\cdot|^{\theta_1}\wedge|\cdot|^{\theta_2}}\Big(\mathscr L_{t^{-\theta/\alpha}X_t^{\mu}},\pi\Big)  \\
	&\le C_2^*
\begin{cases}		
		 (t^{\theta-1}\vee\phi_2(\varphi_t^{-1}/ 2 ) )^{1\vee\theta_2}\vee t^{-  \theta(\theta_1/\alpha-1) }
		,&\kk=1,\\
		t^{-\theta_1(1-\theta )/{(\kk-1)}}\vee\big(t^{-\theta(\theta_1/\alpha-1)} +(t^{ \theta-1}\vee\phi_2(\varphi_t^{-1} /2 ) )^{1\vee\theta_2}\big)^{1/{(1+(\kk-1)/{\theta_1})}},&\kk>1,
		\end{cases}
		\end{split}
	\end{equation}
	where  $\pi\in \mathcal P_{\theta_2}(\R^d)$ is the unique IPM of $(\overline Y_{\!\! t})_{t\ge0}$  determined by the  SDE below:  for all $t\ge0,$
\begin{equation} \label{EP-1}
	\d  \overline{Y}_{\!\! t} =
  F_2( \overline{Y}_{\!\! t}   )   \,\d t+ \d   L_t^{(\alpha)},
\end{equation}	
 $\varphi_t^{-1}=(t^{1-\theta}-1)/( 1-\theta) $,
 and the constant
	$C_2^*>0$ depends on $\mu$ and $\pi$.
\end{theorem}

\begin{theorem}\label{thm2-1}
	Assume that $({\bf H}_1)$,
	$\int_{\{|z|\le1\}}|z|a(z)\,\d z<\infty$ and  $\beta\ge 1+ (\gamma-1)/\alpha$ hold
	for
	$0<\gamma<\alpha<1$.
	Suppose further that the following $({\bf H}_3')$
is satisfied:
\begin{enumerate}	
		\item[$({\bf H}_3')$]$({\bf H}_3)$  	holds true for some $ \gamma_1\in[0,\alpha)$, where  \eqref{TY-3} is replaced by the following condition: for some constants $\lambda_1,\lambda_2>0$,
		$\alpha_*\in(0,1]$,
		$\ell_0\ge1$ and  all $x,y\in\R^d,$
		\begin{align}\label{EP-24}
			\<x-y,-(x-y)/\alpha+F_1 (x)-F_1 (y)\>\le \lambda_1\I_{\{|x-y|\le\ell_0\}}|x-y|^{1+\alpha_*}-\lambda_2 \I_{\{|x-y|>\ell_0\}}|x-y|^2.
		\end{align}	
	\end{enumerate}	
	 Then,
	there
	is a constant $\lambda^*>0$
	such that for any $t\ge1$ and $\mu\in\mathcal P_{\gamma\vee\gamma_1}(\R^d)$,
	\begin{equation}\label{P10}
		\mathcal W_{|\cdot|\wedge|\cdot|^\gamma}\Big(\mathscr L_{t^{-{1}/{\alpha}}X_t^{\mu}},\pi\Big)   \le    C^* t^{-\lambda^*},
	\end{equation}
	where $\pi\in\mathcal P_\gamma(\R^d)$ is the unique IPM  of
	$(\overline{Y}_{\!\! t})_{t\ge0}$ solving   the  SDE in \eqref{EP},
	and $C^*>0$ depends linearly  on $\mu(|\cdot|^{\gamma\vee\gamma_1})$ and $\pi(|\cdot|^\gamma)$.
 \end{theorem}

\subsection{Proofs of Theorem \ref{theorem-1} and Theorem \ref{thm0}}
Before we proceed to prove Theorem \ref{thm0}, we begin with some preliminaries for later use.  In the following analysis, the increasing $C^1$-function $\varphi:[0,\infty)\to[0,\infty)$ and the decreasing $C^1$-function $g:[0,\infty)\to(0,1]$ is postulated to satisfy $\varphi_t'g_t^\alpha=1$ for all $t\ge0,$  where their concrete expressions  are unimportant for the moment. In terms of Proposition \ref{lem} and Proposition  \ref{lem2},
the driven noise $(Z_t)_{t\ge0}$  defined  in \eqref{RR-}
is an additive process with the (time-dependent) L\'{e}vy measure $(\nu_t(\d z))_{t\ge0}$ being given as follows: for all $t\ge0, $
\begin{align}\label{EW3}
	\nu_t (\d z)  =\varphi_t'(\nu\circ (\bar g_t)^{-1})(\d z)  =g_t^{-d-\alpha}a(z/{g_t})\I_{\{| z/{g_t}|\le1\}}\,\d z+\frac{c_*}{|z|^{d+\alpha}}\I_{\{| z/{g_t}|>1\}}\,\d z,
\end{align} where in the second identity  we used  the precondition  $\varphi_t'g_t^\alpha=1$ for all $t\ge0.$ Here, we also used the fact that for all $t\ge0,$
$$\int_{\R^d}  z\I_{\{g_t<|z|\le 1\}} \, \nu_t (\d z)=c_*g_t^{1-\alpha}\int_{\{1<| z|\le 1/{g_t} \}}   \frac{z}{|z|^{d+\alpha}} \,\d z=0.  $$
Obviously,    $(\nu_t(\d z))_{t\ge0}$ is symmetric by invoking   $a(-z)=a(z)$ for all $z\in\bar B_1.$

Next, under the condition that   $ \displaystyle\int_{\{|z|\le 1\}} |z|^{\theta_1} a(z)\,\d z<\infty$ for some $ \theta_1\in(\alpha,2] $,
 we claim that  for all $\theta_2\in(0,\alpha)$  and $t\ge0,$
\begin{equation}\label{P}\begin{split}
		H(t):=&\int_{\R^d}\big(|z|^{\theta_1}\I_{\{|z|\le1\}}+|z|^{\theta_2}\I_{\{|z|>1\}}\I_{\{\theta_2\in(0,1]\}}\big)
		|\nu_t-\nu^{(\alpha)}|(\d z)\\
		&\quad +\left(\int_{\{|z|>1\}}|z|\,|\nu_t-\nu^{(\alpha)}|(\d z)\right)^{\theta_2}\I_{\{\theta_2>1\}}\\
		=&C^\star g_t^{\theta_1-\alpha}<\infty,
\end{split}\end{equation}
where the quantity $C^\star>0$ is defined as below:
$$C^\star:=\int_{\{|z|\le1\}} |z|^{\theta_1} \Big|a( z)-\frac{c_*}{|z|^{d+\alpha}}\Big| \,\d z<\infty.$$
The estimate above   (which corresponds to the prerequisite \eqref{EE6}  in Theorem \ref{thm1})
is necessary when Theorem \ref{thm1} is to be  applied in the subsequent analysis.
Indeed, due to \eqref{EW3} and
$\nu^{(\alpha)}(\d z)=\frac{c_*}{|z|^{d+\alpha}}\,\d z$,  it follows that for all $\theta_2\in(0,\alpha)$ and  $t\ge0$,
\begin{equation}\label{P8}
	\begin{split}
	&\int_{\R^d}\big(|z|^{\theta_1}\I_{\{|z|\le1\}}+|z|^{\theta_2}\I_{\{|z|>1\}}\I_{\{\theta_2\in(0,1]\}}\big)
	|\nu_t-\nu^{(\alpha)}|(\d z)\\
		 &=\int_{\R^d}\big(|z|^{\theta_1}\I_{\{|z|\le1\}}+|z|^{\theta_2}\I_{\{|z|>1\}}\I_{\{\theta_2\in(0,1]\}}\big)
		 \Big|g_t^{-d-\alpha}a(z/{g_t})-\frac{c_*}{|z|^{d+\alpha}}\Big|\I_{\{|z/{g_t}|\le1\}}\,\d z\\
		&= g_t^{ -\alpha}\int_{\{| z|\le1\}}\big((g_t|z|)^{\theta_1}\I_{\{|g_tz|\le 1\}}+(g_t|z|)^{\theta_2}\I_{\{|g_tz|>1\}}\I_{\{\theta_2\in(0,1]\}}\big)\Big|a( z)-\frac{c_*}{|z|^{d+\alpha}}\Big| \,\d z\\
		&= g_t^{ \theta_1-\alpha}\int_{\{| z|\le1\}} |z|^{\theta_1} \Big|a( z)-\frac{c_*}{|z|^{d+\alpha}}\Big| \,\d z\\
		&=C^\star g_t^{ \theta_1-\alpha}.
	\end{split}
\end{equation}
Whence, \eqref{P} follows from the fact  that  $$\int_{\{|z|>1\}}|z|\,|\nu_t-\nu^{(\alpha)}|(\d z)=g_t^{1-\alpha}\int_{\{|g_tz|>1\}}\Big|a( z)-\frac{c_*}{|z|^{d+\alpha}}\Big|\I_{\{|z|\le1\}}\d z=0$$
and   $C^\star<\infty$  by taking $ \theta_1\in(\alpha,2] $ and $ \displaystyle\int_{\{|z|\le 1\}} |z|^{\theta_1} a(z)\,\d z<\infty$ into account.

Whereafter, for $(\nu_t(\d z))_{t\ge0}$ given in \eqref{EW3},
 we move forward to prove the following statement that for any $p\in(0,\alpha),$
\begin{align}\label{P3}
		\sup_{t\ge0}\bigg(\int_{\{|z|\le1\}}   |z|^2\, \nu_t(\d z)\bigg)+\sup_{t\ge0}\bigg(\int_{\{|z|>1\}}   |z|^p    \,\nu_t(\d z)\bigg)<\8,
\end{align}
which is indispensable when   Proposition \ref{lem3} is to be applied. To achieve \eqref{P3},
on the one hand, by exploiting  \eqref{EW3}, along with $g_\cdot\in(0,1]$ and $\int_{\{| z|\le 1\}}|z|^2 a( z) \,\d z<\8$ (which is based on the precondition $\int_{\{| z|\le 1\}}|z|^{\theta_1} a( z) \,\d z<\8$ for $ \theta_1\in(\alpha, 2]$), we find that for all $t\ge0,$
\begin{align*}
	\int_{\{|z|\le1\}}|z|^2\,\nu_t(\d z) &= g_t^{2-\alpha}\bigg(\int_{\{| z|\le 1\}}|z|^2 a( z) \,\d z+c_*\int_{\{1<|z|\le1/{g_t}  \}}\frac{|z|^2}{|z|^{d+\alpha}} \,\d z\bigg) \\
	&\le \int_{\{| z|\le 1\}}|z|^2 a( z) \,\d z+\frac{c_*d\omega_d}{2-\alpha}.
\end{align*}
On  the other hand, we have that  for any $p\in(0,\alpha)$ and $t\ge0$,
$$
\int_{\{|z|>1\}}   |z|^p    \,\nu_t(\d z) =c_*g_t^{p-\alpha}\int_{\{|z|>1/{g_t}\}}\frac{|z|^p}{|z|^{d+\alpha}}\,\d z=\frac{c_*d\omega_d}{\alpha-p} .
$$
Therefore, for any $p\in(0,\alpha)$, \eqref{P3} follows  by combining with the  two estimates  above.

With the above preparations in place, we turn to the

\begin{proof}[Proof of Theorem   $\ref{thm0}$]
In the sequel, let $(X_t)_{t\ge0}$ and $(\overline{Y}_{\!\! t})_{t\ge0}$ be the solutions to  the SDEs \eqref{eq1}
and \eqref{EP} (resp. \eqref{EP-1}), respectively. Via the triangle inequality, it is easy to see that
 for  $p>0$,
 any function $\psi:[
 t_0,\infty)\to[0,\infty)$, and $t\ge t_0$
 with some $t_0>0$,
 \begin{align}\label{WY-}
\mathcal W_{|\cdot|^{\theta_1}\wedge|\cdot|^{\theta_2}}\Big(\mathscr L_{t^{-p}X_t^{\mu}},\pi\Big)\le \mathcal W_{|\cdot|^{\theta_1}\wedge|\cdot|^{\theta_2}}\Big(\mathscr L_{t^{-p}X_t^{\mu}},\mathscr L_{\overline Y_{\!\! \psi_t}^{\mu}}\Big)+\mathcal W_{|\cdot|^{\theta_1}\wedge|\cdot|^{\theta_2}}\Big( \mathscr L_{\overline Y_{\!\! \psi_t}^{\mu}},\pi\Big),
\end{align}
where $\pi$ is the  unique IPM of the SDE   \eqref{EP} (resp. \eqref{EP-1}). Below, we write
 $F^*(x)$ as the corresponding drift  of the SDE
\eqref{EP}   (resp. \eqref{EP-1}).  No matter which of the  SDE   \eqref{EP} or \eqref{EP-1}, by invoking \eqref{TY-3} or  \eqref{TY-6},
 we find  that for all $x,y\in\R^d,$
\begin{align}\label{LW-2}
\<x-y,F^*(x)-F^*(x)\>\le-K_*|x-y|^{1+\kk_*},
\end{align}
where $\kk_*=1$ and $K_*=K_1$ (resp. $\kk_*=\kk$ and $K_*=L_1$)
 as far as the  SDE    \eqref{EP} (resp. \eqref{EP-1}) is concerned.
Then, for  the function $h:[0,\infty)\to[0,\infty)$,  defined in \eqref{EW11},  we deduce from the chain rule and \eqref{EW--} (with $\kk$ therein being replaced by $\kk_*$) that for all $t\ge0 $,
\begin{align*}
\d h(|\overline Y_{\!\! t}^{\mu}-\overline Y_{\!\! t}^{\pi}|^{\theta_1})&=
\theta_1
h'(|\overline Y_{\!\! t}^{\mu}-\overline Y_{\!\! t}^{\pi}|^{\theta_1})|\overline Y_{\!\! t}^{\mu}-\overline Y_{\!\! t}^{\pi}|^{\theta_1-2}\<\overline Y_{\!\! t}^{\mu}-\overline Y_{\!\! t}^{\pi},F^*(\overline Y_{\!\! t}^{\mu})-F^*(\overline Y_{\!\! t}^{\pi})\>\,\d t\\
&\le -
\theta_1
K_* h'(|\overline Y_{\!\! t}^{\mu}-\overline Y_{\!\! t}^{\pi}|^{\theta_1})|\overline Y_{\!\! t}^{\mu_1}-\overline Y_{\!\! t}^{\pi}|^{\theta_1+\kk_*-1}\,\d t\\
&\le -K^*h(|\overline Y_{\!\! t}^{\mu}-\overline Y_{\!\! t}^{\pi}|^{\theta_1})^{1+(\kk_*-1)/{\theta_1}}\,\d t,
\end{align*}
where $K^*:=\theta_1
K_*(\theta_2/{(2\theta_1)})^{1+(\kk_*-1)/{\theta_2}}$. The previous estimate,
along with Lemma \ref{lemma-*},
 implies that
 for all $t\ge0,$
\begin{align*}
 \E h(|\overline Y_{\!\! t}^{\mu}-\overline Y_{\!\! t}^{\pi}|^{\theta_1})\le
 \begin{cases}
\e^{-K^*t}\E h(|\overline Y_{\!\! 0}^{\mu}-\overline Y_{\!\! 0}^{\pi}|^{\theta_1})&\kk_*=1,\\
C_* t^{-\frac{\theta_1}{\kk_*-1}},&\kk_*>1,
	\end{cases}
\end{align*}
where $C_*>0$ is independent of $t$.
 Subsequently, by  choosing the variable $(\overline Y_{\!\! 0}^{\mu},\overline Y_{\!\! 0}^{\pi} )$ such that  $$\E  \big(|\overline Y_{\!\! 0}^{\mu}-\overline Y_{\!\! 0}^{\pi}|^{\theta_1} \wedge|\overline Y_{\!\! 0}^{\mu}-\overline Y_{\!\! 0}^{\pi}|^{\theta_2}\big)=\mathcal W_{|\cdot|^{\theta_1}\wedge|\cdot|^{\theta_2}}(\mu,\pi),$$
 we obtain from  \eqref{EW-*} and  the invariance of $\pi$  that  there exists  a constant $C_0>0$ such that
 for all $t>0,$
\begin{equation}\label{WY--}
	\begin{split}
\mathcal W_{|\cdot|^{\theta_1}\wedge|\cdot|^{\theta_2}}\Big( \mathscr L_{\overline Y_{\!\!  t}^{\mu}},\pi\Big)
&\le  C_0
\begin{cases}
\e^{-K^*t}\mathcal W_{|\cdot|^{\theta_1}\wedge|\cdot|^{\theta_2}}(\mu,\pi),&\kk_*=1,\\
t^{-\frac{\theta_1}{\kk_*-1}},&\kk_*>1.
	\end{cases}
	\end{split}
\end{equation}

Based on the foregoing preparations, 	
we first prove Theorem \ref{thm0}	under Hypothesis $({\bf H}_3)$.  By employing  \eqref{TY-5}, \eqref{P3} as well as Proposition \ref{lem3},
for $q\in(0,\alpha)$ there is a constant
$C_1>0$
such that for all
$t\ge0$ and $\mu\in \mathcal P_{q}(\R^d)$,
\begin{align}\label{EP-3}
	\E|Y_t^{\mu}|^q \le
	C_1\big(1+ \mu(|\cdot|^q) \big),
\end{align}
where $(Y_t^\mu)_{t\ge0}$ solves the SDE \eqref{UP-2}.
As a result,
we deduce from \eqref{TY-3}, \eqref{TY-4}, \eqref{P} with $g_t=\e^{-t/\alpha}$ and  Theorem \ref{thm1} that
for all
$t\ge0$, $\mu_1\in \mathcal P_{\theta_2\vee(\gamma_1(1\vee \theta_2))}(\R^d)$ and $\mu_2\in\mathcal P_{\theta_2}(\R^d)$,
\begin{align*}
	\mathcal W_{|\cdot|^{\theta_1}\wedge|\cdot|^{\theta_2}}\Big(\mathscr L_{Y_{ t}^{\mu_1}},\mathscr L_{\overline {Y}^{\mu_2}_{\! \! t}}\Big) &\le
	C_2\e^{-\lambda t}\mathcal W_{|\cdot|^{\theta_1}\wedge|\cdot|^{\theta_2}}(\mu_1,\mu_2)\\
	&\quad +
	C_3\int_0^t\e^{-\lambda (t-s)}\big  (      \phi_1(s)^{1\vee\theta_2}+   \e^{-(\theta_1 /\alpha-1)s}  \big )\,\d s\\
	&\le C_2\e^{-\lambda t}\mathcal W_{|\cdot|^{\theta_1}\wedge|\cdot|^{\theta_2}}(\mu_1,\mu_2)+C_4\big(\phi_1(t/2)^{1\vee\theta_2}+\e^{-(\theta_1/\alpha-1)t/2} +t\e^{-\lambda  t/2}\big).
\end{align*}
Herein,  $(\overline {Y}_{\!\! t})_{t\ge0}$ is determined by the   SDE   \eqref{EP},
$\lambda :=2^{\theta_2/{\theta_1}-4}K_1  \theta_2
$,
in the second inequality we exploited the fact that for any $\kk_0>0$ and  decreasing continuous function $\bar\phi:[0,\infty)\to[0,\infty)$,
\begin{equation}\label{EP-7}
	\begin{split}	
	\int_0^t\e^{-\kk_0 s}\bar\phi(t-s)\,\d s& =\int_0^{t/2}\e^{-\kk_0 s}\bar\phi(t-s)\,\d s+\int_{t/2}^t\e^{-\kk_0 s}\bar\phi(t-s)\,\d s\\
	& \le \frac{1}{\kk_0}\bar\phi(t/2)+\frac{1}{2}\|\bar\phi\|_\8t\e^{-\kk_0 t/2},
	\end{split}
\end{equation}
and $C_2,C_3, C_4$ are positive constants, where
$C_3$ and $C_4$  depend linearly on $\mu_1(|\cdot|^{\gamma_1(1\vee\theta_2)})$.
So, under Assumption $({\bf H}_3)$, the assertion \eqref{TY} is verifiable
by
putting \eqref{WY-} with $p=1/\alpha$ and $
\psi_t=\varphi_t^{-1}=\ln t$	and   \eqref{WY--} with $\kk_*=1$ and $t$ therein being replaced by $\ln t$
together,   noting that   $g_{\varphi_t^{-1}}=t^{-{1}/{\alpha}}$ for all $t\ge1$
as well as making use of the fact that $\sup_{r\ge0}(r\e^{-c t})\le 1/{(\e c)}$
for all $c>0.$

We then proceed to prove  Theorem \ref{thm0} under Assumption $({\bf H}_4)$. By virtue of \eqref{TY-8}, \eqref{P3} and Proposition \ref{lem3}, the uniform-in-time moment estimate \eqref{EP-3} is still valid for  $(Y_t^\mu)_{t\ge0}$ solving  the SDE \eqref{EP-1}. Moreover, note from \eqref{TY-7} that for all $t\ge0$ and $x\in\R^d,$
\begin{align*}
|-  \theta x /{(\alpha
	(1+(1-\theta)t) ) }  +  g_t^{1-\alpha} F (\varphi_t,x/{g_t} )-F_2 (x)|\le \theta |x| /{(\alpha
	(1+(1-\theta)t) ) }+\phi_2  (t)(1+|x|^{\gamma_2}).
\end{align*}
Thus, taking \eqref{TY-6} and \eqref{P} with  $g_t$  being given in \eqref{UP-1} into consideration and applying
Theorem \ref{thm1}   yield  that
for all
	$\mu_1\in \mathcal P_{(1\vee\gamma_2)(1\vee \theta_2)}(\R^d)$, $\mu_2\in\mathcal P_{\theta_2}(\R^d)$ and $t\ge0,$
	\begin{align*}
		\mathcal W_{|\cdot|^{\theta_1}\wedge|\cdot|^{\theta_2}}\Big(\mathscr L_{Y_{  t}^{\mu_1}},\mathscr L_{\overline {Y}_{\!\! t}^{\mu_2}}\Big) &\le
		\begin{cases} C_5\e^{-\lambda t}\mathcal W_{|\cdot|^{\theta_1}\wedge|\cdot|^{\theta_2}}(\mu_1,\mu_2)\\
		  \quad +C_6\int_0^t\e^{-\lambda (t-s)}    \big(((1/s)\vee\phi_2(s))^{1\vee\theta_2} +  g_s^{ \theta_1-\alpha}\big)   \,\d s,&\kk=1,\\
	  C_7\big(t^{-\theta_1/{(\kk-1)}}\vee   \big(g_{t/2}^{ \theta_1-\alpha}+((1/t)   \vee\phi_2(t/2))^{1\vee\theta_2}\big)^{1/{(1+(\kk-1)/{\theta_1})}}\big),&\kk>1
		\end{cases}
	\end{align*}
In the estimate above,  $(\overline {Y}_{\! \! t})_{t\ge0}$ solves  the  SDE   \eqref{EP-1}, $\lambda :=2^{\theta_2/{\theta_1}-4}L_1\theta_2$,
and $C_5,C_6,C_7>0$ are  constants, where $C_6$ and $C_7$ depend  linearly on  $\mu_1(|\cdot|^{(1\vee\gamma_2)(1\vee\theta_2)})$.
Consequently, by recalling that $g_t$ is given in \eqref{UP-1} and invoking \eqref{EP-7}, we find that
for all
$t\ge0$, $\mu_1\in \mathcal P_{ (1\vee\gamma_2) (1\vee \theta_2)}(\R^d)$ and $\mu_2\in\mathcal P_{\theta_2}(\R^d)$,
\begin{align*}
	&\mathcal W_{|\cdot|^{\theta_1}\wedge|\cdot|^{\theta_2}}\Big(\mathscr L_{Y_{  t}^{\mu_1}},\mathscr L_{\overline {Y}_{\!\! t}^{\mu_2}}\Big)\\ &\le
	\begin{cases} C_5\e^{-\lambda t}\mathcal W_{|\cdot|^{\theta_1}\wedge|\cdot|^{\theta_2}}(\mu_1,\mu_2)\\ \quad +C_8 \Big(((1/t)\vee\phi_2(t/2))^{1\vee\theta_2}+ (1+t)^{-\frac{\theta(\theta_1-\alpha)}{\alpha(1-\theta)}}+t\e^{-\lambda t/2}\Big),&\kk=1,\\
		 C_9\Big(t^{-\theta_1/{(\kk-1)}}\vee   \Big((1+t)^{-\frac{\theta(\theta_1-\alpha)}{\alpha(1-\theta)}}+   ((1/t)\vee\phi_2(t/2))^{1\vee\theta_2}\Big)^{1/{(1+(\kk-1)/{\theta_1})}}\Big),&\kk>1,
		\end{cases}
\end{align*}
where  $C_8,C_9$ are positive constants, whose dependence on $\mu_1$ is the same as that of $C_6,C_7$, respectively.
At length, under Assumption $({\bf H}_4)$, the desired assertion \eqref{EY} is available at once  by noticing that for all $t\ge1,$
	\begin{align*}
	 g_{\varphi_t^{-1}}=t^{-\theta/\alpha},\quad   \quad \varphi_t^{-1}=  (t^{1-\theta}-1)/{(1-\theta)},
	\end{align*}
and making use of \eqref{WY-} with $
\psi_t=\varphi_t^{-1}$	and \eqref{WY--}.
The proof is therefore complete.
\end{proof}

Subsequently, with the aid of Theorem   \ref{thm0}, we complete the
\begin{proof}[Proof of Theorem $\ref{theorem-1}$] To finish  the proof of Theorem \ref{theorem-1}, it suffices to verify the preconditions in Theorem   \ref{thm0} for the three cases,  separately.

Recall that $F_*(x)=\lambda x(1+|x|^2)^{(\gamma-1)/2}$.  It is easy to see  that for all $t\ge0$ and $x\in\R^d$,
\begin{align}\label{1-TY}
\e^{-(1/\alpha-1)t}F(\e^t,\e^{t/\alpha}x)=\lambda \e^{-(1-\gamma+\alpha(\beta-1))t/\alpha} ( 1+\e^{-t} )^{-\beta} x(\e^{-2t/\alpha}+| x|^2)^{(\gamma-1)/2},
\end{align}
and that  for $\theta=\frac{\alpha\beta}{\alpha+\gamma-1}\in(0,1)$, which is valid as long as
$\beta<1+(\gamma-1)/\alpha$,
\begin{equation}\label{EP-5}
g_t^{1-\alpha} F (\varphi_t,x/{g_t} )
 =\lambda \big(1+1/{\varphi_t}\big)^{-\beta} x \big(g_t^2+|x|^2\big)^{(\gamma-1)/2},
\end{equation}
where   $g_t $ and $\varphi_t$ were defined in \eqref{UP-1}.
Consequently, we deduce  that
for all $x\in\R^d,$
\begin{equation}\label{EP-4}
F_1(x)=\lim_{t\to\infty} \big(\e^{-(1/\alpha-1)t}F(\e^t,\e^{t/\alpha}x) \big)=
\begin{cases}
	{0}&\mbox{ if } \beta>1+(\gamma-1)/\alpha,\\
	\lambda x|x|^{\gamma-1}&\mbox{ if } \beta=1+(\gamma-1)/\alpha,
	\end{cases}
\end{equation}
and
\begin{align*}
F_2(x)=\lim_{t\to\infty}\big(g_t^{1-\alpha} F (\varphi_t,x/{g_t} )\big)=\lambda x|x|^{\gamma-1}\quad \mbox{ if }\beta<1+(\gamma-1)/\alpha.
\end{align*}

For case (I): $\beta>1+(\gamma-1)/\alpha$, by using $F_1(x)\equiv{0}$,
it is ready to see from \eqref{1-TY} that \eqref{TY-3} and $\eqref{TY-4}$ are satisfied    for $K_1=1/\alpha$,  $\phi_1(t)=c_1\e^{-(1-\gamma+\alpha(\beta-1))t/\alpha}$ for some constant $c_1>0$ and $\gamma_1=\gamma<\alpha/{(1\vee\theta_2)}$. Obviously, as soon as $\lambda<0,$
\eqref{TY-5} holds true for $K_2=1/\alpha$, $\psi_1(t)\equiv0$ and $K_3=0.$
In case of $\lambda>0,$   it follows from \eqref{1-TY} that for all $t>0$ and $x\in\R^d,$
\begin{align}\label{EP-21}
	\<x,-x/\alpha+\e^{-(1/\alpha-1)t}F(\e^t,\e^{t/\alpha}x)\>
 \le -|x|^2/\alpha+\lambda \e^{-(1-\gamma+\alpha(\beta-1))t/\alpha}(1  +| x|^2)^{\frac{1}{2}(\gamma+1)} .
\end{align}
Whence, with regard to $\lambda>0$ along with  $\gamma\in(0,1]$,
\eqref{TY-5} is still true.
Based on the preceding analysis, under case $(a), $\eqref{EW8} follows right now by applying Theorem   \ref{thm0}.

 Concerning case (II): $\beta=1+(\gamma-1)/\alpha$, we have $F_1(x)=\lambda x|x|^{\gamma-1}$ as shown in \eqref{EP-4}. Accordingly, we derive
 that for any $\lambda<0$ and $x,y\in\R^d$,
\begin{align}\label{EP-17}
\<x-y,-(x-y)/\alpha+F_1(x)-F_2(y)\>\le-|x-y|^2/\alpha
\end{align}
by making use of  the following fact: for all $x,y\in\R^d,$
\begin{align*}
 \<x-y, x|x|^{\gamma-1}-y|y|^{\gamma-1}\> \ge (|x|^\gamma-|y|^\gamma)(|x|-|y|)\ge0.
\end{align*}
As a consequence, \eqref{TY-3} for $K_1=1/\alpha$ can be verified.

 For case (II), by virtue of  Bernoulli's inequality:  $(1+r)^a\ge 1+ar, \forall r>-1,$ for all $a\ge1$ or $a\le0,$
  we deduce from \eqref{1-TY} that for all $t\ge0$ and $x\in\R^d,$
 \begin{align*}
  |\e^{-(1/\alpha-1)t}F(\e^t,\e^{t/\alpha}x)-F_1(x)|
  &\le|\lambda x| \big|(( 1+\e^{-t} )^{-\beta} -1)  \big|(\e^{-2t/\alpha}+| x|^2)^{(\gamma-1)/2} +G(t,x),\\
  &\le |\lambda|\beta\e^{-t}(1+| x|^2)^{ \gamma/2 }+G (t,x),
 \end{align*}
  where $$G (t,x):=|\lambda x| \big |  (  \e^{-2t/\alpha}+| x|^2)^{(\gamma-1)/2}-  | x|^{ \gamma-1 } \big|.$$
  Furthermore, there exist constants $c_2,c_3>0$ such that  for all $t\ge0$ and $x\in\R^d,$
\begin{equation}\label{EP-18}
	\begin{split}
G (t,x)
&\le\begin{cases}
	c_2 \e^{-\gamma t/\alpha} ,&0<\gamma\le1, \\
	(1\vee(\gamma-1)/2)|x|\big((\e^{-2t/\alpha})^{(\gamma-3)/2}\vee(  \e^{-2t/\alpha}+| x|^2)^{ (\gamma-3)/2}\big)\e^{-2t/\alpha},	& \gamma>1,
	\end{cases}\\
		&\le c_3\begin{cases}
		 \e^{-\gamma t/\alpha} ,&0<\gamma\le1, \\
	 \e^{-(2\wedge(\gamma-1)) t/\alpha}    (1 +| x|^{1\vee(\gamma-2)}),& \gamma>1,
	\end{cases}
\end{split}
\end{equation}
where in the first inequality
we exploited the following  estimate:  for any $c_4>0 $  and   $\kk\in(0,1/2)$,
\begin{align*}
r^{\frac{1}{2}}(r^{-\kk}-(c_4+r )^{-\kk})=c_4^{\frac{1}{2}-\kk}(r/{c_4})^{\frac{1}{2}-\kk}\big( 1-(1+ 1/ (r/{c_4})  )^{-\kk}\big)\le \kk c_4^{\frac{1}{2}-\kk}(r/{c_4})^{-(\kk+1/2)},
\end{align*}
 and
the subsequent inequality:    for all $a,b\ge0,$
\begin{align}\label{EP-6}
(a+b)^p-a^p\le (1\vee p) ( b^{p-1}\vee (a+b)^{p-1} )b.
\end{align}
Consequently,    we arrive at
\begin{equation}\label{EP-23}
	\begin{split}
&|\e^{-(1/\alpha-1)t}F(\e^t,\e^{t/\alpha}x)-F_1(x)|\\
&\le |\lambda| \beta\e^{-t}(1 +| x|^2)^{\frac{1}{2} \gamma }+c_3\begin{cases}
	\e^{-\gamma t/\alpha} ,&0<\gamma\le1, \\
	\e^{-(2\wedge(\gamma-1)) t/\alpha}    (1 +| x|^{1\vee(\gamma-2)}),& \gamma>1.
\end{cases}
\end{split}
\end{equation}
This, along with $\gamma_1=\gamma<\alpha/{(1\vee\theta_2)}$,  yields that
 \eqref{TY-4}  also holds
 when
 $\alpha=1$,  $\gamma\in(0,1)$ and
  $\phi_1(t)=c_5 \e^{-(1\wedge(\gamma/\alpha))t} $ for some constant $c_5>0$ (resp.  $1<\alpha<2$,  $\gamma\in(0,1]$ and
  $\phi_1(t)=c_5 \e^{-(1\wedge(\gamma/\alpha))t} $;  $1<\alpha<2$, $\gamma>1$ and  $\phi_1(t)=c_5 \e^{-(1\wedge ((2\wedge(\gamma-1))/\alpha))t} $).
Furthermore, as for the present case (II), \eqref{TY-5} is readily satisfied
for $K_2=1/\alpha$, $K_3=0$ and $\psi_1(t) \equiv0.$
On the basis of the analysis mentioned above, as far as case  $(b)$  is concerned, \eqref{EW8} is available  by applying Theorem   \ref{thm0} once more.

Before we proceed to prove the  assertion \eqref{EY-*}, regarding case (III): $\beta<1+(\gamma-1)/\alpha$,
we begin by proving the subsequent estimate:
 for all $x,y\in\R^d$ and $\gamma\ge1,$
 \begin{align}\label{2-TY}
 \<x-y,x|x|^{\gamma-1}-y|y|^{\gamma-1}\>\ge (2^{1-\gamma}/\gamma)|x-y|^{1+\gamma}.
 \end{align}
 Once \eqref{2-TY} is available,  \eqref{TY-6}  with $\kk=\gamma$ and $L_1=-\lambda 2^{1-\gamma}/\gamma$ follows directly.
Apparently, \eqref{2-TY} is true for $x=y.$
For  $g(x):=|x|^{1+\gamma}/{(1+\gamma)}$, we have that for all ${ 0}\neq x\in\R^d,$
\begin{align*}
\nn^2 g(x)= |x|^{\gamma-1}I_d+ (\gamma-1) |x|^{\gamma-3}(x\otimes x).
\end{align*}
This subsequently implies that
 for all $x,y\in\R^d,$
\begin{align*}
 \<x-y,x|x|^{\gamma-1}-y|y|^{\gamma-1}\>
 &\ge \int_0^1|y+s(x-y)|^{\gamma-1}\,\d s|x-y|^2\\
 &\ge |x-y|^{1+\gamma}\int_0^1\Big|\frac{\<y,x-y\>}{|x-y|^2}+s  \Big|^{\gamma-1}\,\d s,
\end{align*}
where in the second inequality we used the fact that the vector $y-\<y,x-y\>(x-y)/{|x-y|^2}$ is perpendicular to the vector $\<y,x-y\>(x-y)/{|x-y|^2}+s(x-y)$ so that
\begin{align*}
| y+s(x-y)|\ge
 \big| \<y,x-y\> /{|x-y|}+s|x-y| \big|.
 \end{align*}
Correspondingly, \eqref{2-TY} follows by
 employing the fact that for $\gamma\ge1,$
\begin{align*}
\inf_{c\in\R}\int_c^{c+1}|u|^{\gamma-1}\,\d u=\int_{- 1/2 }^{ 1/2 }|u|^{\gamma-1}\,\d u=2^{1-\gamma}/\gamma.
 \end{align*}

 By following the strategy to treat the term $G(t,x)$ mentioned previously, we derive from \eqref{EP-5} that there is a constant $c_6>0$ such that for any $t\ge0$ and $x\in\R^d$,
 \begin{align}\label{WPP-1}
 \big|g_t^{1-\alpha} F (\varphi_t,x/{g_t} )-F_2 (x)\big|
 &\le \frac{ |\lambda|\beta}{\varphi_t}   (1 +|x|^2 )^{ \gamma/2 } +c_6
 \begin{cases}
 	0,&\gamma=1,\\
  g_t^{2\wedge(\gamma-1) }(1 +| x|^{1\vee(\gamma-2)}),&\gamma>1
  \end{cases}
 \end{align}
Therefore,
\eqref{TY-7} is  fulfilled when $\gamma_2=\gamma$ and $\phi_2(t)=c_7((1/{\varphi_t})\vee g_t^{2\wedge(\gamma-1) })$ for some constant $c_7>0$ when $\gamma>1$ (resp. $\phi_2(t)=c_7/{\varphi_t}$ when $\gamma=1$).
Additionally, \eqref{EP-5}, besides $\varphi_0=1$ and $g_t\in(0,1]$,  enables us to derive that for all $\lambda<0$,
$t\ge0$ and $x\in\R^d,$
\begin{align*}
g_t^{1-\alpha}\<x,  F (\varphi_t,x/{g_t} )\>
	&\le\lambda(1+1/{\varphi_t})^{-\beta} |x|^{ \gamma+1 }-\lambda(1+1/{\varphi_t})^{-\beta} g_t^2\big( g_t^2+|x|^2\big)^{(\gamma-1)/2}\\
	&\le  \lambda(1-\beta/{\varphi_t})|x|^{ \gamma+1 }-\lambda  g_t^2\big( 1+|x|^2\big)^{(\gamma-1)/2}.
\end{align*}
Whence,
 \eqref{TY-8} is examinable by means of  Young's inequality.  Finally, by combining the preceding analysis with Theorem   \ref{thm0}, we complete the proof of Theorem \ref{theorem-1}.
\end{proof}

\subsection{Proofs of Theorem \ref{thm2} and Theorem \ref{thm2-1}}
This subsection aims at presenting the proof of Theorem \ref{thm2-1}. Subsequently, we derive Theorem \ref{thm2} as an application thereof.

\begin{proof}[Proof of Theorem {\rm\ref{thm2-1}}] In the subsequent analysis, we take $g_t=\e^{-t/\alpha}$ for all $t\ge0$.
	To accomplish the proof of Theorem \ref{thm2-1}, it is sufficient to   demonstrate  that
the
 corresponding
\eqref{R5--}, \eqref{WY} as well as \eqref{EE5} hold for $\bar \nu=\nu^{(\alpha)}$. Obviously, \eqref{EE5} follows from \eqref{P3}.
According to \eqref{P8} and $\int_{\{|z|\le1\}}|z|a(z)\,\d z<\infty,$
we have that for all $0<\gamma<\alpha<1$ and $t\ge0$,
\begin{equation*}
H(t):=\int_{\R^d} (|z|^\gamma \wedge|z|)\, |\nu_t-\nu^{(\alpha)}|(\d z)
=\e^{-(1/\alpha-1)t}\int_{\{|z|\le1\}}  | z|\Big|a( z)-\frac{c_*}{|  z|^{d+\alpha}}\Big|
\,\d  z<\8.
\end{equation*}
Thus, \eqref{WY} with $\theta=\gamma$ is attainable. Next, we  verify  the precondition \eqref{R5--}.
In terms of the definitions of $\nu_t(\d z)$ (see \eqref{EW3}) and $\nu^{(\alpha)}(\d z)$, we obtain that  for all $t\ge0,$
\begin{align*}
(\nu_t\wedge\nu^{(\alpha)})(\d z)&=\bigg(\Big(\big(g_t^{-d-\alpha}a( z/{g_t})\big)\wedge\frac{c_*}{|z|^{d+\alpha}}\Big)\I_{\{| z/{g_t}|\le1\}} +\frac{c_*}{|z|^{d+\alpha}}\I_{\{| z/{g_t}|>1\}}  \bigg)\,\d z  =:\Lambda_t(z) \,\d z.
\end{align*}
This  implies that for all $t\ge0$ and $x\in\R^d,$
\begin{align*}
(\delta_x\ast( \nu_t\wedge \nu^{(\alpha)}))(\d z)=\Lambda_t(z-x)\,\d z.
\end{align*}
Accordingly, it follows from \cite[Example A.3]{LW19} that for all $t\ge0$ and $x\in\R^d,$
\begin{align*}
\nu_{t,x}(\d z): =\big(( \nu_t \wedge \nu^{(\alpha)})\wedge(\delta_x\ast( \nu_t\wedge \nu^{(\alpha)}))\big)(\d z)
 =(\Lambda_t(z)\wedge \Lambda_t(z-x))\,\d z.
\end{align*}
Note  from $g_\cdot\in(0,1]$ that for all $t\ge0$ and $x\in\R^d,$
\begin{align}\label{WY1}
\nu_{t,x}(\R^d)&=\int_{\{ | z/{g_t}|\le1\}\cap\{ |(z-x)/{g_t}|\le1\}}\bigg(\big(g_t^{-d-\alpha}\big(a(z/{g_t})\wedge a((z-x)/{g_t})\big)\big)\nonumber\\
&\quad\quad\quad\quad\quad\quad\quad\quad\quad\quad\quad\quad\wedge\frac{c_*}{(|z|\vee|z-x|)^{d+\alpha}} \bigg)\,\d z\nonumber\\
&\quad+\int_{\{ |z/{g_t}|\le1\}\cap\{ |(z-x)/{g_t}|>1\}}\bigg(\big(g_t^{-d-\alpha}a(z/{g_t})\big)\wedge\frac{c_*}{(|z|\vee|z-x|)^{d+\alpha}}  \bigg) \,\d z\nonumber\\
&\quad+\int_{\{ |z/{g_t}|>1\}\cap\{ |(z-x)/{g_t}|\le1\}}\bigg(\big(g_t^{-d-\alpha}a((z-x)/{g_t})\big)\wedge\frac{c_*}{(|z|\vee|z-x|)^{d+\alpha}}  \bigg)\,\d z\\
&\quad+c_*\int_{\{ | z|>g_t\}\cap\{ |  z-x  |>g_t\}}  \frac{1}{(|z|\vee|z-x|)^{d+\alpha}} \,\d z\nonumber\\
&\ge c_*\int_{\{ | z|>g_t\}\cap\{ |  z-x  |>g_t\}\cap\{ |  z|\le|x|/2\}}  \frac{1}{(|z|\vee|z-x|)^{d+\alpha}} \,\d z.\nonumber
\end{align}
In case of $|z|\le
|x|/2$,  it holds that $   |z|\le |x|-|z|\le |z-x|\le
3|x|/2$. Consequently,
by recalling the fact that the    volume of a  ball with  radius $r$ in $\R^d$ is $\frac{\pi^{{d}/{2}}r^d}{\Gamma(1+d/2)}$
with $\Gamma(\cdot)$ being the Gamma function, and taking \eqref{WY1} into account,
the lower bound of $\nu_{t,x}(\R^d)$ can be provided  as below: for any $t\ge0$ and $x\in\R^d$,
\begin{align*}
\nu_{t,x}(\R^d)
&\ge c_*\int_{\{ | z|>g_t\} \cap\{ |  z|\le|x|/2\}}  \frac{1}{(|z|\vee|z-x|)^{d+\alpha}} \,\d z \\
&=\frac{c_*2^{d+\alpha}}{3^{d+\alpha}}|x|^{-(d+\alpha)}\bigg(\int_{ \{ |  z|\le|x|/2\}}  \,\d z
-\int_{\{ | z|\le g_t\} \cap\{ |  z|\le|x|/2\}}   \,\d z\bigg)\\
&=\frac{c_*(2/3)^{d+\alpha}\pi^{{d}/{2}} }{\Gamma(d/2+1)} |x|^{-  \alpha  }\big( 2^{-d } -
|x|^{-d}(g_t\wedge (|x|/2))^d \big).
\end{align*}
Whence, \eqref{R5--} is available  with $\beta=\alpha\in(0,1)$,
\begin{align*}
c^*:=\frac{c_*(2/3)^{d+\alpha}\pi^{{d}/{2}} }{2^d\Gamma(d/2+1)}\quad \mbox{ and }\quad h(t,r):=2^dc^*r^{-d}(g_t\wedge (r/2))^d\in (0,c^*],\quad \forall\, t\ge0,
r\in (0,\kappa].
\end{align*}

Based on the aforementioned analysis, we derive from Theorem \ref{thm3}, Proposition \ref{lem3} and \eqref{TY-4}
that
there exists a      constant    $  \lambda_1 >0$ such that
for all $t\ge0$ and $\mu \in\mathcal P_{\gamma\vee\gamma_1}(\R^d)$,
\begin{equation*}
	\begin{split}
 \mathcal W_{|\cdot|\wedge|\cdot|^\gamma}\Big(\mathscr L_{Y_t^{\mu }},\mathscr L_{\overline{Y}_{\!\! t}^{\mu }}\Big)    \le    +C_1\int_0^t\e^{-\lambda_1(t-s)} \big(g_s^{1-\alpha}+h^*(s,\kappa) +\phi_1(s)\big) \,\d s,
 \end{split}
\end{equation*}
where $h^*(t,\kappa):=\sup_{r\le \kappa} (r^{1+\rho-\alpha}h(t,r)) $ for $\rho\in(0,\alpha+\alpha_*-1)$, and $C_1$ is a positive constant depending on $\mu(|\cdot|^{\gamma_1})$. Next,
notice that for
all $\rho\in(0,\alpha+\alpha_*-1)$ and
$r,t\ge0$,
\begin{equation*}
	\begin{split}
		r^{1+\rho-\alpha-d}
		(g_t\wedge (r/2))^d \I_{\{r\le1\}} &=  r^{1+\rho-\alpha-d} \Big(
		(g_t\wedge (r/2))^d \I_{\{r\le g_t^{{1}/{d}}\}}+
		(g_t\wedge (r/2))^d \I_{\{ g_t^{{1}/{d}}<r\le1\}}\Big)  \\
		&\le 2^{-d} r^{1+\rho-\alpha }\I_{\{r\le g_t^{{1}/{d}}\}}+g_t^{(1+\rho-\alpha-d)/d}g_t^d\I_{\{ g_t^{{1}/{d}}<r\le1\}}\\
		&\le g_t^{(1+\rho-\alpha )/d}\Big(2^{-d} \I_{\{r\le g_t^{{1}/{d}}\}} +g_t^{d-1}\I_{\{ g_t^{{1}/{d}}<r\le1\}}\Big)\\
		&\le g_t^{(1+\rho-\alpha )/d},
	\end{split}
\end{equation*}
where in the second inequality we employed the fact that $1+\rho-\alpha-d<0 $ and in the third inequality we utilized $1+\rho-\alpha>0.$ Consequently, we arrive at the following estimate:
for all $t\ge0$ and $\mu  \in\mathcal P_{\gamma\vee\gamma_1}(\R^d)$,
\begin{equation}\label{P-21}
		\mathcal W_{|\cdot|\wedge|\cdot|^\gamma}\Big(\mathscr L_{Y_t^{\mu_1}},\mathscr L_{\overline{Y}_{\!\! t}^{\mu_2}}\Big)    \le    C_1\int_0^t\e^{-\lambda_1(t-s)} \big(g_s^{1-\alpha}+g_s^{(1+\rho-\alpha )/d} +\phi_1(s)\big) \,\d s.
\end{equation}

Once again, by applying Theorem \ref{thm3} with $H(t)=h(t,r)\equiv0$ and $\beta=\alpha$, there exist     constants   $  \lambda_2,C_2>0$ such that
for all $t\ge0$ and $\mu \in\mathcal P_{\gamma\vee\gamma_1}(\R^d)$,
\begin{equation} \label{P-22}
	\begin{split}
		\mathcal W_{|\cdot|\wedge|\cdot|^\gamma}\Big(\mathscr L_{\overline Y_{\!\! t}^{\mu}},\mathscr L_{\overline Y_{\!\! t}^{\pi}}\Big)   &\le    C_2 \e^{- \lambda_2 t}\mathcal W_{|\cdot|\wedge|\cdot|^\gamma}(\mu ,\pi)\\
		&\quad +C_3\int_0^t\e^{-\lambda_2(t-s)} \big(g_s^{1-\alpha}+g_s^{(1+\rho-\alpha )/d} +\phi_1(s)\big) \,\d s,
	\end{split}
\end{equation}
where the constant $C_3>0$  depends linearly on $\mu(|\cdot|^{\gamma_1})$.

 Finally, the assertion \eqref{P10} follows by combining \eqref{P-21} with \eqref{P-22}
and making use of the triangle inequality.
\end{proof}

\begin{proof}[Proof of Theorem {\rm\ref{thm2}}]
Since
 Assumption $({\bf H}_3')$ has been verified in the course of proving Theorem \ref{theorem-1}, the assertion \eqref{P-23} is available by applying Theorem {\rm\ref{thm2-1}} directly.
\end{proof}

\subsection{
Special cases of Theorem \ref{theorem-1} and Theorem \ref{thm0}: $(L_t)_{t\ge0}=(L_t^{(\alpha)})_{t\ge0}$}\label{section4.4}
Once $(L_t)_{t\ge0}=(L_t^{(\alpha)})_{t\ge0}$,
Theorem \ref{thm0} can be further strengthened, as stated in the following theorem.

\begin{theorem}\label{thm4-1} Assume that   $\alpha\in(1,2)$ and $(L_t)_{t\ge0}=(L_t^{(\alpha)})_{t\ge0}$ hold. Suppose further that either
Assumption $({\bf H}_3')$ holds
	or Assumption $({\bf H}_4)$ is satisfied.
	Then,  under Assumption $({\bf H}_3')$, there
	is a constant $\lambda_1^*>0$ such that for all $t\ge 1$ and $\mu\in\mathcal P_{1\vee \gamma_1}(\R^d)$,
	\begin{equation}\label{WY7-1}
		\mathcal W_1\Big(\mathscr L_{t^{-{1}/{\alpha}}X_t^{\mu}},\pi\Big) 	
		\le C_1^*\big(\phi_1((\ln t)/2) +t^{-\lambda_1^*}\big),
	\end{equation}
	where $\pi\in\mathcal P_1(\R^d)$ is the unique IPM of $(\overline Y_{\!\! t})_{t\ge0}$ which is governed by
	the   SDE \eqref{EP},
	and the constant  $C_1^*>0$ depends linearly  on   $\mu(|\cdot|^{1\vee\gamma_1})$ and $\pi(|\cdot|)${\rm;}
	under Assumption $({\bf H}_4)$, there
	exists a constant $\lambda_2^*>0$
	such that for all  $t\ge  1
	$ and  $\mu \in\mathcal P_{1\vee\gamma_2}(\R^d)$,
	\begin{equation}\label{WY7-2}
		\mathcal W_1\Big(\mathscr L_{t^{-\theta/\alpha}X_t^{\mu}},\pi\Big)  \le
		 C_2
		\begin{cases}
		  t^{-(1-\theta)}\vee\phi_2((t^{1-\theta}-1)/(2(1-\theta)))
		  +\e^{-L_1t^{1-\theta}/(4(1-\theta)) } ,&\kk=1,\\
		 \big(t^{-(1-\theta)}\vee\phi_2((t^{1-\theta}-1)/(2(1-\theta)))\big)^{1/\kk},&\kk>1,
		 \end{cases}
	\end{equation}	
	where  $\pi\in\mathcal P_1(\R^d)$ is the unique IPM of $(\overline Y_{\!\! t})_{t\ge0}$ solving  the SDE   \eqref{EP-1},
	and the constant
	$C_2>0$ is linearly dependent  on   $\mu(|\cdot|^{1\vee\gamma_2})$ and $\pi(|\cdot|)$.
\end{theorem}

Based on the previous theorem, as far as the case $(L_t)_{t\ge0}=(L_t^{(\alpha)})_{t\ge0}$ is concerned, Theorem \ref{theorem-1}
can also be further refined as the following theorem demonstrates.

\begin{theorem}\label{thm4} Assume that $({\bf H}_2)$, $\alpha\in(1,2)$ and $(L_t)_{t\ge0}=(L_t^{(\alpha)})_{t\ge0}$ hold. Suppose further that one of the following conditions holds:
	\begin{itemize}
		\item[$(a)$] $\beta>1+ (\gamma-1)/\alpha$,     $(i)$ $\lambda>0$ and  $\gamma\in(0,1]$, or $(ii)$  $\lambda<0 $ and $0<\gamma<\alpha $;
		\item[$(b)$]$\beta=1+ (\gamma-1)/\alpha$, 	  $(i)$ $\lambda<0$  and $0<\gamma< \alpha$, or $(ii)$ $0<\lambda<1/\alpha$ and $\gamma=1$, or $(iii)$ $\lambda>0$ and $\gamma\in(0,1)$;
		\item[$(c)$]$\beta<1+(\gamma-1)/\alpha$,  $\lambda<0$ and
		$1\le \gamma<\alpha/{(1\vee\theta_2)}$.
	\end{itemize}
	Then,  under case $(a)$ $($resp. case $(b)$$)$, there
is a constant $\lambda_1^*>0$ such that for all $t\ge 1$ and $\mu\in\mathcal P_1(\R^d)$,
\begin{equation}\label{WY7}
	\mathcal W_1\Big(\mathscr L_{t^{-{1}/{\alpha}}X_t^{\mu}},\pi\Big) 	
	\le C_1^*t^{-\lambda_1^*},
\end{equation}
where $\pi\in\mathcal P_1(\R^d)$ is the unique IPM of $(\overline Y_{\!\! t})_{t\ge0}$ which is governed by
 the first $($resp. the second$)$ SDE in \eqref{e:limit} with $\eta=\alpha$,
 and the constant  $C_1^*>0$ depends linearly  on $\mu(|\cdot|)$ and $\pi(|\cdot|)$;
regarding  case $(c)$, there
exists a constant $\lambda_2^*>0$
such that for all  $t\ge  1
$ and  $\mu \in\mathcal P_1(\R^d)$,
\begin{equation}\label{WY7-}
	\mathcal W_1\Big(\mathscr L_{t^{-q}X_t^{\mu}},\pi\Big)  \le C_2^*  t^{-\lambda_2^*},
\end{equation}	
where $q:= \beta/{(\alpha+\gamma-1)}$, $\pi\in\mathcal P_1(\R^d)$ is the unique IPM of $(\overline Y_{\!\! t})_{t\ge0}$ solving  the third SDE in \eqref{e:limit} with $\eta=\alpha$ and  $\gamma_0=0$,
and the constant
$C_2^*>0$ is linearly dependent  on $\mu(|\cdot|)$ and $\pi(|\cdot|)$.
\end{theorem}

Below, we first prove Theorem \ref{thm4-1}. As an application of this result, we proceed  to  prove Theorem \ref{thm4}.

\begin{proof}[Proof of Theorem $\ref{thm4-1}$]
	As long as  $(L_t)_{t\ge0}=(L_t^{(\alpha)})_{t\ge0}$, it is easy to see that $(Z_t)_{t\ge0}=(L_t^{(\alpha)})_{t\ge0}$ so \eqref{P3} with $(\nu_t)_{t\ge0}=\nu^{(\alpha)}$ holds true for all $p \in(0,\alpha)$. Then,  for any $p\in(0,\alpha)$, \eqref{TY-5} (resp. \eqref{TY-8})  and
	 Proposition \ref{lem3} enable us to derive that there exists a constant $C_1>0$ such that for all $t\ge0$ and $\mu\in\mathcal P_p(\R^d)$,
	\begin{align}\label{EP-8}
		\E|Y_t^\mu|^p\le C_1(1+\mu(|\cdot|^p)),
	\end{align}
	where $(Y_t)_{t\ge0}$ solves the SDE \eqref{UP-2} (resp. the SDE \eqref{UP-3}).

By virtue of \eqref{EP-24}, it follows from \cite[Theorem 4.2]{LW19} that there exist constants $C_2,\lambda^\star_1>0$ such that for all $\mu_1,\mu_2\in\mathcal P_1(\R^d)$ and $t\ge0,$
\begin{align}\label{EP-9}	
\mathcal W_1\big(\mathscr L_{\overline{Y}^{\mu_1}_{\!\! t}},\mathscr L_{\overline{Y}^{\mu_2}_{\!\! t}}\big)	\le C_2\e^{-\lambda_1^\star t}\mathcal W_1(\mu_1,\mu_2).
\end{align}	
Moreover, under Assumption $({\bf H}_3')$, applying  Theorem \ref{thm3} with $H(t)=h(t,r)\equiv0$ and $\theta=1$ (see also the assertion in Remark \ref{REm}) and taking \eqref{EP-7} and  \eqref{EP-8} into account yields that there is some constant  $\lambda_2^\star>0$ such that for all $\mu_1\in\mathcal P_{1\vee\gamma_1}(\R^d)$ and $\mu_2\in\mathcal P_1(\R^d)$,
\begin{equation}\label{EP-10}
	\begin{split}
	\mathcal W_1\big(\mathscr L_{Y_{\! t}^{\mu_1}},\mathscr L_{\overline{Y}^{\mu_2}_{\!\! t}}\big) & \le C_3\bigg(\e^{- \lambda_2^\star t}\mathcal W_1(\mu_1,\mu_2)   + \int_0^t\e^{-\lambda_2^\star(t-s)} \phi_1(s) \,\d s\bigg)\\
	&\le C_4\bigg(\e^{- \lambda_2^\star t}\mathcal W_1(\mu_1,\mu_2)   +\phi_1(t/2)  +\e^{-\lambda_2^\star/4}\bigg),
	\end{split}
\end{equation}
where positive constants $C_3,C_4 $ depend linearly on $\mu_1(|\cdot|^{\gamma_1})$.
Thus, under Assumption $({\bf H}_3')$, the assertion \eqref{WY7-1} follows by combining \eqref{EP-9} with \eqref{EP-10}
and taking advantage of the following
 fact: for all $t\ge1 $ and $\mu\in\mathcal P_1(\R^d)$,
 \begin{align}\label{EP-22}
	\mathcal W_1\big(\mathscr L_{t^{-1/\alpha}X_t^{\mu}},\pi\big)\le \mathcal W_1\big(\mathscr L_{t^{-1/\alpha}X_t^{\mu}},\mathscr L_{\overline Y_{\!\!  \ln t}^{\mu}}\big)+\mathcal W_1\big( \mathscr L_{\overline Y_{\!\! \ln t}^{\mu}},\pi\big),
\end{align}
where $\pi\in\mathcal P_1(\R^d)$ is the unique IPM of $(\overline Y_{\!\! t})_{t\ge0}$ solving
the   SDE \eqref{EP}.

By virtue of \eqref{TY-6},  along with $(\nu_t)_{t\ge0}=\nu^{(\alpha)}$,
we obtain from Jensen's inequality that for all $t\ge0$ and $\mu_1,\mu_2\in\mathcal P_1(\R^d)$,
\begin{align}\label{EP-11}
\mathcal W_1\big(\mathscr L_{Y_{\! t}^{\mu_1}},\mathscr L_{\overline{Y}^{\mu_2}_{\!\! t}}\big)\le \big(\mathcal W_1 (\mu_1,\mu_2 )^{1-\kk}+L_1(\kk-1)t\big)^{-1/{(\kk-1)}}.
\end{align}	
Next, making use of \eqref{TY-6} and \eqref{TY-7} yields that for all $t\ge0$, $\mu_1\in\mathcal P_{1\vee \gamma_2}(\R^d)$ and  $\mu_2\in\mathcal P_1(\R^d)$,
\begin{align*}
 \d |Y_{\! t}^{\mu_1}-\overline{Y}^{\mu_2}_{\!\! t}|\le \big(-L_1|Y_{\! t}^{\mu_1}-\overline{Y}^{\mu_2}_{\!\! t}|^\kk+\phi_2  (t)(1+|Y_{\! t}^{\mu_1}|^{\gamma_2})+\theta |Y_{\! t}^{\mu_1}| /{(\alpha
 	(1+(1-\theta)t) ) } \big)\,\d t.
 \end{align*}
Thus, we conclude from \eqref{EP-7}, \eqref{EP-8} and Lemma \ref{lemma-*} that for all $\mu_1\in\mathcal P_{1\vee \gamma_2}(\R^d)$, $\mu_2\in\mathcal P_1(\R^d)$ and $t\ge0,$
 \begin{equation}\label{EP-12}
 	\begin{split}
 	\mathcal W_1\big(\mathscr L_{Y_{\! t}^{\mu_1}},\mathscr L_{\overline{Y}^{\mu_2}_{\!\! t}}\big)
 	 &\le
 	 \begin{cases}
 	 	\e^{-L_1t}\mathcal W_1 (\mu_1,\mu_2 ) +C_5\big((1/t)\vee\phi_2  (t/2)+\e^{-L_1t/4}\big),&\kk=1,\\
 	 	C_6\big( t^{-1/{(\kk-1)}}\vee\big((1/t)\vee\phi_2(t/2)\big)^{1/\kk}\big),&\kk>1,
 	 \end{cases}
 	 \end{split}
 \end{equation}
 where $C_5,C_6>0$ depend  linearly on $\mu_1(|\cdot|^{1\vee\gamma_2})$. Subsequently, under Assumption $({\bf H}_4)$, \eqref{WY7-2}
 is available by taking
 \eqref{EP-11}
 and \eqref{EP-12} into consideration, and
invoking the triangle inequality below:
 \begin{align*}
 	\mathcal W_1\big(\mathscr L_{t^{-\theta/\alpha}X_t^{\mu}},\pi\big)\le \mathcal W_1\Big(\mathscr L_{t^{-\theta/\alpha}X_t^{\mu}},\mathscr L_{\overline Y_{\!\!  \varphi_t^{-1}}^{\mu}}\Big)+\mathcal W_1\Big( \mathscr L_{\overline Y_{\!\! \varphi_t^{-1}}^{\mu}},\pi\Big),
 \end{align*}
where $\varphi_t^{-1}=(t^{1-\theta}-1)/{(1-\theta)}$, and  $\pi\in\mathcal P_1(\R^d)$ is the unique IPM of $(\overline Y_{\!\! t})_{t\ge0}$ determined by  the SDE   \eqref{EP-1}.
\end{proof}

With the aid of Theorem \ref{thm4-1}, we are ready to establish Theorem  \ref{thm4}.

\begin{proof}[Proof  of Theorem $\ref{thm4}$]
Concerning case (I): $\beta>1+(\gamma-1)/\alpha$, for any $p\in(0,\alpha)$, there exists a constant $C_1>0$ such that for all $t\ge0$ and $\mu\in\mathcal P_p(\R^d)$,
\begin{align}\label{EP-14}
\E|Y_t^\mu|^p\le C_1(1+\mu(|\cdot|^p))
\end{align}
in case of $\lambda<0$ or $\lambda>0$ and $\gamma\in(0,1]$; see \eqref{EP-21} for more details. As for Case (I),
it is easy to see that for any $\mu_1,\mu_2\in\mathcal P_1(\R^d)$ and $t\ge0,$
\begin{align}\label{EP-15}
\mathcal W_1\big(\mathscr L_{\overline Y_{\!\!  t}^{\mu_1}},\mathscr L_{\overline Y_{\!\!  t}^{\mu_2}}\big)=\e^{-t/\alpha}\mathcal W_1(\mu_1,\mu_2),
\end{align}
where $(\overline Y_{\!\!  t})_{t\ge0}$ solves the first SDE in \eqref{e:limit},
and from \eqref{1-TY} and \eqref{EP-14} with $p=1\vee\gamma<\alpha$ that for any $\mu \in\mathcal P_{1\vee\gamma}(\R^d)$ and $t\ge0,$
\begin{equation}\label{EP-16}
	\begin{split}
\mathcal W_1\big(\mathscr L_{Y_{\! t}^{\mu }},\mathscr L_{\overline{Y}^{\mu }_{\!\! t}}\big) &\le  |\lambda|\int_0^t\e^{-(t-s)/\alpha} \e^{-(1-\gamma+\alpha(\beta-1))s/\alpha}  \E(1  +| Y_s^{\mu} |^2)^{\frac{1}{2} \gamma }\,\d s\\
&\le C_2\e^{-(1\wedge (1-\gamma+\alpha(\beta-1))t/\alpha}\big(1+t\I_{\{\gamma=\alpha(\beta-1)\}}\big),
\end{split}
\end{equation}
where $C_2$ is a positive constant  depending on $\mu(|\cdot|^{1\vee\gamma})$. As far as case $(a)$ is concerned,
the assertion \eqref{WY7} follows from \eqref{EP-15} and \eqref{EP-16} and by taking advantage of
\eqref{EP-22}.

With regard to case (II):  $\beta=1+(\gamma-1)/\alpha$, the moment estimate
\eqref{EP-14} is still valid when $\lambda<0$, $0<\lambda<1/\alpha$ with $\gamma=1$,
as well as $\lambda>0$ with $\gamma\in(0,1)$, respectively,
by noting from \eqref{1-TY} that for all $\lambda>0$, $t\ge0$ and $x\in\R^d,$
\begin{align*}
\<x,-x/\alpha+\e^{-(1/\alpha-1)t}F(\e^t,\e^{t/\alpha}x)\>&\le-|x|^2/\alpha+\lambda  ( 1+\e^{-t} )^{-\beta}  (\e^{-2t/\alpha}+| x|^2)^{\frac{1}{2}(\gamma+1)}.
\end{align*}
For case (II), it is easy to see from \eqref{EP-23} and  \eqref{EP-14} that there exists a constant $C_4>0$ such that  for all $t\ge0 $ and $\mu\in\mathcal P_{1\vee\gamma}(\R^d)$,
\begin{align*}
 \mathcal W_1\big(\mathscr L_{Y_{\! t}^{\mu }},\mathscr L_{\overline{Y}^{\mu }_{\!\! t}}\big) &\le C_4 (1 +\mu(| \cdot|^\gamma))\int_0^t\e^{-(t-s)/\alpha}\e^{-s}  \,\d s\\
 &\quad+C_4	\begin{cases}
 	\int_0^t\e^{-(t-s)/\alpha} \e^{-\gamma s/\alpha}\,\d s ,&0<\gamma\le1 ,\\
 	   (1 +\mu(| \cdot|)^{1\vee(\gamma-2)}) \int_0^t\e^{-(t-s)/\alpha} \e^{-(2\wedge(\gamma-1)) s/\alpha}\,\d s ,& \gamma>1.	
 \end{cases}
\end{align*}
Thus, as far as   case  $(i) $ is concerned, the assertion \eqref{WY7}
is available by taking \eqref{EP-7} into consideration and noting that
  for any $\mu_1,\mu_2\in\mathcal P_1(\R^d)$ and $t\ge0,$
 \begin{align*}
 	\mathcal W_1\big(\mathscr L_{\overline Y_{\!\!  t}^{\mu_1}},\mathscr L_{\overline Y_{\!\!  t}^{\mu_2}}\big)\le\e^{-t/\alpha}\mathcal W_1(\mu_1,\mu_2).
 \end{align*}

When $\lambda>0,$
\eqref{TY-4} and \eqref{TY-5} are verifiable as shown in   \eqref{EP-21} and \eqref{EP-23},  respectively. Next,
it follows from Lemma \ref{lem0}  that for any  $\gamma\in(0,1)$ and $x,y\in\R^d,$
\begin{align*}
 \<x-y,-(x-y)/\alpha+F_1(x)-F_1(y)\>&\le-|x-y|^2/\alpha +\lambda (2^\gamma+3^\gamma)|x-y|^{1+\gamma}\\
 &\le\lambda (2^\gamma+3^\gamma)|x-y|^{1+\gamma}\I_{\{|x-y|\le\ell_0\}}-|x-y|^2/{(2\alpha)}\I_{\{|x-y|>\ell_0\}},
 \end{align*}
where $\ell_0:=1\vee (2\lambda(2^\gamma+3^\gamma)/\alpha)^{1/{(1-\gamma)}}$. Hence, \eqref{EP-24} holds true right now. Accordingly, regarding   cases $(ii)$ and $(iii)$,
\eqref{WY7} is provable  by applying Theorem \ref{thm4-1}.

Concerning case $(c)$,
in terms of Theorem \ref{thm4-1} (see  \eqref{WY7-2} for more details), it suffices to examine $({\bf H}_4)$ in order to establish the assertion \eqref{WY7-}. The
  verification of Hypothesis $({\bf H}_4)$ has been presented in detail in the proof of Theorem \ref{theorem-1} so we herein omit the related details.
\end{proof}

\section{Quantitative estimates for the long-time behavior of time-inhomogeneous SDEs: an asymptotic pseudotrajectory approach
}
\label{Section4}
In this section, we consider the following two SDEs: for $t>0, $
\begin{align}\label{D12}
\d Y_t  =b(t,Y_t )\,\d t+\d Z_t,
\end{align} and \begin{align}\label{D11}
\d \overline{Y}_{\!\! t}  =\bar b(\overline{Y}_{\!\! t})\,\d t+ \si\d W_t,
\end{align}
where $b:[0,\infty)\times \R^d\to\R^d$, $\bar b:\R^d\to\R^d$, $ \si\neq0$,
 $(Z_t)_{t\ge0}$ is a $d$-dimensional additive process with the L\'{e}vy measure $(\nu_t(\d z))_{t\ge0}$, and  $(W_t)_{t\ge0}$
is a $d$-dimensional
standard Brownian motion.
Unless stated otherwise, in this section we always assume that the SDEs \eqref{D12} and \eqref{D11} are strongly well-posed.
Denote $(\mathscr L_{\! t})_{t\ge0}$
 and $(P_{\! s,t})_{t\ge s\ge0}$ (resp.\ $\bar {\mathscr L} $
 and $(\overline{P}_{\! s,t})_{t\ge s\ge0}$) by the infinitesimal generator and the semigroup  of $(Y_t)_{t\ge0}$
 (resp. $(\overline Y_{\!\! t})_{t\ge0}$). On some occasions, we write $({Y_{s,t}^x})_{t\ge s\ge0}$ and $(\overline Y_{\!\! s,t}^{x})_{t\ge s\ge0}$ instead of $(Y_t)_{t\ge0}$ and $(\overline Y_{\!\! t})_{t\ge0}$ to emphasize the starting time
  $s$ and the initial value
  $Y_s=\overline Y_{\!\!s}=x\in\R^d$, respectively. In addition, for simplicity,  we write $\overline {P}_{\!\! t}=\overline P_{\!\!0,t}=\overline P_{\!\! s,s+t}$ by taking  the homogeneous property of $(\overline P_{\!\! s,t})_{t\ge s\ge0} $ into consideration. Additionally,
we need to consider the perturbed version of the SDE \eqref{D11}: for
any $\vv\in(0,1]$ and $t>0$,
\begin{align}\label{WW-32}
\d \overline{Y}_{\!\! t}^\vv=\bar b^\vv (\overline{Y}_{\!\! t}^\vv)\,\d t+\sigma\d W_t,
\end{align}
where
$\R^d\ni x\mapsto \bar b^\vv(x)$ is a $C^3(\R^d;\R^d)$-function, $\lim_{\vv\to0}\bar b^\vv(x)=\bar b(x)$ for each fixed $x\in\R^d$, and $(\si,(W_t)_{t\ge0})$ remains untouched as in the SDE \eqref{D11}.
In the sequel,   the SDE \eqref{WW-32} is also supposed to be strongly well-posed. Furthermore, we write $(\overline{Y}_{\!\! t}^{\vv,x})_{t\ge0}$ instead of $(\overline{Y}_{\!\! t}^{\vv })_{t\ge0}$ when $\overline{Y}_{\!\! 0}^{\vv}=x$, and  denote by  $(\overline {P}_{\!\! t}^\vv)_{t\ge0}$ and $\overline {\mathscr{L}}^\vv$    the respective semigroup and infinitesimal generator of $(\overline{Y}_{\!\! t}^\vv)_{t\ge0}$.
Sometimes, we also write $\overline {P}^\vv_{\!\! t}=\overline P^\vv_{\!\!0,t}=\overline P^\vv_{\!\! s,s+t}$ for all $t>0$ and $s\ge0.$

The following theorem provides some sufficient conditions to establish the asymptotics of time-inhomogeneous SDEs with pure-jump additive processes, which is measured
by
using an asymptotic pseudotrajectory.

\begin{theorem}\label{Th4.1}Let $(\mu_t )_{t\ge0}$ be the distribution of the process $(Y_t)_{t\ge0}$
	solving \eqref{D12}.
	Suppose that there
	exist
	a non-increasing function $\Phi: \R_+\to \R_+$ and a
	measurable function $\mathcal V: \R^d\to [1,\infty)$ such that the following conditions hold{\rm:}
	\begin{itemize}
		
		\item[{\rm(i)}] for any $t>0$ and $\vv\in(0,1]$,
		\begin{align*}
			&\overline P_{\!\! t}^\vv( C_c^\infty(\R^d))\cup \overline{\mathscr L}^\vv\,\overline P_{\!\! t}^\vv(C_c^\infty(\R^d))\\
			& \subset \Big\{ f: \frac{\partial}{\partial t}\overline P_{\!\! t}^\vv f=\overline P_{\!\! t}^\vv \overline {\mathscr{L}}^\vv f \hbox{ and } \frac{\partial}{\partial t}P_{s,t} f=P_{s,t}\mathscr{L}_t f \hbox{ for all }  0<s<t\Big\};
		\end{align*}

		\item[{\rm(ii)}] there is some constant $C_0>0$ such that for all $0< r\le s< t $, $x\in\R^d$, $f\in C_c^\infty(\R^d)$ and $\vv\in(0,1]$,
		 \begin{equation}\label{WW-38}
			\big| (\mathscr L_r  - \overline{\mathscr L}^\vv  )(\overline  P_{\!\! s,t}^\vv   f)(x)\big| \le   C_0\sum_{i=0}^3\|\nn^if\|_\8\big( \Phi(r)  \mathcal V(x)+|\bar b^\vv(x)-\bar b(x)|\big),
		\end{equation}
		and
\begin{align}\label{WW-37}		
\lim_{\vv\to0}	\bigg(\sup_{s\le r\le t}\int_{\R^d}\sup_{s\le u\le t}	\E|\bar b^\vv-\bar b|(\overline{Y}_{\!\!  u}^{ x})\mu_r(\d x)\bigg)=0;
\end{align}		
		
		\item [{\rm(iii)}]$\sup_{t\ge0}\mu_t(\Theta)<\infty$    with  $$ \Theta(x):=\sup_{t\ge0} \E \mathcal V(Y_t^{ x}),\quad x\in\R^d;$$
		\item [{\rm(iv)}]	 there is some constant $C_0^*>0$ such that for all $t>0$, $x\in\R^d$ and $f\in C_c^\infty(\R^d)$,
$$|\nn \overline {P}_{\!\! t}  f(x)|\le C_0^*\|\nn f\|_\8,$$ and, for all $0\le s<t$,
		\begin{align}\label{WW-30}
		\int_0^t\int_{\R^d}\varlimsup_{\vv\downarrow0}\E |\bar   b^\vv-\bar b | (\overline{Y}_{\!\!  r}^{\vv,x})\,\d r\mu_s(\d x)=0.
	\end{align}	
		
	\end{itemize} Then, there exists a  constant  $C_* >0$ such that for any $t>0$, $s\ge0$ and $f\in  C_c^\infty(\R^d)$,
	\begin{align}\label{D19-}
		\big|(\mu_s \overline {P}_{\!\! t}  )( f)-\mu_{s+t} (f)\big|
		\le  C_*   \sum_{i=0}^3\|\nn^i f\|_\8   \int_{s }^{s+t} \Phi(u) \,\d u.
	\end{align}
\end{theorem}

Before we move forward  to finish the proof of Theorem \ref{Th4.1}, let's make some comments on Theorem \ref{Th4.1}.

\begin{remark}\label{remark5.2}
It is worth emphasizing  that it is unnecessary to smooth $\bar b $  in case of $\bar b\in C^3(\R^d;\R^d)$.   In such a case, the corresponding $\overline P_{\!\! t}^\vv$ and
 $\overline{\mathscr L}^\vv$ can be replaced by $\overline P_{\!\! t} $ and
 $\overline{\mathscr L} $, respectively; 
 moveover, Assumption (\rm iv) and \eqref{WW-37} are redundant, and the remainder term $|\bar b^\vv(x)-\bar b(x)|$ in \eqref{WW-38} vanishes. In view of this, we merely adopt a smooth strategy
as soon as $\bar b$ is not a $C^3(\R^d,\R^d)$-function.
Once the noise intensity $\si$ in the SDE \eqref{D11} is non-degenerate,
Assumption \rm(i) can be easily examined. Assumption \rm(ii) describes the uniform
discrepancy between $(\mathscr L_r)_{r\ge0}$ and $\overline{\mathscr L}^\vv  $. It is validated provided that $\nn (\overline P_{\!\! t}^\vv f)(x)$,
$\nn^2 (\overline P_{\!\! t}^\vv f)(x)$ and $\nn^3 (\overline P_{\!\! t}^\vv f)(x)$
(so $\R^d\ni x\mapsto \bar b^\vv(x)$ is required to be a $C^3(\R^d;\R^d)$-function) can be dominated  uniformly with respect to the parameter $\vv\in(0,1]$;
see Lemma \ref{lemma5}
below for more details.
Assumption \rm(iii) demonstrates that $\Theta$ is uniformly integrable with respect to $(\mu_t)_{t\ge0}$, which  can be readily checked in case $\mathcal V$ is of polynomial growth and $(Y_t^x)_{t\ge0}$ enjoys a uniform $p$-th moment for some $p>0.$
As long as $\<y,\nn_y \bar b(x)\>\le0$ for all $x,y\in\R^d,$ the estimate $|\nn \overline {P}_{\!\! t}  f(x)|\le C_0\|\nn f\|_\8$ follows directly via the chain rule.
Assume that there exist a function $h:[0,\infty)\to[0,\infty)$ satisfying $\lim_{\vv\to0}h(\vv)=0$ and some constant $p>0$ such that for all $x\in\R^d $ and $\vv\in(0,1]$,
\begin{align}\label{WW-33}
|\bar b^\vv(x)-\bar b (x)|\le h(\vv)(1+|x|^p).
\end{align}
For example, $b(x)=-x|x|^{\gamma-1}$ and its perturbed version  $\bar b^\vv(x)=-x(\vv+|x|^2)^{(\gamma-1)/2}$ for some $\gamma>1$ satisfy the precondition \eqref{WW-33} by noting from \eqref{EP-6} that for all $x\in\R^d$ and $\vv\in(0,1]$,
\begin{equation}\label{WW-39}
	\begin{split}
|\bar b^\vv(x)-\bar b (x)|&\le|x|\big((\vv+|x|^2)^{(\gamma-1)/2}-|x|^{\gamma-1}\big)\\
&\le (1\vee((\gamma-1)/2))|x|\big(\vv^{(\gamma-3)/2}\vee(\vv+|x|^2)^{(\gamma-3)/2}\big)\vv.
\end{split}
\end{equation}
Whence, \eqref{WW-30} is verifiable as soon as $(\overline{Y}_{\!\!  t}^{\vv,x})_{t\ge0}$ admits finite-time $(1\vee(\gamma-2))$-th moment. For the typical candidate mentioned above, by invoking \eqref{WW-39}, \eqref{WW-37} is verifiable as long as the moment $\E|
\bar Y_t^x|^{1\vee(\gamma-2)}$ is finite in
any finite time interval.
\end{remark}

Now, we proceed to carry out the

\begin{proof}[Proof of Theorem $\ref{Th4.1}$]
	Throughout the proof, we assume    $t>0$, $s\ge0$, $x\in\R^d$, $f\in C_c^\infty(\R^d)$ and $\vv\in(0,1]$.
  All the constants below are independent of them.
  Due to the Markov property of the process $(Y_t)_{t\ge0}$,  $\mu_r =\mu_uP_{u,r} $ for all $r\ge u\ge0.$ Then,
	via
	the semigroup property  of $(P_{s,t} )_{t\ge s\ge 0}$ and $(\overline P_{\!\! s,t}^\vv )_{t\ge s\ge 0}$,
	we deduce that
	for any integer  $n\ge1$,
	\begin{align*}
		(\mu_s \overline P_{\!\! t}^\vv )( f)-\mu_{s+t} (f) &=\mu_s \big( \overline P_{\!\! s,s+t}^\vv  f-P_{\! s,s+t}  f\big)\\
		&=\sum_{k=0}^{n-1}\mu_s \big(P_{s,s_k} \overline P_{\!\! s_k,s+t}^\vv  f-P_{s,  s_{k+1} } \overline P_{\!\! s_{k+1},s+t}^\vv f\big)\\
		&=\sum_{k=0}^{n-1}\mu_{s_k}  \big(\big(\overline P_{\!\! s_k, s_{k+1} }^\vv    - P_{s_k, s_{k+1} }  \big)\psi_k^\vv(f)\big),
	\end{align*}
	in which  $s_k:=s+ \frac{kt}{n}$ and  $\psi_k^\vv(f)(x):=\overline P_{\!\! s_{k+1},s+t}^\vv f(x)$. Next, by leveraging  the Kolmogorov equation and
	(i),
	the quantity $(\overline P_{\!\! s_k, s_{k+1} }^\vv  - P_{s_k, s_{k+1} } )\psi_k^\vv(f)$ involved in the equality above can be rephrased as follows:  for all $k=0,\cdots, n-1$,
	\begin{align*}
		&\big(\overline P_{\!\! s_k, s_{k+1} }^\vv    - P_{s_k, s_{k+1} }  \big)\psi_k^\vv(f)(x)\\
		&=\big(\overline P_{\!\! s_k, s_{k+1} }^\vv  -I\big)\psi_k^\vv(f)(x)-\big(P_{s_k, s_{k+1} } -I\big)\psi_k^\vv(f)(x)\\
		&=\int_{s_k}^{s_{k+1}}\big(\overline P_{\!\! s_k,u}^\vv  \bar{\mathscr L}^\vv  -P_{s_k,u}  \mathscr L_u  \big)\psi_k^\vv(f)(x)\,\d u\\
		&=\int_{s_k}^{s_{k+1}}\big( (\overline P_{\!\! s_k,u}^\vv  -I)\overline{\mathscr L}^\vv  \psi_k^\vv(f)(x)+(I-P_{s_k,u} )\overline{\mathscr L}^\vv  \psi_k^\vv(f)(x)  +P_{s_k,u}  (\overline{\mathscr L}^\vv  -\mathscr L_u )\psi_k^\vv(f)(x)\big)\,\d u\\
		&=\bigg(\int_{s_k}^{s_{k+1}}\int_{s_k}^u  \overline P_{\!\! s_k,r}^\vv   (\overline{\mathscr L}^\vv  )^2 \psi_k^\vv(f)(x)\, \d r\,\d u -\int_{s_k}^{s_{k+1}}\int_{s_k}^u   P_{s_k,r}   (\overline{\mathscr L}^\vv   )^2 \psi_k^\vv(f)(x)\,\d r\,\d u\bigg)\\
		&\quad+\bigg(-\int_{s_k}^{s_{k+1}}\int_{s_k}^u   P_{s_k,r}  (\mathscr L_r  -\overline{\mathscr L}^\vv  ) \overline{\mathscr L}^\vv   \psi_k^\vv(f)(x)\,\d r\,\d u +\int_{s_k}^{s_{k+1}}P_{s_k,u} (\overline{\mathscr L}^\vv  -\mathscr L_u  )\psi_k^\vv(f)(x) \,\d u\bigg)\\
		&=: I_{k,1}^\vv(x)+I_{k,2}^\vv(x),
	\end{align*}
	where $I$ represents  the identity map on $\R$.

	In the subsequent analysis, we aim at treating  the terms $I_{k,1}^\vv(x) $  and $I_{k,2}^\vv(x) $, separately.
	First of all, owing to $s_{k+1}-s_k=t/n$,
	\begin{align}\label{D17}
		I_{k,1}^\vv(x) \le \frac{t^2}{ n^2}\big\|(\overline{\mathscr L}^\vv )^2 \psi_k^\vv(f)\big\|_\8 \le \frac{t^2}{ n^2}\big\| (\overline{\mathscr L}^\vv )^2  f \big\|_\8,
	\end{align}
	where we also utilized the fact that $(\overline{\mathscr L}^\vv )^2 \psi_k^\vv(f)(x)=\overline P_{\!\! s_{k+1},s+t}^\vv  (\overline{\mathscr L}^\vv  )^2f(x)$ (via Kolmogorov's backward equation) to derive the second inequality.
	Note that the quantity  $I_{k,2}^\vv(x) $ can be reformulated as follows:
	\begin{align*}
		I_{k,2}^\vv(x)&=-\int_{s_k}^{s_{k+1}}\int_{s_k}^u   \E\big((\mathscr L_r  -\overline{\mathscr L}^\vv  ) \overline P_{\!\! s_{k+1},s+t}^\vv \overline{\mathscr L}^\vv  (f)(Y_{s_k,r}^{x})\big)\,\d r\,\d u\\
		&\quad+\int_{s_k}^{s_{k+1}}\E\big( (\overline{\mathscr L}^\vv  -\mathscr L_u  )\overline P_{\!\! s_{k+1},s+t}^\vv  (f)(Y_{s_k,u}^{x}) \big) \,\d u.
	\end{align*}
	Thus, applying \eqref{WW-38} enables us to derive that
	\begin{align*}
		|I_{k,2}^\vv(x)|\le C_0 \sum_{i=0}^3\bigg(& \big\|\nn^i \overline{\mathscr L}^\vv    f \big\|_\8\int_{s_k}^{s_{k+1}}\int_{s_k}^u     \big(\Phi(r)\E \mathcal V(Y_{s_k,r}^{ x}) +\E|\bar b^\vv(Y_{s_k,r}^{ x})-\bar b (Y_{s_k,r}^{ x})|\big)\,\d r\,\d u\\
		&+   \|\nn^i f\|_\8 \int_{s_k}^{s_{k+1}}  \big(\Phi(u)\E \mathcal V(Y_{s_k,u}^{ x}) +\E|\bar b^\vv(Y_{s_k,u}^{ x})-\bar b (Y_{s_k,u}^{ x})|\big) \,\d u\bigg).
	\end{align*}
Next, we find  that for all $u\in[s_k,s_{k+1}]$, 	
\begin{align*}
 \int_{\R^d}\E|\bar b^\vv(Y_{s_k,u}^{ x})-\bar b (Y_{s_k,u}^{ x})| \mu_{s_k} (\d x)&\le  \int_{\R^d}\sup_{s\le u\le s+t}\E|\bar b^\vv(Y_{u}^{ x})-\bar b (Y_{u}^{ x})| \mu_{s_k} (\d x)\\
 &\le  \sup_{s\le r\le s+t}\int_{\R^d}\sup_{s\le u\le s+t}\E|\bar b^\vv(Y_{u}^{ x})-\bar b (Y_{u}^{ x})| \mu_r (\d x)\\
 & =:\Gamma(s,s+t,\vv).
\end{align*}
This, besides the definition of $\Theta(x)$ and
	$s_{k+1}-s_k=\frac{t}{n}$, yields that
	\begin{equation}\label{D18}
		\begin{split}
			\mu_{s_k}(|I_{k,2}^\vv  |)&\le C_0  \sum_{i=0}^3\bigg(\frac{t}{n}\big\|\nn^i \bar{\mathscr L}^\vv    f \big\|_\8+ \|\nn^i f\|_\8  \bigg)\\
			&\quad\times\bigg(\Theta(x)\int_{s_k}^{s_{k+1}} \Phi(u)  \,\d u + (s_{k+1}-s_k) \Gamma(s,s+t,\vv)\bigg).
		\end{split}
	\end{equation}
	Now, combining \eqref{D17} with \eqref{D18} and {\rm(iii)}
	implies that
	\begin{equation*}
		\begin{split}
			\big|(\mu_s \overline P_{\!\! t}^\vv  )( f)-\mu_{s+t}(f)\big|
			 \le\frac{t^2}{n }\big\| (\overline{\mathscr L}^\vv )^2  f \big\|_\8 + C_0&\sum_{i=0}^3\bigg(\frac{t}{n}\big\|\nn^i \overline{\mathscr L}^\vv   (f)\big\|_\8+ \|\nn^i f\|_\8  \bigg)\\
			 &\times\bigg(
			\left[\sup_{r\ge0}\mu_r (\Theta)\right] \int_{s }^{s+t}  \Phi(u)  \,\d u+  t   \Gamma(s,s+t,\vv)\bigg).
		\end{split}
	\end{equation*}
Subsequently, 	  by sending $n\to\infty$ it follows that
\begin{align*}
		\big|(\mu_s \overline P_{\!\! t}^\vv  )( f)-\mu_{s+t}(f)\big|
	 &\le  C_0\sum_{i=0}^3  \|\nn^i f\|_\8 \bigg(
		\left[\sup_{r\ge0}\mu_r (\Theta)\right] \int_{s }^{s+t}  \Phi(u)  \,\d u+ t   \Gamma(s,s+t,\vv)\bigg).
\end{align*}	
Whence, by virtue of \eqref{WW-37}, the assertion \eqref{D19-} is available provided that we can show that
\begin{align}\label{WW-34}
\varlimsup_{\vv\downarrow0}\big|(\mu_s \overline P_{\!\! t}^\vv  )( f)-(\mu_s \overline P_{\!\! t}  )( f)\big| =0.
\end{align}

Note that the following Kolmogorov backward equations:
\begin{align*}	
\partial_tu(t,x)=\overline{\mathscr L}u(t,x), \quad u(0,x)=f(x); \quad \partial_tu^\vv(t,x)=\overline{\mathscr L}^\vv u^\vv(t,x), \quad u^\vv (0,x)=f(x)
\end{align*}
hold, where $u(t,x):=\overline P_{\!\! t} f(x)$ and $u^\vv(t,x):=P_t^\vv f(x)$. Thereby, we have that
\begin{align*}
\partial_tv^\vv(t,x)=\overline{\mathscr L}^\vv v^\vv(t,x)+(\overline{\mathscr L}^\vv-\overline{\mathscr L})u(t,x),
\end{align*}
where $v^\vv(t,x):=u^\vv(t,x)-u(t,x)$. 	Via Duhamel's principle, $v^\vv(t,x)$ can be expressed as follows:
\begin{align*}	
	v^\vv(t,x) =\int_0^t\bar P_{t-r}^\vv\<\bar b^\vv-\bar b,\nn \overline P_{\!\! r}f\>(x) \,\d r =\int_0^t \E\big(\<\bar b^\vv-\bar b,\nn \overline P_{\!\! r} f\>(\overline{Y}_{\!\! t-r}^{\vv,x})\big) \,\d r.
\end{align*}	
This, along with Fatou's lemma,  implies that
\begin{align*}	
\varlimsup_{\vv\downarrow0}\big|(\mu_s \overline P_{\!\! t}^\vv  )( f)-(\mu_s \overline P_{\!\! t}  )( f)\big|
&\le  \varlimsup_{\vv\downarrow0}\int_0^t\int_{\R^d}\E\big(|\bar b^\vv-\bar b| (\overline{Y}_{\!\! t-r}^{\vv,x})|\nn \overline P_{\!\! r}f|(\overline{Y}_{\!\! t-r}^{\vv,x})|\big)\,\d r\mu_s(\d x)\\
&\le C_0^*\|\nn f\|_\8\int_0^t\int_{\R^d}\varlimsup_{\vv\downarrow0}\E |\bar b^\vv-\bar b | (\overline{Y}_{\!\!  r}^{\vv,x})\,\d r\mu_s(\d x).
\end{align*}	
Correspondingly, \eqref{WW-34} is available right now by taking {\rm(iv)} into account.
\end{proof}

 To examine \eqref{WW-38}, we further need to introduce an assumption concerning   the difference between drift terms, which is presented precisely as below.
\begin{itemize}\it
	\item[$({\bf A})$]There are  functions  $\phi: \R_+\to\R_+$ and $V:\R^d\to[1,\infty)$ such that for all $t>0$ and $x\in \R^d$,
	\begin{align*}
		|b(t,x)-\bar b(x)|\le \phi(t)V(x).
\end{align*}\end{itemize}

In the sequel,
 we assume that the additive process $(Z_t)_{t\ge0}$  in the SDE \eqref{D12} is given  as in \eqref{RR-}, where   the time-dependent
L\'evy measure $(\nu_t(\d z))_{t\ge0}$ is fixed as follows: for all $t\ge0$,
\begin{align}\label{F3}
	\nu_t(\d z):=g_t^{-(d+ 2)}\bigg( a( z/{g_t})\I_{\{| z/{g_t}|\le1\}}+\frac{c_* }{| z/{g_t}|^{d+\alpha}}\I_{\{| z/{g_t}|>1\}} \bigg)\,\d z,
\end{align}
where $a(z)=a(|z|)$ for all $z\in \bar B_1$. In particular, the L\'evy measure $\nu_t(\d z)$ is rotationally symmetric.
Moreover, we suppose  that $   \si=(\nu(|\cdot|^2)/d)^{1/2}>0$ involved in the SDE \eqref{D11} so  the condition (i) in Theorem \ref{Th4.1} is satisfied.

Under the preceding setting, the following lemma    provides a quantitative discrepancy  between $(\mathscr L_r)_{r\ge0}$ and $\overline{\mathscr L}^\vv$ acting respectively on the semigroup $(\overline{P}^\vv_{\!\! t})_{t\ge0}$.

\begin{lemma}\label{lemma5}
	Assume that  $({\bf A} )$ holds and   that there exist constants $ \theta_1\in[0,\alpha-2), \theta_2\ge0$  and $C_0  >0$ such that for all  $t>0$, $x\in\R^d$, $f\in C_c^\infty(\R^d)$ and  $\vv\in(0,1]$,
	\begin{equation}\label{D8}
		\begin{split}
			&\|\nn \overline {P}_{\!\! t}^\vv  f \|_\8\le C_0\|\nn f\|_\8,\quad \|\nn^2 (\overline {P}_{\!\! t}^\vv  f) (x)\|_{\rm op}\le C_0\sum_{i=0}^2\|\nn^i f\|_\8(1+|x|^{\theta_1}),\\
			&\|\nn^3 (\overline {P}_{\!\! t}^\vv  f)(x) \|_{\rm op}\le C_0\sum_{i=0}^3\|\nn^i f\|_\8(1+|x|^{\theta_2}).
		\end{split}
	\end{equation}
	Then,
	there is a constant $C_0^*>0$ such that for all $0< r\le s< t $,  $x\in\R^d$, $f\in C_c^\infty(\R^d)$ and $\vv\in(0,1],$
	\begin{equation}\label{D2}
		\begin{split}
		\big| (\mathscr L_r  - \overline{\mathscr L}^\vv  )(\overline  P_{\!\! s,t}^\vv  f)(x)\big| \le C_0^*\sum_{i=0}^3\|\nn^if\|_\8  \big(&\phi(r)V(x)+ (1+|x|^{ \theta_1\vee\theta_2} )g_r^{\alpha_*} +|\bar b^\vv(x)-\bar b (x)|\big),
			\end{split}
	\end{equation}
where $\alpha_*:=((1\wedge(\alpha-2))/2)\wedge(\alpha-(2+\theta_1))$.

\end{lemma}

\begin{proof}
	Hereinafter, we assume  $\theta_1\in[0,\alpha-2)$, $\theta_2\ge0$,  $0<r\le s< t $, $x\in\R^d$, $f\in C_c^\infty(\R^d)$ and $\vv\in(0,1]$. Apparently, by recalling $   \si=(\nu(|\cdot|^2)/d)^{1/2}$ involved in the SDE \eqref{D11},
	we have that
	\begin{align*}
		(\mathscr L_r  - \overline{\mathscr L}^\vv )(\overline P_{\!\! s,t}^\vv f)(x)=&\<\nn (\overline P_{\!\! s,t}^\vv f)(x), b (r,x)-\bar b^\vv(x)\> +\Lambda^\vv (r,s,t,x)\\
		&-\frac{1}{2d} \nu(|\cdot|^2) \mbox{trace}\big(\nn^2 (\overline P_{\!\! s,t}^\vv  f)(x)\big),
	\end{align*}
	where
	\begin{equation}\begin{split}\label{F1}
			\Lambda^\vv (r,s,t,x): =&\int_{\R^d}\big((\overline P_{\!\! s,t}^\vv f)(x+z)-(\overline P_{\!\! s,t}^\vv f)(x)-\<\nn (\overline P_{\!\! s,t}^\vv f)(x),z\>\I_{\{|z|\le1\}}\big) \nu_r(\d z)\\
			=&\int_{\R^d}\big((\overline P_{\!\! s,t}^\vv f)(x+z)-(\overline P_{\!\! s,t}^\vv f)(x)-\<\nn (\overline P_{\!\! s,t}^\vv f)(x),z\> \big) \nu_r(\d z).
	\end{split}\end{equation}Indeed,  in \eqref{F1},  we used the fact that $\int_{\{|z|>1\}}\< \nn  (\overline P_{\!\! s,t}^\vv f)(x),z\>\,\nu_r(\d z)=0 $ by taking advantage of  the symmetry of the L\'evy measure $\nu_r(\d z)$ (which is true due to $a(z)=a(-z)$) and $\alpha>2$.

	By means of  $({\bf A} )$ and  the first hypothesis in \eqref{D8},
	we obtain that
	\begin{align}\label{G3}
		|\<\nn (\overline P_{\!\! s,t}^\vv f)(x), b(r,x)-\bar b^\vv(x)\>|\le C_0\|\nn f\|_\8\big(\phi(r)V(x)+|\bar b^\vv(x)-\bar b (x)|\big).
	\end{align}
	Next, applying the third-order Taylor expansion   yields   that
	\begin{align*}
		&\Lambda^\vv (r,s,t,x) -\frac{1}{2d} \nu(|\cdot|^2) \mbox{trace}\big(\nn^2 (\overline P_{\!\! s,t}^\vv  f)(x)\big)\\
		&= \int_{\{|z|\le 1\}}\int_0^1\int_0^u\int_0^\theta
		\nn_z
		\nn_z \nn_z (\overline P_{\!\! s,t}^\vv f)(x+\rho z)
		\,\d \rho\,\d \theta\,\d u \,\nu_r(\d z)\\
		&\quad +\int_{\{|z|> 1\}} \int_0^1\int_0^u\<\nn^2 (\overline P_{\!\! s,t}^\vv f)(x+\theta z)-\nn^2 (\overline P_{\!\! s,t}^\vv f)(x),z\otimes z\>_{\rm HS}\,\d\theta\,\d u \,\nu_r(\d z) \\
		&=: J_1^\vv(r,s,t,x)+J_2^\vv(r,s,t,x),
	\end{align*}
	where in the first equality we employed the fact that
	\begin{align*}
		\int_{\R^d}\<\nn^2 (\overline P_{\!\! s,t}^\vv f)(x),z\otimes z\>_{\rm HS} \nu_r (\d z)=\frac{1}{d} \nu(|\cdot|^2) \mbox{trace}\big(\nn^2 (\overline P_{\!\! s,t}^\vv  f)(x)\big)
	\end{align*}
	by taking the   structure  of $\nu_r(\d z)$, given in \eqref{F3}, into consideration and making use of
  the rotationally symmetry of the L\'evy measure $\nu_t(\d z).$

	By invoking   \eqref{F3} once more, besides the third prerequisite in  \eqref{D8},
	there exist constants $C_1,C_2>0$ such that
	\begin{align}\label{WW-36}
		\begin{split}
	&	|J_1^\vv(r,s,t,x)|\\
	&\le  C_0\sum_{i=0}^3\|\nn^if\|_\8\int_{\{|z|\le 1\}}\int_0^1\int_0^u\int_0^\theta  |z|^3 (1+|x+\rho z|^{\theta_2})  \,\d \rho\,\d \theta\,\d u\, \nu_r(\d z)\\
		&\le \frac{1}{6}C_02^{( \theta_2-1)^+}\sum_{i=0}^3\|\nn^if\|_\8\int_{\{|z|\le 1\}}\big( (1+|x|)^{ \theta_2 } |z|^3 +|z|^{3+ \theta_2} \big)  \,\nu_r(\d z)\\
		&\le \frac{1}{6}C_02^{(\theta_2-1)^+}\sum_{0=1}^3\|\nn^if\|_\8\\
		&\quad\times\Bigg((1+|x|)^{ \theta_2} g_r\bigg( \int_{\{|z|\le1\}}|z|^3a(z)\,\d z+\frac{c_*d\omega_d}{3-\alpha}\big(g_r^{\alpha-3}-1\big)\bigg)\\
		&\quad\quad\quad+g_r^{1+ \theta_2}\bigg( \int_{\{|z|\le1\}}|z|^{3+ \theta_2}a(z)\,\d z+\frac{c_*d\omega_d}{3+ \theta_2-\alpha}\big(g_r^{\alpha-3-\theta_2}-1\big)\bigg)\Bigg)\\
		&\le C_1 \sum_{i=0}^3\|\nn^if\|_\8\Big( \big(g_r^{1\wedge(\alpha-2)}\I_{\{\alpha\neq 3\}}+g_r(1+\log(1/{g_r}))\I_{\{\alpha=3\}}\big) (1+|x|)^{ \theta_2}\\
		&\quad\quad\quad\quad\quad\quad\quad\qquad+g_r^{(1+\theta_2)\wedge(\alpha-2) }\I_{\{\alpha\neq3+\theta_2\}}+g_r^{1+\theta_2}(1+\log(1/ {g_r})\I_{\{\alpha=3+\theta_2\}}\Big)\\
		&\le C_2 \sum_{i=0}^3\|\nn^if\|_\8 g_r^{(1\wedge(\alpha-2))/2} (1+|x|^{ \theta_2}),
	\end{split}
		\end{align}
where in the penultimate inequality we used the fact that $\lim_{r\downarrow0}(a^r-1)/r=\ln a$  for $a>0$	and in the last display we employed the basic inequality: $r\log(1/r)\le r^\theta/{(\e(1-\theta))}$ for all $r\in(0,1]$ and $\theta\in(0,1)$.

	In view  of $\theta_1\in[0,\alpha-2),$  we deduce from the second assumption  in \eqref{D8} that for some constant $C_3>0,$
	\begin{align}\label{WW-35}
		\begin{split}
		|J_2^\vv(r,s,t,x)|
		&\le C_0\sum_{i=0}^2\|\nn^i f\|_\8
		 \int_{\{|z|> 1\}} \int_0^1\int_0^s\big( 2+|x+\theta z|^{ \theta_1}  +|x|^{ \theta_1} \big) |z|^2\,\d\theta\,\d s \,\nu_r(\d z)\\
		&\le \frac{1}{2}C_0\big(1+2^{(\theta_1-1)^+}\big)\sum_{i=0}^2\|\nn^i f\|_\8 \int_{\{|z|> 1\}} \big( 1+|x|^{\theta_1}+|z|^{\theta_1}\big) |z|^2  \,\nu_r(\d z)\\
		&\le\frac{1}{2}c_*d\omega_dC_0\big(1+2^{(\theta_1-1)^+}\big)\sum_{i=0}^2\|\nn^i f\|_\8  g_r^{\alpha-2} \bigg(\frac{  1+|x|^{\theta_1}  }{\alpha-2}+\frac{ {g_r}^{- \theta_1 }}{\alpha-(2+\theta_1)}\bigg)\\
		&\le C_3 \sum_{i=0}^2\|\nn^i f\|_\8 g_r^{\alpha-(2+\theta_1)} (1+|x|^{\theta_1}).
	\end{split}
	\end{align}
	
At last,  \eqref{D2} follows by taking advantage of 	\eqref{G3} and combining \eqref{WW-36} with \eqref{WW-35}. 	
\end{proof}

 Concerning the case $\bar b(x)=-\lambda x|x|^{\gamma-1}$ for $\lambda>0$ and $\gamma>1 $ and the perturbed version $\bar b^\vv(x)=-\lambda x(\vv+|x|^2)^{(\gamma-1)/2}$,
the prerequisite \eqref{D8} will be  verified in detail in Lemma \ref{lemma-1}, where the super-Poincar\'e inequality and the high-order Bismut-Elworthy-Li formulas
play  an important role in deriving the finite-time and infinite-horizon bounds, respectively.

By combining  Theorem \ref{Th4.1} with Lemma \ref{lemma5}, we immediately have the following statement, which indicates that $(\mu_t)_{t\ge0}$ as  the law of $(Y_t)_{t\ge0}$
is   an asymptotic pseudotrajectory of $(\overline {P}_{\!\! t})_{t\ge0}$ provided that the function $\phi$ involved in (${\bf A} $) satisfies   $\lim_{t\to\infty}\phi(t)=0$.

\begin{proposition}\label{pro}
	Assume that $({\bf A})$ and  Assumptions {\rm(i)-(\rm iv)}  in Theorem $\ref{Th4.1}$ hold, where \eqref{WW-38} therein is replaced by \eqref{D8} and
	\begin{align}\label{D16}
	  \Theta(x):=\sup_{t\ge0}\E \big(V(Y_t^{x})+ |Y_t^{x}|^{ \theta_1\vee\theta_2 }\big),\quad \forall\, x\in\R^d.
	\end{align}
	Then, there exists a  constant  $C_1^* >0$ such that for any $t>0$, $s\ge0$ and $f\in C_c^\infty(\R^d)$,
	\begin{align}\label{D19}
		\big|(\mu_s \overline {P}_{\!\! t}  )( f)-\mu_{s+t} (f)\big|
		\le  C_1^* \sum_{i=0}^3 \|\nn^if\|_\8   \int_{s }^{s+t} \big(\phi(u) + g_u^{\alpha_*}\big) \,\d u,
	\end{align}
where	$\alpha_*:=((1\wedge(\alpha-2))/2)\wedge(\alpha-(2+\theta_1))$.	
	In particular, there exists a constant $C_2^*>0$  such that for all $t>0$ and
	$s\ge 0$,
	$${\rm d}_{\mathcal F_c}\big(  \mu_{s}\overline {P }_{\!\! t },\mu_{s+t} \big)\le C_2^*  \int_{s }^{s+t} \big(\phi(u) + g_u^{\alpha_*}\big) \,\d u, $$
	where $\mathcal F_c:=\{f\in C_c^\infty(\R^d): \sum_{i=0}^3 \|\nn^if\|_\8\le 1\}.$
\end{proposition}

\section{Proof of Theorem \ref{thm5-}}\label{sec6}
 Throughout this section, we focus on the SDE
  \eqref{WW-32}, where the associated $\bar b^\vv(x)$ is given as follows:  for all $x\in\R^d$ and $\vv\in(0,1]$,
 \begin{align}\label{WW-40}
  \bar b^\vv(x)=-\lambda x(\vv+|x|^2)^{(\gamma-1)/2}\quad \mbox{ with } \gamma>1 \quad \mbox{ and }\quad  \si= (\nu(|\cdot|^2)/d)^{{1}/{2}}.
 \end{align}
 Obviously, $\R^d\ni x\mapsto \bar b^\vv(x)$ is a $C^3(\R^d;\R^d)$-function,  although
 the original $b(x)=-\lambda x|x|^{\gamma-1}$ with $\gamma\in( 1,3)$ need not be.

 As a  preliminary warm-up, we first derive  the uniform-in-time  growth-bound estimates on $\nn (\overline {P}_{\!\! t}^\vv f)(x)$, $\nn^2 (\overline {P}_{\!\! t}^\vv f)(x)$ and $\nn^3 (\overline {P}_{\!\! t}^\vv f)(x)$, which further shows the validity of \eqref{D8}
 for the case we are interested in.

\begin{lemma}\label{lemma-1}
For any $\gamma>1\vee(3-d/2)$,  there exists a constant $C_0>0$ such that  for all $t>0$, $x\in \R^d$ and $f\in C_b^3(\R^d)$,
\begin{equation}\label{WW-27}
	\begin{split}
\|\nn \overline {P}_{\!\! t}^\vv f \|_\8 &\le \|\nn f\|_\8, \quad
\|\nn^2 (\overline {P}_{\!\! t}^\vv f)(x)\|_{\rm op}\le C_0 \sum_{i=0}^2 \|\nn^i f\|_\8(  1+|x|^{(\gamma-2)^+} ),\\
\|\nn^3 (\overline {P}_{\!\! t}^\vv f)(x)\|_{\rm op}&\le C_0\sum_{i=0}^3 \|\nn^i f\|_\8 \big(1+|x|^{2( \gamma-2)^+}+|x|^{ ( \gamma-3)^+}\big),
\end{split}
\end{equation}
where $(\overline {P}_{\!\! t}^\vv)_{t\ge0} $ is the semigroup of $(\overline {Y}_{\!\! t}^\vv)_{t\ge0}$ solving \eqref{D11} with $\bar b^\vv(x) $ and $\si $ being  given in \eqref{WW-40}.
\end{lemma}

\begin{proof}
In the following analysis, let $(\overline {Y}_{\!\! t}^{\vv,x})_{t\ge0}$ solve the SDE \eqref{WW-32}, where $\bar b^\vv(x) $ and $\si $ are given in \eqref{WW-40}, and set $t>0$, $x,y\in\R^d,$ $f\in C_b^3(\R^d)$ and $\vv\in(0,1]$.
By   It\^o's  formula and Young's inequality, for any $p>0$  there exist  constants $C_1,C_2>0$ such that
\begin{align*}
\d(1+|\overline {Y}_{\!\! t}^{\vv,x}|^2)^{p/2}&=-p\lambda  (1+|\overline {Y}_{\!\! t}^{\vv,x}|^2)^{ p/2 -1} |\overline {Y}_{\!\! t}^{\vv,x}|^2(\vv+|\overline {Y}_{\!\! t}^{\vv,x}|^2)^{ (\gamma -1)/2} \,\d t+pd\sigma^2(1+|\overline {Y}_{\!\! t}^{\vv,x}|^2)^{p/2-1} \,\d t\\
&\quad+p(p-2)\sigma^2(1+|\overline {Y}_{\!\! t}^{\vv,x}|^2)^{p/2-2}| \overline {Y}_{\!\! t}^{\vv,x}|^2\,\d t+\d M_t^{\vv,p} \\
&\le \Big(-\frac{1}{2}p\lambda |\overline {Y}_{\!\! t}^{\vv,x}|^{ p+\gamma-1  }+C_1\Big)  \,\d t+\d M_t^{\vv,p}\\
&\le \big(-2^{-(p+\gamma-3)/2} p\lambda (1+|\overline {Y}_{\!\! t}^{\vv,x}|^2)^{ (p+\gamma-1) /2 } +C_2\big)  \,\d t+\d M_t^{\vv,p}\\
&\le \big(-2^{-(p+\gamma-3)/2} p\lambda (1+|\overline {Y}_{\!\! t}^{\vv,x}|^2)^{  p  /2 } +C_2\big)  \,\d t+\d M_t^{\vv,p},
\end{align*}
where $(M_t^{\vv,p})_{t\ge0}$ is a martingale.
Thus, for any $p>0,$ the Gronwall inequality enables  us to derive that there exist constants $C_3,C_4>0$ such that
\begin{align}\label{WW-4}
 \E |\overline {Y}_{\!\! t}^{\vv,x}|^p\le  \e^{-C_3 t }(1+|x|^2)^{p/2}+C_4.
\end{align}

Below, we will verify \eqref{WW-27} according to $t\in (0,1]$ and $t>1$, respectively. First, assume that $t\in (0,1]$.
With \eqref{WW-4} at hand, in the following part we aim at showing  that, for   $\gamma>1\vee(3-d/2)$,  there exists a constant $C_5>0$ (independent of $\vv$) such that  for all   $t\in(0,1]$ and $j=1,2$,
\begin{equation}\label{WW-26}
	\begin{split}
\|\nn  \overline {Y}_{\!\! t}^{\vv,x}\|_{\rm op} &\le 1,\quad 	\E\|\nn^2 \overline {Y}_{\!\! t}^{\vv,x}\|_{\rm op}^{2j} \le  C_5 (1+|x|^{2j(\gamma-2)^+}),\\
 \E\|\nn^3 \overline {Y}_{\!\! t}^{\vv,x}\|_{\rm op}^2 &\le  C_5 \big(1+|x|^{4( \gamma-2)^+}+|x|^{2( \gamma-3)^+}\big).
\end{split}
\end{equation}
Once \eqref{WW-26} is available, \eqref{WW-27} holds true for all $t\in(0,1]$ and $\vv\in(0,1]$ by noting that
\begin{align*}
\|\nn (\overline {P}_{\!\! t}^{\vv}f)(x)\|_{\rm op}&\le \|\nn f\|_\8\E\| \nn \overline {Y}_{\!\! t}^{\vv,x}\|_{\rm op},\\
\|\nn^2(\overline {P}_{\!\! t}^\vv f)(x)\|_{\rm op}&\le\|\nn^2 f\|_\8\E \| \nn \overline {Y}_{\!\! t}^{\vv,x}\|_{\rm op}^2+\|\nn f\|_\8\E\|\nn^2\overline {Y}_{\!\! t}^{\vv,x}\|_{\rm op},\\
\|\nn^3(\overline {P}_{\!\! t}^\vv f)(x)\|_{\rm op}&\le\|\nn^3 f\|_\8\E \| \nn \overline {Y}_{\!\! t}^{\vv,x}\|_{\rm op}^3+3\|\nn^2 f\|_\8\E(\|\nn^2\overline {Y}_{\!\! t}^{\vv,x}\|_{\rm op}^2\|\nn\overline {Y}_{\!\! t}^{\vv,x}\|_{\rm op})\\
&\quad+\|\nn f\|_\8\E\|\nn^3\overline {Y}_{\!\! t}^{\vv,x}\|_{\rm op}.
\end{align*}

Below, we proceed to prove \eqref{WW-26}.
For $\gamma>1\vee(3-d/2)$, it is easy to see that  there exists a constant  $C_6>0$ such that for $t\in (0,1]$ and $j=1,2,$
\begin{equation}\label{WW-28}
	\begin{split}
\|\nn  \overline {Y}_{\!\! t}^{\vv,x}\|_{\rm op} &\le 1, \quad
 \|\nn^2 \overline {Y}_{\!\! t}^{\vv,x}\|_{\rm op}^{2j} \le C_6t^{2j-1} \int_0^t(\vv+ |\overline {Y}_{\!\! s}^{\vv,x}|^2)^{   j(\gamma -2)  }\,\d s ,\\
\ \|\nn^3 \overline {Y}_{\!\! t}^{\vv,x}\|_{\rm op}^2&\le C_6t^{3}\int_0^t (\vv+  |\overline {Y}_{\!\! t}^{\vv,x}|^2)^{ 2 ( \gamma -2 )}\,\d s+C_6t \int_0^t(\vv+ |\overline {Y}_{\!\! s}^{\vv,x}|^2)^{   \gamma -3   }\,\d s.
\end{split}
\end{equation}
When $\bar b^\vv(x)=-\lambda x(\vv+|x|^2)^{(\gamma-1)/2}$ and $\si=(\nu(|\cdot|^2)/d)^{{1}/{2}}$,
the associated SDE \eqref{WW-32} can be rewritten as below: for all $t>0,$
\begin{align*}
\d \overline {Y}_{\!\! t}^{\vv,x}=-\nn V_\vv(\overline {Y}_{\!\! t}^{\vv,x})\,\d t+\sigma\d W_t,
\end{align*}
where $V_\vv(x):= \frac{\lambda}{1+\gamma} (\vv+|x|^2)^{(1+\gamma)/2}$.
Obviously, the unique  IPM of $(\overline {Y}_{\!\! t}^\vv)_{t\ge0}$ admits the following form:
\begin{align*}
\mu_\vv(\d x)=C_{Z_\vv}^{-1} \exp(-2 V_\vv(x)/{\sigma^2})\,\d x,
\end{align*}
where $C_{Z_\vv}$ is the positive normalization constant,
and the corresponding infinitesimal generator $\bar{\mathscr L}^\vv$   is symmetric in $L^2(\mu_\vv)$. Since the semigroup $(\overline {P}_{\!\! t}^\vv )_{t\ge0}$ satisfies Wang's Harnack inequality,
  the transition kernel $p^\vv(t,x,y)$ is absolutely continuous with respect to the IPM $\mu_\vv$; see e.g. \cite[Theorem 1.4.1]{Wang} for related details. Let $q^\vv(t,x,y)$  be the transition density with respect to the invariant measure $\mu_\vv$, that is,
\begin{align*}
p^\vv(t,x,y)=C_{Z_\vv}^{-1}q^\vv(t,x,y)\exp(-2 V_\vv(y)/{\sigma^2}).
\end{align*}

According to \cite[Corollary 2.5]{Wang-b}, the IPM $\mu_\vv$ satisfies the following super-Poincar\'e inequality: for all  $f\in C_c^1(\R^d)$ and $r>0$,
\begin{align*}
\mu_\vv(f^2)\le r\mathscr E_\vv(f,f)+\beta_{\rm SP}(r)\mu_\vv(|f|)^2,
\end{align*}
where  $\mathscr E_\vv(f,f):=\mu_\vv(|\nn f|^2)$ and $ \beta_{\rm SP}(r):=\exp(C_7(1+r^{-(1+\gamma)/{(2\gamma)}}))$ for some constant $C_7>0$ independent of $\vv$. Thus,
\cite[Theorem 3.3.15]{Wang-c} enables us to derive that there is some constant $C_8>0$ (independent of $\vv\in(0,1]$) such that for all $t>0$,
\begin{align*}
\|\overline {P}_{\!\! t}^\vv\|_{1\to\8}\le\exp\big(C_8(1+t^{-(1+\gamma)/{(\gamma-1)}})\big).
\end{align*}
This, besides \cite[Proposition 3.3.11]{Wang-c}, implies that for all $t>0$ and $x,y\in \R^d$,
\begin{align*}
q^\vv(t,x,y)\le \exp\big(C_8(1+t^{-(1+\gamma)/{(\gamma-1)}})\big).
\end{align*}
 As a result, we derive that
\begin{align*}
p^\vv(t,x,y)\le C_{Z_\vv}^{-1} \exp\big(C_8(1+t^{-(1+\gamma)/{(\gamma-1)}})\big)\exp
 (- 2 V(y)/ {\sigma^2} ).
\end{align*}
Furthermore, note that
\begin{align*}
C_{Z_\vv}&\ge\int_{\R^d}\exp(- c_\star (1+|x|^2)^{(1+\gamma)/2} )\,\d x\\
&\ge\exp(- 2^{(\gamma-1)/2} c_\star    )\int_{\R^d}\exp(-  2^{(\gamma-1)/2}c_\star |x|^{1+\gamma}   )\,\d x\\
&=C_9:=\frac{\exp(- 2^{(\gamma-1)/2} c_\star    )d\omega_d\Gamma(d/{(1+\gamma)})}{(1+\gamma)(2^{(\gamma-1)/2}c_\star)^{d/{(1+\gamma)}}},
\end{align*}
where $c_\star:=\frac{2\lambda}{\sigma^2(1+\gamma)}$. Therefore,
we find that for any $q>-d$,
\begin{equation}\label{WW-25}
\begin{split}
\E|\overline {Y}_{\!\! t}^{\vv,x}|^q&\le \frac{1}{C_9}\exp \big(C_8(1+t^{-(1+\gamma)/{(\gamma-1)}})\big)\int_{\R^d}|y|^q\exp(-c_\star
 |y|^{1+\gamma} )\,\d y\\
 &=\frac{1}{C_9} \exp\big(C_8(1+t^{-(1+\gamma)/{(\gamma-1)}})\big)\frac{2\pi^{d/2}c_\star^{-\frac{q+d}{1+\gamma}}}{\Gamma(d/2)(1+\gamma)}\Gamma\Big(\frac{q+d}{1+\gamma}\Big),
\end{split}
\end{equation}
where in the identity we exploited the following identity:  for all $q>-d$,
\begin{align*}
\int_{\R^d}|y|^q\e^{-c_\star|y|^{1+\gamma}}\,\d y= \frac{2\pi^{d/2}c_\star^{-\frac{q+d}{1+\gamma}}}{\Gamma(d/2)(1+\gamma)}\Gamma\Big(\frac{q+d}{1+\gamma}\Big).
\end{align*}
 Consequently, \eqref{WW-26} is attainable by combining \eqref{WW-28}
with \eqref{WW-25}.

 In order to show that \eqref{WW-27} for all $t>1$ is still valid, we appeal to the Bismut-Elworthy-Li formula. By the semigroup property of $(\overline {P}_{\!\! t}^\vv)_{t\ge0}$,
  the  Bismut-Elworthy-Li formula up to third order enables us to  deduce that for all $t>1$, $x,y_1,y_2,y_3\in\R^d$ and  $f\in\mathscr B_b(\R^d)$,
\begin{align*}
	\nn_{y_1}\nn_{y_2}(\overline {P}_{\!\! t}^\vv f)(x)&= \E \big((\overline {P}_{\!\! t-1}^\vv f)(\overline {Y}_{\!\! 1}^{\vv,x})\mathcal M_1^{\vv, (2) } \big),\quad   \nn_{y_1}\nn_{y_2}\nn_{y_3}(\overline {P}_{\!\! t}^\vv f)(x) = \E \big((\overline {P}_{\!\! t-1}^\vv f)(\overline {Y}_{\!\! 1}^{\vv,x})\mathcal M_1^{\vv, (3) } \big),
\end{align*}
where $(\mathcal M_t^{\vv,(2)})_{t>0}$ and $(\mathcal M_t^{\vv,(3)})_{t>0}$ are given respectively as below: for all $t>0,$
\begin{align*}
\mathcal M_t^{\vv,(2)} :&=\frac{1}{t}\int_0^t\<\nn_{y_1}\nn_{y_2}\overline {Y}_{\!\! s}^{\vv,x},\d W_s\>\\
&\quad+\frac{1}{t^2} \bigg( \int_0^t\<\nn_{y_1}\overline {Y}_{\!\! s}^{\vv,x},\d W_s\>\int_0^t\<\nn_{y_2}\overline {Y}_{\!\! s}^{\vv,x},\d W_s\>-\int_0^t\<\nn_{y_1}\overline {Y}_{\!\! s}^{\vv,x},\nn_{y_2}\overline {Y}_{\!\! s}^{\vv,x}\>\,\d s \bigg),
\end{align*}
and for $\Lambda:=\{(1,2,3),(1,3,2),(2,3,1)\}$,
\begin{align*}
\mathcal M_t^{\vv,(3)} :&=\frac{1}{t^3}\int_0^t\<\nn_{y_1} \overline {Y}_{\!\! s}^{\vv,x},\d W_s\>\int_0^t\<\nn_{y_2} \overline {Y}_{\!\! s}^{\vv,x},\d W_s\>\int_0^t\<\nn_{y_3} \overline {Y}_{\!\! s}^{\vv,x},\d W_s\>\\
	&\quad+\frac{1}{t^2}\sum_{(i,j,k)\in \Lambda}\bigg(\int_0^t\<\nn_{y_i}\nn_{y_j}\overline {Y}_{\!\! s}^{\vv,x},\d W_s\>\int_0^t\< \nn_{y_k}\overline {Y}_{\!\! s}^{\vv,x},\d W_s\>\\
	&\quad \quad\quad\quad \quad\quad\quad \quad-\int_0^t\<\nn_{y_i}\overline {Y}_{\!\! s}^{\vv,x}, \nn_{y_j}\overline {Y}_{\!\! s}^{\vv,x}\>\d s\int_0^t\<\nn_{y_k}\overline {Y}_{\!\! s}^{\vv,x}, \d W_s\>\bigg)\\
	&\quad+\frac{1}{t}\int_0^t\<\nn_{y_1}\nn_{y_2}\nn_{y_3}\overline {Y}_{\!\! s}^{\vv,x},\d W_s\>-\frac{1}{t^2}\int_0^t\<\nn_{y_1}\nn_{y_2}\overline {Y}_{\!\! s}^{\vv,x},\nn_{y_3}\overline {Y}_{\!\! s}^{\vv,x}\>\,\d s\\
	&\quad-\frac{1}{t^2}\int_0^t\int_s^t\<(D_s\nn_{y_1}\nn_{y_2}\overline {Y}_{\!\! r}^{\vv,x})^\top\nn_{y_3}\overline {Y}_{\!\! s}^{\vv,x},\d W_r\>\,\d s
\end{align*}
with $D_\cdot$ being the Malliavin divergence operator.
Subsequently,
from  $\|\nn  \overline {Y}_{\!\! t}^{\vv,x}\|_{\rm op}  \le 1$,  \eqref{WW-26} as well as It\^o's isometry,
 it follows    that for some constant $c_1>0,$
\begin{align*}
	\E|\mathcal M_1^{\vv,(2)}| \le  \bigg(  \int_0^1\E|\nn_{y_1}\nn_{y_2}\overline {Y}_{\!\! s}^{\vv,x}|^2 \d  s\bigg)^{1/2} +2|y_1|\cdot|y_1|  \le c_1  (  1+|x|^{(\gamma-2)^+}   )|y_1|\cdot|y_1|.
\end{align*}
Thereby, we arrive at the estimate below: for all $t>1$,
\begin{align*}
	\|\nn^2( \overline {P}_{\!\! t}^\vv f)(x)\|_{\rm op}\le c_1\|f\|_\8  (  1+|x|^{(\gamma-2)^+}).
\end{align*}
Once more, by invoking $\|\nn  \overline {Y}_{\!\! t}^{\vv,x}\|_{\rm op}  \le 1$ and \eqref{WW-26}, we find that for some constant $c_2>0,$
\begin{align*}
	\E|\mathcal M_1^{\vv,(3)}|\le c_2\bigg(1+ |x|^{2( \gamma-2)^+}+|x|^{ ( \gamma-3)^+}+\bigg(\int_0^t\int_s^t\E\| D_s\nn ^2 \overline {Y}_{\!\! r}^{\vv,x}  \|^2_{\rm op}\,\d r\,\d s\bigg)^{1/2}\bigg)|y_1|\cdot|y_2|\cdot|y_3|.
\end{align*}
Via the chain rule,  we have  that for all $t\ge s $,
\begin{align*}
 \d D_s\nn_{y_1}\nn_{y_2}\overline {Y}_{\!\! t}^{\vv,x}
 &= \big(\nn \bar b^\vv(\overline {Y}_{\!\! t}^{\vv,x})D_s\nn_{y_1}\nn_{y_2}\overline {Y}_{\!\! t}^{\vv,x}+\nn^2\bar b^\vv(\overline {Y}_{\!\! t}^{\vv,x})(D_s\overline {Y}_{\!\! t}^x,\nn_{y_1}\nn_{y_2}\overline {Y}_{\!\! t}^{\vv,x})\big)\,\d t\\
 &\quad+ \big(\nn^3\bar b^\vv(\overline {Y}_{\!\! t}^{\vv,x})(D_s\overline {Y}_{\!\! t}^{\vv,x},\nn_{y_1} \overline {Y}_{\!\! t}^{\vv,x},\nn_{y_2} \overline {Y}_{\!\! t}^{\vv,x} )+
 \nn^2\bar b^\vv(\overline {Y}_{\!\! t}^{\vv,x})(D_s\nn_{y_1} \overline {Y}_{\!\! t}^{\vv,x},\nn_{y_2} \overline {Y}_{\!\! t}^{\vv,x} )\\
 &\quad\quad\quad+\nn^2\bar b^\vv(\overline {Y}_{\!\! t}^{\vv,x})(\nn_{y_1} \overline {Y}_{\!\! t}^{\vv,x},D_s\nn_{y_2} \overline {Y}_{\!\! t}^{\vv,x}) \big)\,\d t,
\end{align*}
where for all $t\ge s,$
\begin{align*}
 \d D_s\overline {Y}_{\!\! t}^{\vv,x}&=  \nn b(\overline {Y}_{\!\! t}^{\vv,x})D_s\overline {Y}_{\!\! t}^{\vv,x}\,\d t\quad \mbox{ with } \quad D_s\overline {Y}_{\!\! s}^{\vv,x}=\si I_d,\\
 \d D_s\nabla_y\overline {Y}_{\!\! t}^{\vv,x}&= \big(\nn \bar b^\vv(\overline {Y}_{\!\! t}^{\vv,x})D_s\nabla_y\overline {Y}_{\!\! t}^{\vv,x}+\nn^2\bar b^\vv(\overline {Y}_{\!\! t}^{\vv,x})(D_s\overline {Y}_{\!\! t}^{\vv,x},\nabla_y\overline {Y}_{\!\! t}^{\vv,x})\big)\,\d t.
\end{align*}
By virtue of
\begin{align*}
\|D_s\overline {Y}_{\!\! t}^{\vv,x}\|_{\rm op}\le |\sigma|\quad \mbox{ and } \quad \|D_s\nabla_y\overline {Y}_{\!\! t}^{\vv,x}\|_{\rm op}\le c_3|y|\int_s^t(\vv+ |\overline {Y}_{\!\! r}^{\vv,x}|^2)^{   (\gamma -2)/2  }\,\d r,\quad \forall t\ge s\ge0
\end{align*}
for some constant $c_3>0,$ along with  $\|\nn  \overline {Y}_{\!\! t}^{\vv,x}\|_{\rm op}  \le 1$,
there exist constants $c_4,c_5>0$ such that for all $0<s\le r\le1$ and $\gamma>1\vee(3-d/2)$,
 \begin{align*}
\E\|D_s \nn^2 \overline {Y}_{\!\! r}^{\vv,x})\|_{\rm op}^2&\le c_4 \int_s^r\big(\E\big(\|\nn^2\bar b^\vv(\overline {Y}_{\!\! u}^{\vv,x})\|_{\rm op}^2 \|\nn^2\overline {Y}_{\!\! u}^{\vv,x}\|_{\rm op}^2\big) +\E\|\nn^3\bar b^\vv(\overline {Y}_{\!\! u}^{\vv,x})\|_{\rm op}^2\big)\,\d u\\
&\quad+c_4\int_s^r\|\nn^2\bar b^\vv(\overline {Y}_{\!\! v}^{\vv,x}) \|_{\rm op}^2\int_s^v(\vv+ |\overline {Y}_{\!\! u}^{\vv,x}|^2)^{    \gamma -2   }\d u\,\d v\\
&\le c_5\big(1+|x|^{4(\gamma-2)^+}+|x|^{2(\gamma-3)^+}\big).
 \end{align*}
Based on the analysis above, the assertion \eqref{WW-27} for all $t>1$ is established.
\end{proof}

With Proposition \ref{pro} at hand, we proceed  to accomplish the

\begin{proof}[Proof of Theorem $\ref{thm5-}$]
 Below, let $q =  2\beta/{ (1+\gamma)}\in(0,1)$ in case of $\beta<(1+\gamma)/2$.
 For
  $g_t=\e^{-t/2}$ with $\varphi_t= \e^t$, and
  \begin{align}\label{WPPP-1}
 g_t =  (1+(1- q)t)
 ^{-\frac{q}{2(1- q)}} \quad \mbox{ with } \quad  \varphi_t:= (1+(1-   q)t)^{\frac{1}{1-   q}},
  \end{align}
 the  process  $(Y_t)_{t\ge0}:=(g_tX_{\varphi_t})_{t\ge0} $
solves respectively the following SDEs: for $t>0,$
\begin{equation}\label{F6}\begin{split}
\d Y_t
 &=\big(-  Y_t/2 +   \lambda \e^{-(\beta -(\gamma+   1)/2)t } ( 1+\e^{-t} )^{-\beta} Y_t(\e^{- t }+| Y_t|^2)^{(\gamma-1)/2} \big)\,\d t+ \d Z_t,\\
\d Y_t,
 &=  \bigg(-\frac{ Y_t }{ 2(1+(1- q)t) }   +   \lambda \big(1+1/{\varphi_t}\big)^{-\beta} Y_t \big(g_t^2+|Y_t|^2\big)^{(\gamma-1)/2}\bigg)\,\d t+  \d  Z_t,
\end{split}\end{equation}
where $(Z_t)_{t\ge0}$ is an additive process with the L\'{e}vy measure $(\nu_t(\d z))_{t\ge0}$ being given in \eqref{F3}.

For  case  (\rm I): $\beta>(1+\gamma)/2$,
we take
\begin{align*} b(t,x)=-x/2+\lambda \e^{-(\beta -(\gamma+   1)/2)t } ( 1+\e^{-t} )^{-\beta} x(\e^{- t }+| x|^2)^{(\gamma-1)/2}~ \mbox{ and }~  \bar b(x)=-x/2.
\end{align*}
 Obviously, \eqref{D8} with $\overline {P}_{\!\! t}^\vv$ therein being replaced by $\overline {P}_{\!\! t} $ holds true with $\theta_1=\theta_2=0,$ and for $\gamma>1,$
 \begin{align*}
 |b(t,x)-\bar b(x)|\le |\lambda| \e^{-(\beta -(\gamma+   1)/2)t }   (1+| x|^2)^{ \gamma/2 },
 \end{align*}
 so that $({\bf A} )$ is satisfied with $\phi(t)=|\lambda| \e^{-(\beta -(\gamma+   1)/2)t }$ and $V(x)=(1+| x|^2)^{ \gamma/2 }$. Furthermore, in case of $\lambda<0$
 or $\lambda>0$ with $\gamma\in(0,1]$,
 it follows that
  for some constants $c_1,c_2>0,$
\begin{align}\label{KL}
\<x,b (t,x)\>\le -\big(1/2-c_1\e^{-(\beta -(\gamma+   1)/2)t }\big)|x|^2+c_2.
\end{align}
Moreover,  we find from \eqref{F3} that
\begin{equation}\label{WW-42}
	\begin{split}
\sup_{t\ge0}\bigg(\int_{\{|z|\le 1\}}|z|^2\,\nu_t(\d z)\bigg)= &\int_{\{|z|\le 1\}}|z|^2  a(z) \,\d z+c_*\sup_{t\ge0}\bigg(\int_{\{1<|z|<1/{g_t}\}}\frac{|z|^2}{|z|^{d+\alpha}}\,\d z\bigg)\\
\le&\int_{\{|z|\le 1\}}|z|^2  a(z) \,\d z+c_* \bigg(\int_{\{|z|>1\}}\frac{|z|^2}{|z|^{d+\alpha}}\,\d z\bigg) <\8,
\end{split}
\end{equation}
and that for any  $t\ge0$ and $p\in(0,\alpha)$ with $\alpha>2$,
\begin{align}\label{WW-41}
\int_{\{|z|>1\}}|z|^p\,\nu_t(\d z)= c_*\e^{-\frac{1}{
2}(p-2)t}\int_{\{|z|>1/{g_t}\}}\frac{|z|^p  }{|z|^{d+\alpha}}   \,\d z=\frac{c_*d\omega_d}{(\alpha-p)\e^{
(\alpha-2)t/2}}\le \frac{c_*d\omega_d}{\alpha-p}.
\end{align}
Hence,   \eqref{EE5} is fulfilled. Next, with the aid of   \eqref{KL} and Proposition \ref{lem3}, for any $p\in(0,\alpha)$, there exists a constant  $c_3:=c_3(p)>0 $ such that
\begin{equation}\label{e:sss}
\sup_{t\ge0}\E  |Y_t |^p\le c_3(1+\E  |Y_0 |^p).
\end{equation}

Below, let $(\mu_t)_{t\ge0}$ be the law of $(Y_t)_{t\ge0}$ solving the SDEs in \eqref{F6} and
$\mathscr{L}_{Y_0}=\mu_0=\mu\in\mathscr P(\R^d)$.
Therefore, by applying Proposition \ref{pro}, there exists a  constant  $c_4 >0$
(which is dependent linearly on $\mu(|\cdot|^{\gamma})$)
 such that for all
 $s,t\ge0$,
$$
 {\rm d}_{\mathcal F_c}\big(\mu_{s+t},\mu_s\overline {P }_{\!\! t} \big)
 \le  c_4
 \int_{s}^{s+t} \big(\e^{-(\beta -(\gamma+   1)/2)u }  + \e^{- \alpha_*u/2}\big)  \,\d u,
$$
where  $(\overline {P}_{\!\! t})_{t\ge0}$ is the Markov semigroup generated by $(\overline Y_{\!\! t})_{t\ge0}$ solving the first SDE in \eqref{e:limit} with $\eta=2$. So,  with the help of the fact that for all $s\ge1$ and $t\ge0,$
\begin{align*}
\varphi_s^{-1}=\ln s, \quad \varphi_{s+t}^{-1}-\varphi_s^{-1}=\ln(1+t/s),\quad g_{\varphi_s^{-1}}=s^{-1/2},
\end{align*}
we deduce that for all  $s\ge1 $ and $t\ge0$,
\begin{align*}
{\rm d}_{\mathcal F_c}\big(\bar\mu_{s+t},\bar\mu_s\overline {P }_{\!\! \ln(1+t/s)} \big)
\le c_4
 \int_{\ln s}^{\ln(s+t)} \big(\e^{-(\beta -(\gamma+   1)/2)u }  + \e^{- \alpha_*u/2}\big)  \,\d u,
\end{align*}
where $\bar\mu_s $ means the law of $ X_s/{s^{  1/2 }} $. Apparently,
 there exists a   constant   $c_5 >0 $
(which   depends linearly on $\mu(|\cdot|)$ and $\pi(|\cdot|)$)
such that for all $t\ge1$ and $\theta\in(0,1),$
\begin{align*}
{\rm d}_{\mathcal F_c}
\big(\bar\mu_{t^\theta}\overline {P }_{\!\!(1-\theta) \ln t },\pi \big)
\le \mathcal W_1\big(\bar\mu_{t^\theta}\overline {P }_{\!\!(1-\theta) \ln t },\pi \big) \le  \mathcal W_1\big(\bar\mu_{t^\theta},\pi \big) t^{-
(1-\theta)/2}\le c_5t^{-
(1-\theta)/2},
\end{align*}
where $\pi$ is the unique IPM associated with $(\overline Y_{\!\! t})_{t\ge0}$ solving
 the first   SDE in \eqref{e:limit} with $\eta=2$,
 and in the last inequality we used \eqref{e:sss}.
Then,
by the triangle inequality, we deduce that for all $\theta\in(0,1)$ and $t\ge1,$
\begin{equation}\label{D1}
\begin{split}
{\rm d}_{\mathcal F_c}\big(\bar\mu_{t},\pi\big)&={\rm d}_{\mathcal F_c}\big(\bar\mu_{t^\theta+t-t^\theta},\pi\big)\\
&\le {\rm d}_{\mathcal F_c}\big(\bar\mu_{t},\bar\mu_{t^\theta}\overline {P }_{\!\! (1-\theta)\ln t } \big)+{\rm d}_{\mathcal F_c}\big(\bar\mu_{t^\theta}\overline {P }_{\!\!(1-\theta) \ln t },\pi \big)\\
&\le c_4
\int_{\theta\ln t }^{\ln t} \big(\e^{-(\beta -(\gamma+   1)/2)u } + \e^{-
\alpha_*u/2}\big)  \,\d u+c_5t^{-
(1-\theta)/2}.
\end{split}
\end{equation}
Furthermore, we have that for all $\kk>0$, $s\ge1$ and $t\ge0,$
\begin{align*}
	\int_{\ln s}^{\ln(s+t)}\e^{-\kk u}\,\d u=\frac{1}{\kk}s^{-\kk}\big(1-(1+t/s)^{-\kk}\big).
\end{align*}
This, besides \eqref{D1}  with $\theta=1/2$, subsequently implies that for all $t\ge1$,
\begin{align*}
	{\rm d}_{\mathcal F_c}\big(\bar\mu_{t},\pi\big)
	&\le c_4 \left(\frac{t^{-  (2\beta-\gamma-1)/4}}{ \beta-(\gamma+1)/2}+\frac{ 2t^{-  \alpha^*/4}}{\alpha^*}\right)+c_5t^{-1/4  }.
\end{align*}
Therefore, the assertion \eqref{D26} is provable for  case  (\rm I).

For  case  (\rm II): $\beta=(1+\gamma)/2$,  we take
 \begin{align*}
b(t,x)=-x/2+\lambda  ( 1+\e^{-t} )^{-\beta} x(\e^{- t }+| x|^2)^{(\gamma-1)/2} \quad \mbox{ and } \quad \bar b(x)=-x/2+\lambda x|x|^{\gamma-1}.
\end{align*}
Whence, for $\gamma>1$, we obtain from \eqref{EP-23} with $\alpha=2$ that for
some constants $c_6,c_7>0 $,
\begin{equation*}
	\begin{split}
	 |b(t,x)-\bar b(x)|
		 &\le  |\lambda| \beta\e^{-t}(1 +| x|^2)^{ \gamma/2 }+c_6
			\e^{-(2\wedge(\gamma-1)) t/2}    (1 +| x|^{1\vee(\gamma-2)})   \\
			&\le c_7
			\e^{-(2\wedge(\gamma-1)) t/2}  (1+|x|^\gamma),
			\end{split}
\end{equation*}
so that
$({\bf A})$  is      satisfied for $V(x)=1+|x|^\gamma$. For $\vv\in(0,1]$, let $\bar b^\vv(x)=-x/2+\lambda x(\vv+|x|^2)^{(\gamma-1)/2}$. It is easy to see from \eqref{WW-39} that for some constant $c_8 >0 $,
\begin{equation}\label{WW-43}
	\begin{split}
|\bar b^\vv(x)-\bar b (x)|&\le |\lambda|\cdot|x|\cdot|(\vv+|x|^2)^{(\gamma-1)/2}-|x|^{\gamma-1}|\\
&\le  |\lambda|(1\vee((\gamma-1)/2))|x|\big(\vv^{(\gamma-3)/2}\vee(\vv+|x|^2)^{(\gamma-3)/2}\big)\vv \\
&\le c_8(1+|x|^{(\gamma-2)^+})\vv^{1\wedge(\gamma/2)}.
\end{split}
\end{equation}
Next, by virtue of $\lambda<0,$
we obtain that
\begin{align*}
\<x,b(t,x)\>\vee\<x,\bar b(x)\>\vee\<x,\bar b^\vv(x)\> \le -|x|^2/2.
\end{align*}
This, besides  \eqref{WW-42} and \eqref{WW-41},  implies that, for any $p>0,$ there is a constant  $c_9 :=c_9(p)>0$ such that
\begin{align}\label{WW-44}
\E|\overline Y_{\!\! t}|^p+\E|\overline Y_{\!\! t}^\vv|^p\le c_9 (1+\E|\overline Y_{\!\! 0}|^p) \quad \mbox{ and } \quad \E  |Y_t |^p\le c_9(1+\E  |Y_0 |^p)
\end{align}
where $(\overline Y_{\!\! t})_{t\ge0}$   solves the second SDE in \eqref{e:limit} with $\eta=2$, and $(\overline Y_{\!\! t}^\vv)_{t\ge0}$ and $(Y_t)_{t\ge0} $  are determined respectively by the subsequent SDEs:
\begin{align*}
\d \overline Y_{\!\! t}^\vv
&=\big(-  \overline Y_{\!\! t}^\vv/2 +   \lambda   \overline Y_{\!\! t}^\vv(\vv+| \overline Y_{\!\! t}^\vv|^2)^{ (\gamma-1)/2} \big)\,\d t+ (\nu(|\cdot|^2)/d)^{{1}/{2}}\d W_t,\\
 \d Y_t
 &=\big(-  Y_t/2 +   \lambda  ( 1+\e^{-t} )^{-\beta} Y_t(\e^{- t }+| Y_t|^2)^{{1}/{2}(\gamma-1)} \big)\,\d t+ \d Z_t.
\end{align*}
Consequently, by invoking \eqref{WW-43} and \eqref{WW-44}, assumptions
 \eqref{WW-37}, (\rm iii) and (\rm iv) in Theorem \ref{Th4.1} are satisfied; see Remark \ref{remark5.2} for more details.
In addition,
\eqref{WW-38} is valid as long as $ 1\vee(3-d/2)<\gamma<\alpha$ by taking advantage of Lemma \ref{lemma5} and
Lemma \ref{lemma-1}. Accordingly,
 concerning  case  (\rm II), the  assertion \eqref{D26} is also available by following   the procedure to derive \eqref{D26} as in     case  (\rm I).

For case (\rm III): $\beta<(1+\gamma)/2$,  we set
\begin{align*}
b(t,x):=  -\frac{ x }{ 2(1+\kk t) }   +   \lambda \big(1+1/{\varphi_t}\big)^{-\beta} x \big(g_t^2+|x|^2\big)^{(\gamma-1)/2}  \quad   \mbox{ and }  \quad  \bar b(x):=\lambda x|x|^{\gamma-1},
\end{align*}
where $\kk:=1-q.$
Thus, by virtue of \eqref{WPP-1}, there exists a constant $c_1^*>0$ such that
\begin{align*}
|b(t,x)-\bar b(x)|\le   c_1^*(1+\kk t)^{-\alpha_1}(1 +|x|^\gamma ),
\end{align*}
where  $\alpha_1:=1\wedge  (q(1\wedge((\gamma-1)/2))/\kk). $
Therefore,
$({\bf A} )$ holds true for $V(x)=1 +|x|^\gamma$. It is easy to see that
 for some constants $c_2^*,c_3^*>0,$
\begin{align*}
\<x,b(t,x)\>\le (\lambda-c_2^*/{\varphi_t})
|x|^{1+\gamma}+c_3^*.
\end{align*}
Thus, \eqref{WW-44} is still valid so  assumptions  \eqref{WW-37}, (\rm iii) and (\rm iv) in Theorem \ref{Th4.1} remain true, where the corresponding  $(Y_t)_{t\ge0} $  solves the second SDE in \eqref{F6} and
$(\overline Y_{\!\! t}^\vv)_{t\ge0}$ is  governed  by the   SDE below:
\begin{align*}
	\d \overline Y_{\!\! t}^\vv
	 =   \lambda   \overline Y_{\!\! t}^\vv(\vv+| \overline Y_{\!\! t}^\vv|^2)^{ (\gamma-1)/2}  \,\d t+ (\nu(|\cdot|^2)/d)^{{1}/{2}}\d W_t.
\end{align*}
Recall that \eqref{WW-38} is  true  by combining  Lemma \ref{lemma5} with
Lemma \ref{lemma-1} in case of   $ 1\vee(3-d/2)<\gamma<\alpha$.
Subsequently,   once $ 1\vee(3-d/2)<\gamma<\alpha$,
we obtain  from Proposition \ref{pro} that there exists a  constant  $c_4^* >0$
(which    depends linearly  on $\mu(|\cdot|^\gamma)$)
such that for all
 $s,t\ge0$,
\begin{align}\label{KL-2}
 {\rm d}_{\mathcal F_c
 }\big(\mu_{s+t},\mu_s \overline {P}_{\!\! t}  \big)
 \le  c_4^*   \int_{s}^{s+t} (1+\kk u)^{-(\alpha_1\wedge\alpha_2)} \,\d u,
\end{align}
where $(\bar {P_t})_{t\ge0}$ is the Markov semigroup corresponding to $(\overline Y_{\!\! t})_{t\ge0}$, which is governed by  the third  SDE in \eqref{e:limit} with $\eta=2$ and $\gamma_0=0$, and  $\alpha_2:=  q\alpha_*/{(2\kk)}.$

Direct calculations show  that for $s\ge1$ and $t\ge0,$
\begin{align*}
 \varphi_s^{-1}= ( s^{\kk}-1)/\kk,\quad g_{\varphi_s^{-1}}=s^{-q/2},\quad h_{s,s+t}:=\varphi_{s+t}^{-1}-\varphi_s^{-1}= ((s+t)^\kk-s^\kk)/{\kk}.
 \end{align*}
Thus,
 it follows from \eqref{KL-2} that for all  $s\ge1 $
 and $t\ge0$,
 \begin{align*}
 	{\rm d}_{\mathcal F_c
 	}\big(\bar\mu_{s+t},\bar\mu_s\overline {P }_{\!\! h_{s,s+t}} \big)
 	\le  c_4^*    \int_{\varphi_s^{-1}}^{\varphi_{s+t}^{-1}} (1+\kk u)^{-(\alpha_1\wedge\alpha_2)}  \,\d u,
 \end{align*}
 where $\bar \mu_s $ denotes the law of $s^{-q/2}X_s  $. Accordingly,
 the triangle inequality yields  that for all $t\ge1$,
\begin{align*}
	{\rm d}_{\mathcal F_c}\big(\bar\mu_{t+t^\varrho},\pi\big)
	&\le {\rm d}_{\mathcal F_c}\Big(\bar\mu_{t +t^\varrho},\bar\mu_{t}\overline {P }_{\!\! h_{t,t+t^\varrho}} \Big)+{\rm d}_{\mathcal F_c}\Big(\bar\mu_{t }\overline {P }_{\!\! h_{t,t+t^\varrho}},\pi \Big)\\
	&\le c_4^*\int_{(t^{\kk}-1)/\kk}^{((t+t^\varrho)^\kk-1)/\kk}(1+\kk u)^{-(\alpha_1\wedge\alpha_2)}  \,\d u+{\rm d}_{\mathcal F_c}\Big(\bar\mu_{t}\overline {P }_{\!\! h_{t,t+t^\varrho}},\pi \Big),
\end{align*}
where $\varrho:=\frac{1}{2}(1+q) \in (0,1)$, and
$\pi$ is the unique IPM associated with $(\overline Y_{\!\! t})_{t\ge0}$ determined by
 the third   SDE in \eqref{e:limit} with $\eta=2$ and $\gamma_0=0$.
Moreover, we find that  for   all $ t >0$, $s\ge1$ and $ \theta>0$,
\begin{align*}
 	\int_{(s^\kk-1)/\kk}^{((s+t)^\kk-1)/\kk}   (1+\kk u)^{ -\theta } \,\d u&= \ln(1+t/s) \I_{\{\theta =1\}}  +\frac{(s+t)^{\kk(1-\theta)}-s^{\kk(1-\theta )}}{\kk(1-\theta )}
 	\I_{\{\theta \neq1\}}.
  \end{align*}
 This implies that for all $t\ge1$,
 \begin{align*}
 & \int_{(t^\kk-1)/\kk}^{((t+t^\varrho)^\kk-1)/\kk}(1+\kk u)^{-(\alpha_1\wedge\alpha_2)}   \,\d u\\
 &=   \ln (1+t^{\varrho-1})\I_{\{\alpha_1\wedge\alpha_2 =1\}}+\frac{(t+t^\varrho)^{\kk(1-(\alpha_1\wedge\alpha_2))}-t^{\kk(1-(\alpha_1\wedge\alpha_2))}}{\kk(1-(\alpha_1\wedge\alpha_2))}\I_{\{\alpha_1\wedge\alpha_2\neq1\}}\\
 &\le t^{\varrho-1}  +t^{\varrho+\kk(1-(\alpha_1\wedge\alpha_2))-1}\I_{\{ \alpha_1\wedge\alpha_2\in(0,1) \}},
 \end{align*}
 where in the inequality we employed the basic inequality: $\ln(1+r)\le r$ for all $r>-1,$
 and that  $((s+t)^{\theta}-s^{\theta })/{\theta }\le  ts^{\theta -1}$  for all $s>0,t\ge0$ and $\theta \in(0,1)$.
 Furthermore, note that
 \begin{align*}
(2^{\kk }-1)r\le (1+r)^{\kk }-1 ,\quad \forall\, r\in(0,1) .
 \end{align*}
 This thus leads to the estimate below: for all $t\ge1 $,
$$
 h_{t,t+t^\varrho} = \frac{t^\kk}{\kk}((1+t^{\varrho-1})^{\kk}-1) \ge \frac{1}{\kk}(2^\kk-1)t^{\kk+\varrho-1}.
$$
 Hence,
 there exist  constants $\lambda_0,c_5^*>0$
 such that for all $t\ge1,$
\begin{align*}
{\rm d}_{\mathcal F_c}\Big(\bar\mu_{t}\overline {P }_{\!\! h_{t,t+t^\varrho}},\pi \Big)\le \mathcal W_1 \Big(\bar\mu_{t}\overline {P }_{\!\! h_{t,t+t^\varrho}},\pi \Big)\le c_5^*\e^{-\lambda_0 h_{t,t+t^\varrho} }\le c_5^*\e^{-\frac{1}{\kk}\lambda_0 (2^\kk-1)t^{\kk+\varrho-1} },
\end{align*} where in the second inequality we used \cite[Theorem 1.3]{LW16} (or \cite[Corollary 2.3]{Ebe}).

Based on the aforementioned analysis, we
obtain the assertion \eqref{D26}
by noting that $\frac{1}{2}(t+t^\varrho)\le t\le t+t^\varrho$ for all $t\ge1$.
\end{proof}

At the end of this section, inspired by the proof of Theorem \ref{thm5-}, we further elaborate on  potential generalizations of Theorem \ref{thm5-}.
\begin{remark}\label{rema}
Assume that there exist constants $\lambda,C_0>0$ such that for all $t>0$ and $x\in\R^d,$
\begin{align*}
\<x,-x/2+\e^{t/2}F(\e^t,\e^{t/2}x)\> &\le -\lambda|x|^2+C_0  \\
\mbox{ or } \quad \<x,  F(\varphi_t, x/{}g_t)\>/{g_t} &\le -\lambda|x|^2+C_0,
\end{align*}
 where $(g_t)_{t\ge0}$ and $(\varphi_t)_{t\ge0}$ are defined in \eqref{WPPP-1}, and
 that there exist constants $c_1^*,c_2^*,\ell_0>0$ such that for all $x,y\in\R^d,$
\begin{align*}
 	\<x-y,-(x-y)/2+ F_1 ( x )-F_1 (   y ) \> &\le
	c_1^*|x-y|^2\I_{\{|x-y|\le \ell_0\}}-c_2^*|x-y|^2\I_{\{|x-y|> \ell_0\}} \\
	\mbox{ or }\quad \<x-y,  F_2 ( x )-F_2 (   y ) \> &\le
	c_1^*|x-y|^2\I_{\{|x-y|\le \ell_0\}}-c_2^*|x-y|^2\I_{\{|x-y|> \ell_0\}},
\end{align*}
where
$
	F_1(x):=\lim_{t\to\infty}(\e^{ t/2 }F(\e^t,\e^{t/2}x))$ and   $   F_2(x):=\lim_{t\to\infty}(g_t^{1-\alpha}F(\varphi_t, x/{g_t})). $
Additionally, assume that  there exist a function $h_*:[0,\infty)\to[0,\infty)$ satisfying $\lim_{t\to\infty}h_*(t)=0$ and a constant $\alpha_*\in(0,\alpha)$ such that for all $t>0$ and $x,y\in\R^d,$
\begin{align*}
|\e^{t/2}F(\e^t,\e^{t/2}x)-F_1(x)|&\le h_*(t)(1+|x|^{\alpha_*})\\
\mbox{ or } \quad |g_t^{1-\alpha}F(\varphi_t, x/{g_t})-F_2(x)|&\le h_*(t)(1+|x|^{\alpha_*}).
\end{align*}
Furthermore, the corresponding parameterized semigroup satisfies \eqref{D8}. Subsequently, we can establish a counterpart of Theorem   \ref{thm5-} provided that some additional  prerequisites are satisfied.
\end{remark}

\section{Proof of Theorem \ref{thm6}}\label{sec7}
In this part, we continue to  consider the SDE \eqref{eq1},  where we set   the index $\alpha=2$ in $({\bf H}_1)$. Throughout this section, we assume that $ \si=(c_* \omega_d)^{1/2}$  in the SDE \eqref{D11}, and that
the additive process $(Z_t)_{t\ge0}$ involved in the SDE \eqref{D12} was given in \eqref{RR-}, in which   the time-dependent
L\'evy measure $(\nu_t(\d z))_{t\ge0}$ is stipulated as follows: for
$(\varphi_t)_{t\ge0}$ and $(g_t)_{t\ge0}$ satisfying $\varphi_t'g_t^2\log g_t=-1$,
\begin{align}\label{F3-}
\nu_t(\d z):=\varphi_t'g_t^{-d}\left( a( z/{g_t})\I_{\{|z/{g_t}|\le1\}}+c_*| z/{g_t}|^{-(d+2)} \I_{\{| z/{g_t}|>1\}} \right)\d z,\quad \forall\,t\ge0.
\end{align}

In what follows, we derive an analogous version of
 Lemma \ref{lemma5}.

\begin{lemma}\label{lemma7}
Assume that $({\bf A} )$  and \eqref{D8} with $\theta_1=\theta_2=0$ hold. Then, there exists a constant $C_*>0$ such that
for all $0< r\le s< t$, $x\in\R^d$, $f\in C_c^\infty(\R^d)$ and $\vv\in(0,1]$,
\begin{equation}\label{D9}
\big|  (\mathscr L_r  - \overline{\mathscr L}^\vv  )(\overline  P_{\!\! s,t}^\vv  f)(x)\big|  \le C_*\left(\phi(r)V(x)+\varphi_r'g_r^{ 2}+|\bar b^\vv(x)-\bar b (x)|\right)\sum_{i=0}^3 \|\nn^i f\|_\8 .
\end{equation}
\end{lemma}

\begin{proof}
Below, we stipulate  $0< r\le s< t $, $x\in\R^d$, $f\in C_c^\infty(\R^d)$ and $\vv\in(0,1]$.
Note that
\begin{align*}
(\mathscr L_r- \overline{\mathscr L}^\vv )(\overline P_{\!\! s,t}^\vv  f)(x)=&\<\nn (\overline P_{\!\! s,t}^\vv f)(x), b (r,x)-\bar b^\vv(x)\>+\Lambda^\vv (r,s,t,x)\\ &-\frac{1}{2}c_* \omega_d\mbox{trace}\big(\nn^2 (\overline P_{\!\! s,t}^\vv  f)(x)\big),
\end{align*}
where the quantity $\Lambda^\vv(r,s,t,x)$ is defined as in \eqref{F1} with    $\nu_r(\d z)$   therein being   given in \eqref{F3-}.

By applying Taylor's expansion and invoking variable substitution, we derive that
\begin{align*}
\Lambda^\vv(r,s,t,x)
&=\varphi_r'g_r^{-d}\bigg(\int_{\{|z|\le 1\}}\big((\overline P_{\!\! s,t}^\vv  f)(x+ z)-(\bar P_{s,t}^\vv  f)(x)-\<\nn (\overline P_{\!\! s,t}^\vv  f)(x), z\> \big)\, \nu_r^*(\d z)\\
&\quad\quad\quad\quad\quad+ \int_{\{|z|>1\}}\big((\overline P_{\!\! s,t}^\vv  f)(x+ z)-(\overline P_{\!\! s,t}^\vv  f)(x) \big) \,\nu_r^*(\d z)\bigg)\\
&=  \frac{1}{2}\varphi_r'g_r^{ 2} \int_{\{|z|\le 1/g_r\}}\<\nn^2(\overline P_{\!\! s,t}^\vv  f)(x),z\otimes z\>_{\rm HS} \,\nu( \d z)\\
&\quad  +\varphi_r'g_r^{ 3}\int_{\{|z|\le 1/{g_r}\}}\int_0^1\int_0^u\int_0^\theta  \nn_z\nn_z\nn_z (\overline P_{\!\! s,t}^\vv  f)(x+\rho g_rz)
\,\d \rho\,\d \theta\,\d u \, \nu (\d z)\\
&\quad   + c_*\varphi_r'g_r\int_{\{|z|>1/g_r\}}\int_0^1\<\nn (\overline P_{\!\! s,t}^\vv  f)(x+ \theta g_rz), z\> |z|^{-(d+2)} \,\d \theta \,\d z \\
&=:\Lambda_1^\vv(r,s,t,x)+\Lambda_2^\vv(r,s,t,x)+\Lambda_3^\vv(r,s,t,x),
\end{align*}
where
\begin{align*}
\nu_r^*(\d z):= \left( a( z/{g_r})\I_{\{|z/{g_r}|\le1\}}+c_*| z/{g_r}|^{-(d+2)} \I_{\{| z/{g_r}|>1\}} \right)\d z.
\end{align*}
With  the aid of  \eqref{D8} with $\theta_1=0$, it follows from the symmetry of $a(z)$ that
\begin{align*}
\bigg|\int_{\{|z|\le 1\}}\<\nn^2(\overline P_{\!\! s,t}^\vv f)(x),z\otimes z\>_{\rm HS} a(z) \,\d z\bigg|&=\frac{c^\star}{d}\big|\mbox{trace}(\nn^2(\overline P_{\!\! s,t}^\vv f)(x))\big| \\
&\le 2c^\star C_0(\|  f\|_\8+\|\nn f\|_\8+\|\nn^2 f\|_\8),
\end{align*}
where $c^\star:=\int_{\{|z|\le1\}}|z|^2a(z)\,\d z<\infty$, and in the inequality we used the fact that for $A\in\R^d\otimes\R^d,$
\begin{align*}
|\mbox{trace}(A)|=\Big|\sum_{i=1}^d\<e_i,Ae_i\>\Big|\le d\|A\|_{\rm op}
\end{align*}
with $(e_i)_{1\le i\le d}$ being the orthogonal basis in $\R^d$.
Next, due to
$\varphi_r'g_r^2\log g_r=-1$, we derive that
\begin{align*}
\varphi_r'g_r^2\int_{\{1<|z|\le 1/g_r\}}\<\nn^2(\overline P_{\!\! s,t}^\vv  f)(x),z\otimes z\>_{\rm HS} \nu( \d z)
&=c_*\omega_d\mbox{trace}( \nn^2 (\overline P_{\!\! s,t}^\vv  f)(x)).
\end{align*}
Whereafter, we arrive at the following estimate:
\begin{align*}
 \Big|\Lambda_1^\vv(r,s,t,x)-\frac{1}{2}c_*\omega_d\mbox{trace}( \nn^2 (\overline P_{\!\! s,t}^\vv  f)(x))\Big|
 \le c^\star\varphi_r'g_r^{ 2}C_0(\|f\|_\8+\|\nn f\|_\8+\|\nn^2 f\|_\8).
\end{align*}
Next, we derive  from \eqref{D8} with $\theta_2=0$  that
\begin{align*}
\Lambda_{2}^\vv(r,s,t,x)&\le\frac{1}{3} C_0\varphi_r'g_r^{ 3}
\sum_{i=0}^3 \|\nn^i f\|_\8 \int_{\{|z|\le 1/g_r\}} |z|^3\nu(\d z)\\
&\le \frac{1}{3}C_0 \varphi_r'g_r^{ 3}\sum_{i=0}^3 \|\nn^i f\|_\8\left(c^\star+ c_*d\omega_d  (1/{g_r}-1)\right).
\end{align*}
In addition, by means of the first assumption  in \eqref{D8}, we obtain that
\begin{align*}
|\Lambda_3^\vv(r,s,t,x)|&=c_*\varphi_r'g_r\bigg|\int_{\{|z|>1/g_r\}}\int_0^1\<\nn (\bar P_{s,t}^\vv  f)(x+ \theta g_rz), z\>\frac{1}{|z|^{d+2}}\d \theta \,\d z\bigg|\\
&\le  c_*d\omega_d \varphi_r'g_r^2C_0\|\nn f\|_\8.
\end{align*}
Finally, based on the estimates of  $\Lambda_1^\vv ,\Lambda_2^\vv ,\Lambda_3^\vv $, as well as the fact \eqref{G3}, we conclude that the  assertion \eqref{D9} follows.
\end{proof}

With the help  of Lemma \ref{lemma7}, along with $\varphi_t'g_t^2\log g_t=-1$ for all $t\ge0$,
we have the following statement, which is parallel to Proposition \ref{pro}.

\begin{proposition}\label{pro-}
Assume that $({\bf A})$ and  Assumptions {\rm(i)-(\rm iv)} in Theorem $\ref{Th4.1}$  hold, where \eqref{WW-38} therein is replaced by \eqref{D8} with $\theta_1=\theta_2=0$.
Then, there exists a  constant  $C_{**} >0$ such that for any $t>0$, $s\ge0$ and $f\in C_c^\infty(\R^d)$,
$$
\big|(\mu_s \bar {P_t}  )( f)-\mu_{s+t} (f)\big|
 \le  C_{**}  \sum_{i=0}^3 \|\nn^if\|_\8   \int_{s }^{s+t} \Big(\phi(u) -  \frac{1}{\ln g_u}  \Big) \,\d u,
$$
where  $\phi(\cdot)$ is introduced in $({\bf A} )$.
\end{proposition}

Now, we are in the position to present the
\begin{proof}[Proof of Theorem $\ref{thm6}$]
 For the present setting (i.e., the stable index $\alpha=2$),  examinations of
assumptions  \eqref{WW-37}, (\rm iii) and (\rm iv) in Theorem \ref{Th4.1} can be carried out analogously to those
 in the proof of Theorem   \ref{thm5-},  so we herein omit the corresponding  details.  Based on this, in the subsequent proof we mainly focus on the parts that differ from the proof of Theorem   \ref{thm5-}.

For the  case  $\beta\ge (1+\gamma)/2$,
by  taking
\begin{align}\label{G1}
g_t=\e^{-(t+1)/2}\quad\mbox{ and } \quad \varphi_t=1+2\int_1^{1+t} \frac{\e^u}{u} \,\d u,\quad \forall\, t\ge0,
\end{align}
the associated  process $(Y_t)_{t\ge0}:=(g_tX_{\varphi_t})_{t\ge0}$ solves the SDE: for all $t>0,$
\begin{align}\label{G13}
\d Y_t=
\Big(-  \frac{1}{2}Y_t+\frac{2\lambda\e^{ \frac{1}{2}(1+\gamma )(1+t)}Y_t(g_t^2+|Y_t|^2)^{(\gamma-1)/2}}{(1+t)(1+\varphi_t)^{\beta}}\Big)\,\d t+\d Z_t .
\end{align}
As for the case $1=\beta<(1+\gamma)/2,$  we choose
\begin{align}\label{P21}
g_t=(\e+t)^{-\gamma_0}\quad \mbox{ and } \quad \varphi_t=1+\frac{1}{
\gamma_0}\int_\e^{\e+t}\frac{ u^{2\gamma_0}}{ \ln u}\,\d u,\quad \forall\, t\ge 0,
\end{align}
where $\gamma_0:= 1/(\gamma-1)$.
Whence, the corresponding $(Y_t)_{t\ge0}:=(g_tX_{\varphi_t})_{t\ge0}$ satisfies the following SDE: for all $t>0, $
\begin{align}\label{P26}
\d Y_t=\Big(-\frac{\gamma_0 Y_t}{\e+t} +\frac{\lambda(\e+t)^{\gamma_0(1+\gamma)}Y_t(g_t^2+|Y_t|^2)^{(\gamma-1)/2}}{
\gamma_0(1+\varphi_t)^{\beta}\ln (\e+t)}  \Big)\,\d t+\d Z_t.
\end{align}

First of all, for   $(\nu_t(\d z))_{t\ge0}$ defined  in \eqref{F3-},
we  claim that for all $p\in(0,2)$,
 \begin{align}\label{P22}
\sup_{t\ge0
}\bigg(\int_{\{|z|\le1\}}|z|^2\,\nu_t(\d z)\bigg)+\sup_{t\ge0}\bigg(\int_{\{|z|>1\}}|z|^p\,\nu_t(\d z)\bigg)<\infty.
 \end{align}
 In what follows, we assume that functions $(g_t)_{t\ge0}$ and $(\varphi_t)_{t\ge0}$ are given by either \eqref{G1} or \eqref{P21}, and set     $(\mu_t)_{t\ge0}$ to be the law of $(Y_t)_{t\ge0}$ determined by the SDE   \eqref{G13} or the SDE \eqref{P26}, where the associated  initial distribution is
 $\mathscr{L}_{Y_0}=\mu_0=\mu\in\mathscr P(\R^d)$.
It is ready to see that  $g_t^2\varphi_t'\ln g_t=-1$. Thus, due to $-1/{\ln g_t}\le  (1/{\gamma_0})\vee 2 $ for all $t\ge0$,
 we find that for all $t\ge0$,
\begin{align*}
\int_{\{|z|\le1\}}|z|^2\nu_t(\d z) =g_t^2\varphi_t'\int_{\{|z|\le 1\}}|z|^2 a(z) \,\d z+c_*d\omega_d \le ((1/{\gamma_0})\vee 2)\int_{\{|z|\le 1\}}|z|^2 a(z) \,\d z+c_*d\omega_d,
\end{align*}
and that  for any $p\in(0,2)$ and $t\ge0$,
\begin{align*}
\int_{\{|z|>1\}}|z|^p\nu_t(\d z) =\frac{c_*d\omega_dg_t^2\varphi_t' }{ 2-p  }\le \frac{((1/{\gamma_0})\vee 2)c_*d\omega_d  }{  2-p  }.
\end{align*}
Accordingly,  \eqref{P22} follows immediately.

\ \

For case (I): $\beta>(1+\gamma)/2$,
 we take
\begin{align*}
b(t,x)=-  x/2  +\frac{2\lambda \e^{ \frac{1}{2}(1+\gamma )(1+t)}x(g_t^2+|x|^2)^{(\gamma-1)/2}}{(1+t)(1+\varphi_t)^{\beta}}\quad \mbox{ and } \quad  \bar b(x)=-   x/2,\quad \forall\,t\ge0,\,x\in\R^d,
\end{align*}
so that
for all $t\ge0$ and $x\in\R^d,$
\begin{align*}
|b (t,x)-\bar b(x)| \le  \phi_1(t)(1+|x|^{\gamma})\quad \mbox{ with } \quad \phi_1(t):=\frac{2|\lambda| \e^{ \frac{1}{2}(1+\gamma )(1+t)} }{(1+t)(1+\varphi_t)^{\beta}}
\end{align*}
and
\begin{align}\label{KL-6}
\<x,b(t,x)\>\le  - \frac{1}{2}  |x|^2+\phi_1(t)(1+|x|^2)^{(\gamma+1)/2}\I_{\{\lambda>0\}}.
\end{align}
A direct calculation shows that $\int_0^t\frac{\e^{1+s}}{1+s}\,\d s\sim\frac{\e^{1+t}}{1+t}$ as $t\to\infty$ (here and in what follows $f(t)\sim g(t)$ means that $\lim_{t\to\infty} f(t)/g(t)=1$) so  there exist some constants $c_1 >0$ and $c_2>1$ such that for all $t\ge0,$
\begin{align}\label{KL-8}
 c_1\e^{1+t}/{(1+t)}\le \varphi_t\le  c_2\e^{1+t}/{(1+t)},
\end{align}
and
\begin{align}\label{KL-7}
c_1\e^{-(\beta-(1+\gamma)/2)(1+t)}(1+t)^{\beta-1}\le  \phi_1  (t)\le c_2\e^{-(\beta-(1+\gamma)/2)(1+t)}(1+t)^{\beta-1}.
\end{align}
Thus, \eqref{KL-7}, along with \eqref{P22}, \eqref{KL-6} and  Proposition \ref{lem3}, enables us to derive that, for any $p\in(0,2)$ and $\beta>(1+\gamma)/2$,
there is a constant $c_3:=c_3(p)$
such that
\begin{align}\label{G5}
\sup_{t\ge0}\E|Y_t|^p\le c_3(1+\E|Y_0|^p).
\end{align}
Subsequently, we deduce from Proposition \ref{pro-} that there exists a constant $c_4>0$
 (which is  dependent linearly  on $\mu(|\cdot|^{\gamma})$)
 such that for any $t,s>0$ and $f\in C_c^\infty(\R^d)$,
 $$
 \big|\mu_{s+t} (f)-(\mu_s \overline {P}_{\!\! t}  )( f)\big|
 \le  c_4   \sum_{i=1}^3 \|\nn^if\|_\8   \int_{s }^{s+t} \Big( \frac{(1+u)^{\beta-1} }{ \e^{  (\beta-(1+\gamma)/2) u }}+\frac{2}{u+1}\Big)\,\d u ,
 $$
where $(\overline {P}_{\!\! t})_{t\ge0} $ is the semigroup of $(\overline{Y}_{\!\! t})_{t\ge0}$ solved by the first SDE in \eqref{e:limit} with $\eta=2.$ The previous estimate, along with the triangle inequality,  obviously implies that for all $t\ge\e$,
\begin{equation}\label{D13}
\begin{split}
	{\rm d}_{\mathcal F_c} (\bar\mu_{h_t},\pi )&\le {\rm d}_{\mathcal F_c}\Big(\bar\mu_{h_t},\bar\mu_{t}\overline {P }_{\!\! \varphi_{h_t}^{-1}-\varphi_{t}^{-1}} \Big)+{\rm d}_{\mathcal F_c}\Big( \bar\mu_{t}\overline {P }_{\!\! \varphi_{  h_t}^{-1}-\varphi_{t}^{-1}} ,\pi\Big)\\
&\le c_4    \int_{\varphi_{t}^{-1}}^{ \varphi_{h_t}^{-1}} \Big( \frac{(1+u)^{\beta-1} }{ \e^{  (\beta-(1+\gamma)/2) u }}+\frac{2}{u+1}\Big)\,\d u +{\rm d}_{\mathcal F_c}\Big( \bar\mu_{t}\overline {P }_{\!\! \varphi_{h_t}^{-1}-\varphi_{t}^{-1}} ,\pi\Big),
\end{split}
\end{equation}
where $\bar\mu_t$ is the law of $g_{\varphi_t^{-1}}X_t$, and $h_t:=t(1+t^{\phi(t)}) $ for $\phi(t):=\frac{3\ln\ln t}{\ln t}$.
In particular, we have that  for all $t\ge\e,$
\begin{align}\label{D24}
\ln t\le \ln h_t\le \ln 2+(1+\phi(t))\ln t \le(2+\phi(t))\ln t\le 5\ln t
\end{align}
by taking advantage of $\phi(t)\in[0,3]$.

From \eqref{KL-8}, there exists a constant $t_1\ge\e$ such that for all $t\ge t_1,$
\begin{align}\label{D0}
\ln t\le \varphi_t^{-1}\le \ln t+2\ln\ln t.
\end{align}
This yields that for all $t\ge t_1,$
\begin{align}\label{D23}
\varphi_{h_t}^{-1}-\varphi_{t}^{-1}&\ge \ln (t(1+t^{\phi(t)} ))-(\ln t+2\ln\ln t) \ge \phi(t)\ln t-2\ln\ln t=\ln\ln t
\end{align}
by making use of the definition of $\phi(t)$,  and that for some constant $c_5>0$ and all $t\ge t_1,$
\begin{align}\label{D10}
\varphi_{h_t}^{-1}-\varphi_{t}^{-1}\le c_5(1+\ln t+ \ln\ln t).
\end{align}
By means of \eqref{D0}, there exist constants $c_6 >0$ and $t_2\ge t_1$ such that for all $t\ge t_2,$
\begin{equation}\label{D15}
\begin{split}
\int_{\varphi_{t}^{-1}}^{ \varphi_{h_t }^{-1}}\frac{1}{u+1}\,\d u\le\ln \varphi_{ h_t }^{-1} -\ln\varphi_{t}^{-1}
&\le\ln\left(\frac{ \ln h_t +2\ln\ln h_t}{\ln t}\right)\\
&\le\ln\left(1+\phi(t)+\frac{  c_6\ln\ln t}{\ln t}\right) \le  \frac{(c_6+3)\ln\ln t}{\ln t},
\end{split}
\end{equation}
where in the second inequality we employed \eqref{D24}.
Due to \eqref{D0} and \eqref{D10},
we find that for some constant $c_7>0$ and all $t\ge t_2,$
\begin{equation}\label{D20}
\begin{split}
&\int_{\varphi_{t}^{-1}}^{ \varphi_{h_t}^{-1}} \big( (1+u)^{\beta-1}  \e^{- (\beta-(1+\gamma)/2) u }\big)\,\d u\\
&\le \big(\I_{\{\beta\le1\}} +(1+\varphi_{ h_t }^{-1})^{\beta-1}\I_{\{\beta>1\}}\big)\big(\varphi_{ h_t}^{-1}-\varphi_{t}^{-1}\big)\e^{- (\beta-(1+\gamma)/2)\varphi_{t}^{-1}}\\
&\le c_7(\ln t)^{1\vee \beta}t^{- (\beta-(1+\gamma)/2) }.
\end{split}
\end{equation}
Furthermore, note from \eqref{D23} that there exists a  constant  $c_8 >0$
 (which is linearly dependent on $\mu(|\cdot|)$ and $\pi(|\cdot|)$)
 such that for all $t\ge t_1,$
\begin{align}\label{D22}
{\rm d}_{\mathcal F_c}\Big( \bar\mu_{t}\overline {P }_{\!\! \varphi_{h_t}^{-1}-\varphi_{t}^{-1}} ,\pi\Big)
\le \mathcal W_1\Big( \bar\mu_{t}\overline {P }_{\!\! \varphi_{h_t}^{-1}-\varphi_{t}^{-1}} ,\pi\Big)
\le c_8 \e^{-\frac{1}{2}(\varphi_{h_t}^{-1}-\varphi_{t}^{-1} )}\le c_8(\ln t)^{- 1/2}.
\end{align}
Now, plugging  \eqref{D15}, \eqref{D20} as well as \eqref{D22} back into \eqref{D13} yields that for all $t\ge t_2$ and some
constant $c_9>0$,
\begin{equation}\label{D21}
\begin{split}
	{\rm d}_{\mathcal F_c} (\bar\mu_{h_t},\pi )
 \le c_9   \bigg(\frac{\ln\ln t}{\ln t} +(\ln t)^{1\vee \beta}t^{- (\beta-(1+\gamma)/2) }+(\ln t)^{-
1/2}\bigg).
\end{split}
\end{equation}
This immediately yields the assertion \eqref{G4}.

\ \

For case (II): $\beta=(1+\gamma)/2$ with $\gamma=1$, we choose
\begin{align*}
	b(t,x)=-  x/2  +\lambda\phi_2(t)x\quad \mbox{ and } \quad  \bar b(x)=(-   1/2+\lambda)x,\quad \forall\,t\ge0,\,x\in\R^d,
\end{align*}
where $\phi_2(t):=\frac{2  \e^{  (1+t)} }{(1+t)(1+\varphi_t) }$.
It follows that for all $t\ge0$ and $x\in\R^d,$
\begin{align*}
\<x,b(t,x)\>\le -(1/2-\lambda)|x|^2+ |\lambda|\cdot  |\phi_2(t)-1|\cdot|x|^2.
\end{align*}
This, together with \eqref{P22} and  $\phi_2(t)\to1$ as $t\to\infty$, implies that
\eqref{G5} still holds true.
Applying twice the integration by parts formula yields that for any $t\ge 0,$
\begin{align*}
\varphi_t=1-4\e+\frac{2\e^{1+t}}{1+t} +\frac{2\e^{1+t}}{(1+t)^2}+4\int_0^t\frac{\e^{1+s}}{(1+s)^3}\,\d s.
\end{align*}
This leads to the following identity: for any $t\ge0, $
 \begin{align*}
\phi_2(t)  -1  =-\frac{1/{ (1+t) }+\Lambda_t}{ 1 + 1/{ (1+t) }+\Lambda_t}
\end{align*}
with
 \begin{align*}
 \Lambda_t:=\frac{1+t}{ \e^{1+t}}\bigg(1-2\e+2\int_0^t\frac{\e^{1+s}}{(1+s)^3}\,\d s \bigg).
 \end{align*}
Next, there exist constants $c_0,t_0>0$ such that
$
 0\le \Lambda_t\le c_0/{(1+t)^2}
$
for all $t\ge t_0 $. Accordingly, we have that for all $t\ge t_0,$
\begin{align*}
\int_{\varphi_t^{-1}}^{\varphi_{h_t}^{-1}}|\phi_2(u) -1|\,\d u
 \le (1+c_0)\big(\ln \varphi_{ h_t }^{-1} -\ln\varphi_{t}^{-1}\big),
\end{align*}
 where $h_t$ was defined as in \eqref{D13}.  Moreover, \eqref{D22} still holds true with   $
 1/2-\lambda$
 in place of $1/2$. Correspondingly, regarding case (II), the assertion \eqref{G4} is available by applying  Proposition \ref{pro-} once more,
 taking \eqref{D15} into account, and noting that for all $t\ge0$ and $x\in\R^d,$
 \begin{align*}
|b(t,x)-\bar b(x)|\le |\lambda|\cdot  |\phi_2(t)-1|\cdot|x|.
 \end{align*}

\ \

In the following analysis, we consider   case (III): $1=\beta<(1+\gamma)/2$ to verify \eqref{G4}. In this case, the corresponding $(Y_t)_{t\ge0}$ is determined by the SDE \eqref{P26}. Accordingly, we can take
\begin{align*}
\bar b(x)=\lambda
 (1+2  \gamma_0 )
 x|x|^{\gamma-1}\quad \, \mbox{ and }\, \quad b(t,x)=- \frac{\lambda_0 x}{\e+t} +\lambda\phi_3(t)x(g_t^2+|x|^2)^{(\gamma-1)/2},
\end{align*}
where
$  \phi_3(t):=\frac{(\e+t)^{\gamma_0(1+\gamma)}}{
	\gamma_0(1+\varphi_t) \ln (\e+t)}$
with  $(\varphi_t)_{t\ge0}$ being  defined in \eqref{P21}.
By invoking  $\gamma\in(1,2)$,  we find that
  for all $t\ge0$ and $x\in\R^d,$
\begin{align*}
\<x,b (t,x)\>\le - \gamma_0|x|^2 /{(\e+t)}  +\lambda(1+2\gamma_0)|x|^{\gamma+1}+|\lambda|\cdot|\phi_3(t)-(1+2\gamma_0)|(1+|x|^2)^{\gamma/2},
\end{align*}
 where $\phi_3(t)\to1+2\gamma_0$
 as $t\rightarrow\infty $ by noting from  $\gamma_0(1+\gamma)=2\gamma_0+1$ that
\begin{align} \label{WPPP}
	\int_\e^{\e+t}\frac{s^{2\gamma_0}}{\ln s}\,\d s\sim\frac{ (\e+t)^{2\gamma_0+1}}{ (2\gamma_0+1)\ln(\e+t) }\quad \mbox{ as } t\to\infty.
\end{align}
 Subsequently, \eqref{P22} and  $\phi_3(t)\to1+2\gamma_0$ as $t\to\infty$ enable us to deduce that
\eqref{G5} is still valid for $(Y_t)_{t\ge0}$ solving the SDE \eqref{P26}. Furthermore,
notice from $\gamma\in(1,2)$ that for all $t\ge0$ and $x\in\R^d,$
\begin{align*}
	|b(t,x)-\bar b(x)| \le \frac{\lambda_0 |x|}{\e+t}+|\lambda|\cdot|\phi_3(t)-(1+2\gamma_0)|(1+|x|^2)^{\frac{\gamma}{2}}+|\lambda|(1+2\gamma_0)g_t|x|.
\end{align*}
Whereafter,  by applying Proposition \ref{pro-}, there exists a constant $c_1>0$
 (which  depends linearly   on $\mu(|\cdot|)$ and $\mu(|\cdot|^{\gamma})$)
 such that for all $t\ge0,$
 \begin{equation*}
\begin{split}
	{\rm d}_{\mathcal F_c} (\bar\mu_{h_t},\pi )&\le {\rm d}_{\mathcal F_c}\Big(\bar\mu_{h_t},\bar\mu_{t}\overline {P }_{\!\! \varphi_{h_t}^{-1}-\varphi_{t}^{-1}} \Big)+{\rm d}_{\mathcal F_c}\Big( \bar\mu_{t}\overline {P }_{\!\! \varphi_{  h_t}^{-1}-\varphi_{t}^{-1}} ,\pi\Big)\\
&\le c_1 \int_{\varphi_{t}^{-1}}^{ \varphi_{h_t}^{-1}}\Big( |\phi_3(u)- (1+2\lambda_0)
 |+\frac{1}{\ln(\e+u)}\Big)\,\d u +{\rm d}_{\mathcal F_c}\Big( \bar\mu_{t}\overline {P }_{\!\! \varphi_{h_t}^{-1}-\varphi_{t}^{-1}} ,\pi\Big),
\end{split}
\end{equation*}
where $h_t:=t(1+  \phi (t))$ with $\phi (t):= (\ln t)^{ {1}/{2}}/{\varphi_t^{-1}}$.

By applying twice  the integration by parts formula, we have that for  all $t\ge0,$
 \begin{align*}
 	\varphi_t&=1+\frac{
 (\e+t)^{2\gamma_0+1}}{
 \gamma_0(1+2\gamma_0)\ln(\e+t)}+\frac{ (\e+t)^{2\gamma_0+1}}{
 \gamma_0(1+2\gamma_0)^2(\ln(\e+t))^2}\\
 &\quad+\frac{2}{
 \lambda_0(1+2\gamma_0)^2}\int_\e^{\e+t}\frac{u^{2\gamma_0}}{(\ln u)^3}\,\d u-\frac{\e^{2\gamma_0+1}}{
 \gamma_0(1+2\gamma_0)}-\frac{\e^{2\gamma_0+1}}{
 \gamma_0(1+2\gamma_0)^2}.
 \end{align*}
 The identity above, together with $\lambda_0(1+\gamma)=2\lambda_0+1$, yields that for all $t\ge0,$
 \begin{align*}
 	\phi_3(t)-(1+2\gamma_0)&=
 	(1+2\gamma_0)\Big(\frac{\phi_3(t)}{1+2\gamma_0}-1\Big)\\
 	& =-(1+2\gamma_0)\times\frac{1/{((2\gamma_0+1) \ln(\e+t)) }+\Lambda_t^*}{1+1/{((2\gamma_0+1) \ln(\e+t)) }+\Lambda_t^*},
 \end{align*}
 where
 \begin{align*}
 	\Lambda_t^*:=\frac{
 \gamma_0(1+2\gamma_0)\ln(\e+t)}{(\e+t)^{2\gamma_0+1}}\bigg(\frac{2}{
 \gamma_0(1+2\gamma_0)^2}\int_\e^{\e+t}\frac{u^{2\gamma_0}}{(\ln u)^3}\,\d u
 	+2-\frac{\e^{2\gamma_0+1}}{
\gamma_0(1+2\gamma_0)}-\frac{\e^{2\gamma_0+1}}{
 \gamma_0(1+2\gamma_0)^2}\bigg).
 \end{align*}
As a consequence, there exists a  constant  $c_2>0$ such that for all  $t\ge0,$
\begin{align*}
 |\phi_3(t)-
 (1+2\gamma_0)
 |\le \frac{c_2}{\ln(\e+t)}.
\end{align*}
The previous estimates  further give   that there is a constant $c_3>0$ such that for all $t\ge0$,
 \begin{equation*}
\begin{split}
	{\rm d}_{\mathcal F_c} (\bar\mu_{h_t},\pi )
&\le c_3 \int_{\varphi_{t}^{-1}}^{ \varphi_{h_t}^{-1}} \frac{1}{\ln(\e+u)} \,\d u +{\rm d}_{\mathcal F_c}\Big( \bar\mu_{t}\overline {P }_{\!\! \varphi_{h_t}^{-1}-\varphi_{t}^{-1}} ,\pi\Big)\\
&\le \frac{c_3(\varphi_{h_t}^{-1}-\varphi_{t}^{-1} )}{\ln(\e+\varphi_{t}^{-1})} +{\rm d}_{\mathcal F_c}\Big( \bar\mu_{t}\overline {P }_{\!\! \varphi_{h_t}^{-1}-\varphi_{t}^{-1}} ,\pi\Big).
\end{split}
\end{equation*}
Next, by means of \eqref{WPPP},  there exist constants $c_4,c_5,t_0>0$ such that for all $t\ge t_0,$
\begin{align*}
\frac{c_4\ln\varphi_t^{-1}}{(\varphi_t^{-1})^{2\gamma_0}}\le (\varphi_t^{-1})'\le\frac{c_5\ln\varphi_t^{-1}}{(\varphi_t^{-1})^{2\gamma_0}}\quad \mbox{ and } \quad  c_4 t\ln t \le (\varphi_t^{-1})^{1+2\gamma_0}\le c_5 t\ln t.
\end{align*}
Consequently,   by invoking $h_t-t=t \phi (t) $ and $\phi (t) = (\ln t)^{ {1}/{2}}/{\varphi_t^{-1}}$,
there are constants $c_6,c_7,t_1>0$ such that for all $t\ge t_1,$
 \begin{align*}
 \varphi_{h_t}^{-1}-\varphi_{t}^{-1}\ge \frac{c_6\ln \varphi_t^{-1} }{( \varphi_t^{-1})^{2\gamma_0}}\times t\phi(t) \ge \frac{c_6  \phi(t)\varphi_t^{-1}\ln \varphi_t^{-1} }{c_5 \ln t }\ge \frac{ c_6\phi(t)\varphi_t^{-1}}{c_5(1+2\gamma_0) }=\frac{ c_6(\ln t)^{\frac{1}{2}}}{c_5(1+2\gamma_0) },
 \end{align*}
and
\begin{align*}
 \frac{ \varphi_{h_t}^{-1}-\varphi_{t}^{-1}  }{\ln(\e+\varphi_{t}^{-1})}\le \frac{c_7\ln\varphi_t^{-1}}{(\varphi_t^{-1})^{2\gamma_0}}\times t\phi(t)\times\frac{1}{\ln\varphi_t^{-1}}\le  \frac{c_7\phi(t)\varphi_t^{-1}}{c_4 \ln t}=\frac{c_7 }{c_4 (\ln t)^{\frac{1}{2}}}.
 \end{align*}
Furthermore, for $\gamma\in(1,2)$,
 we have that
  for some constants $c_8,c_9,c_{10}>0$ (which are linearly dependent on $\mu(|\cdot|)$ and $\pi(|\cdot|)$) and $\lambda^\star>0$,
 \begin{align*}
{\rm d}_{\mathcal F_c}\Big( \bar\mu_{t}\overline {P }_{\!\! \varphi_{h_t}^{-1}-\varphi_{t}^{-1}} ,\pi\Big)\le c_8\e^{-\lambda^\star(\varphi_{h_t}^{-1}-\varphi_{t}^{-1})}\le c_9\e^{-c_{10}(\ln t)^{ {1}/{2}}}\le\frac{c_9}{c_{10}\e(\ln t)^{ {1}/{2}}},
\end{align*}
 where
 the first inequality follows from \cite[Theorem 1.3]{LW16} (or \cite[Corollary 2.3]{Ebe}), and in the last inequality we used the fact that $\max_{r\ge0}(r\e^{-cr})\le1/(c\e)$ for given $c>0.$
At length, concerning case (III), the assertion \eqref{G4}
 is verifiable based on the preceding analysis.
\end{proof}

To conclude this paper, we make a  comment  regarding Theorem \ref{thm6}.

\begin{remark}\label{R:7.3}
Under the corresponding assumptions imposed in Remark \ref{rema},  Theorem \ref{thm6} can also be generalized  to the setting in which $F$ under consideration need not satisfy the requirement given in $({\bf H}_2)$. Since the  spirit is similar,  we do not provide further details  in the present work.

\end{remark}

\ \

\noindent {\bf Acknowledgements.}\,\,
The research of Jianhai Bao is supported by the National Key R\&D Program of China (2022YFA1006004) and the National Natural Science Foundation  of China (No. 12531007).
The research of Jian Wang is supported by the National Key R\&D Program of China (2022YFA1006003) and the National Natural Science Foundation  of China (Nos. 12225104 and 12531007).

\end{document}